\documentclass[11pt,amstex]{article}

\usepackage{todonotes}
\usepackage{color,xcolor}             
\usepackage{graphicx}
\usepackage[]{amsmath,amssymb}
\usepackage{mathtools}
\usepackage{hyperref}
\usepackage{enumerate, amsfonts, amsthm, bbm,bm}
\usepackage{tikz}
\usetikzlibrary{arrows,automata,positioning}

\definecolor{MyDarkBlue}{rgb}{0,0.08,0.50}
\definecolor{BrickRed}{rgb}{0.65,0.08,0}

\hypersetup{
colorlinks=true,       % false: boxed links; true: colored links
    linkcolor=MyDarkBlue,          % color of internal links
    citecolor=BrickRed,        % color of links to bibliography
    filecolor=red,      % color of file links
    urlcolor=cyan           % color of external links
}

\theoremstyle{plain}

\newtheorem{LEM}{Lemma}
\newtheorem{PRP}[LEM]{Proposition}
\newtheorem{THM}[LEM]{Theorem}

\newtheorem{COR}[LEM]{Corollary}

\newcommand{\innerthmname}{}
\theoremstyle{definition}

\theoremstyle{definition}
\newtheorem{EXA}[LEM]{Example}
\newtheorem{REM}[LEM]{Remark}

\newtheorem{DEF}[LEM]{Definition}
\newtheorem{COND}[LEM]{Condition}

\newcommand{\xqed}[1]{%
   \leavevmode\unskip\penalty9999 \hbox{}\nobreak\hfill
   \quad\hbox{\ensuremath{#1}}}
\newcommand{\Enddef}{\xqed{\blacktriangleleft}}

\newcommand{\nc}[1]{}

\newcommand{\1}{\mathbbm{1}}
\newcommand{\indic}[1]{\1_{\{#1\}}}
\newcommand{\nn}{\nonumber}

\newcommand{\degK}{\emptyset_K}

\numberwithin{equation}{section}

\newcommand{\ra}{\rightarrow}
\renewcommand{\P}{\mathbb{P}}
\newcommand{\E}{\mathbb{E}}
\newcommand{\N}{\mathbb{N}}
\newcommand{\R}{\mathbb{R}}
\newcommand{\Z}{\mathbb{Z}}
\newcommand{\C}{\mathbb{C}}

\newcommand{\TT}{\mathbb{T}}

\newcommand{\floor}[1]{\lfloor #1 \rfloor}

\newcommand{\mc}[1]{\mathcal{#1}}

\newcommand{\blank}[1]{}
\newcommand{\sss}{\scriptscriptstyle}

\newcommand{\vep}{\varepsilon}
\newcommand{\cweak}{\overset{w}{\longrightarrow}}

\newcommand{\cprobs}{\overset{\P^{s,K}}{\longrightarrow}}
\newcommand{\cfdd}{\overset{f.d.d.}{\longrightarrow}}

\newcommand{\D}{\mc{D}}
\newcommand{\oD}{\overline{\mc{D}}}
\newcommand{\oC}{\overline{\mc{C}}}

\newcommand{\M}{\mc{M}}
\newcommand{\K}{\mc{K}}
\newcommand{\T}{\mc{T}}
\newcommand{\F}{\mc{F}}

\renewcommand{\d}{\mathrm d}

\newcommand{\ara}{\overset{\bs{a}}{\ra}}
\newcommand{\aran}{\overset{\bs{a,n}}{\ra}}
\newcommand{\sn}{\sss(n)}

\newcommand{\tempend}{\end{document}}
\newcommand{\bs}[1]{{\boldsymbol #1}}
\newcommand{\aij}{\langle i,j \rangle}
\newcommand{\apij}{\langle i',j' \rangle}

\newcommand{\lb}{\langle}
\newcommand{\rb}{\rangle}
\newcommand{\supp}{\textrm{supp}}

\newcommand{\mf}[1]{\mathfrak{#1}}
\newcommand{\MF}{\mc{M}_F}
\newcommand{\FL}{\mathfrak{L}}
\newcommand{\what}{\vartheta}

\begin{document}

\author{
Manuel Cabezas\footnote{Pontificia Universidad Cat\'olica di Chile. 
E-mail {\tt cabezas.mn@gmail.com}}
,
Alexander Fribergh\footnote{Universit\'e de Montreal.
E-mail {\tt alexander.fribergh@umontreal.ca}}
,
Mark Holmes\footnote{School of Mathematics and Statistics,
The University of Melbourne.
 E-mail {\tt holmes.m@unimelb.edu.au}. Research supported in part by Future Fellowship FT160100166 from the Australian Research Council}
,
and Edwin Perkins\footnote{Department of Mathematics,
The University of British Columbia.
 E-mail {\tt perkins@math.ubc.ca}. Research supported in part by an NSERC Discovery grant.}.
}

\title{Random skeletons in high-dimensional lattice trees}

\maketitle

\begin{abstract}
We study the behaviour of the rescaled minimal subtree containing the origin and $K$ random vertices selected from a random critical (sufficiently spread-out, and in dimensions $d>8$) lattice tree conditioned to survive until time $ns$, in the limit as $n$ goes to infinity.  We prove joint weak convergence of various quantities associated with these  subtrees under this sequence of conditional measures to their counterparts for historical Brownian motion.  We also show that when $K$ is sufficiently large the entire rescaled tree is close to this rescaled skeleton with high probability, uniformly in $n$.  These two results are the key conditions used in~\cite{B-ACF18} to prove that the simple random walk on sufficiently spread-out lattice trees (conditioned to survive for a long time) converges to Brownian motion on a super-Brownian motion (conditioned to survive). 

The main convergence result is established more generally for a sequence of historical processes converging to historical Brownian motion in the sense of finite dimensional distributions and satisfying a pair of technical conditions. The conditions are readily verified for the lattice trees mentioned above and also for critical branching random walk.  We expect that it will also apply with suitable changes to other lattice models in sufficiently high dimensions such as oriented percolation and the voter model. In addition some forms of the second skeleton density result are already established in this generality.  
\end{abstract}

\noindent {\bf Keywords: } Lattice trees, weak convergence, random subtree, historical Brownian motion,  super-Brownian motion.

\noindent {\bf MSC2020:} 82B41, 60F17, 60G57, 60K35.

\section{Introduction and main results}
\label{sec:intro}
Various critical lattice models have been shown to converge to super-Brownian motion, when the dimension is larger than a certain critical dimension for the model.  Examples include the voter model, oriented percolation, the contact process and lattice trees (discussed below)  \cite{CDP00,BCLG01,HS03b,HofSak10,H08,HHP17}.  Each of these models comes with a natural time parameter, and the convergence is phrased in terms of the random process  conditioned on survival until time $ns$ (where $s>0$) in the limit as $n \to \infty$.  The state of each of these processes at time $t$ is a measure on $\R^d$, and after rescaling space and time, the limit is a measure-on-$\R^d$-valued diffusion (conditioned to survive until time $s$) called super-Brownian motion (SBM).  In the case of oriented percolation and the contact process, the statement is weak convergence of the finite-dimensional distributions while for lattice trees and the voter model convergence on path space has been established.  Each of the above lattice models also comes with a notion of genealogy (e.g.~who infected who in the contact process), that is not encoded in the above measure-valued processes.

In recent years much progress has been made in the context of a model for branched polymers called \emph{lattice trees} \cite{DerSla98,H08,HHP17,H16}, in dimensions $d>8$.  Most relevant for the context of this paper is the  
result that historical lattice trees conditional on survival until time $ns$ converge weakly to historical Brownian motion (HBM) \cite{CFHP20} conditional on survival until time $s$.  The state of the historical process at time $t$ is a measure on genealogical paths of duration $t$, and thus the historical objects encode the genealogy of the model.  It is also known in all of the models mentioned above that, above a critical dimension, the conditional ranges (the random set of spatial points visited by the process) converge  weakly to the conditional range of SBM  \cite{HP19}.  As well as living in spatial dimensions $d>8$, the lattice trees considered here and in the above references are sufficiently spread-out in the sense that edges between vertices in the tree may be long (up to some distance $L$ that is fixed and large).  The random tree is selected according to a critical weighting scheme.  This description will be made more precise in the following sections.  

Random walk (RW) in random media has been studied intensely in recent decades, see for example \cite{Bar17,Kuma10}.  Among the kinds of results that have been frequently studied is the setting where one has a sequence of random graphs with a weak limit, often described as a metric measure space, and one is interested in the scaling limit of a random walk on such graphs.   In Ben-Arous et al \cite{B-ACF18} and~\cite{B-ACF19} (see also Croydon \cite{Croydon09}) sufficient conditions (on the underlying sequence of graphs) are given for weak convergence of RW on the graph to 
Brownian motion on some limiting fractal.  
 
We study properties of the minimal subtree in a random lattice tree $\mc{T}$ 
connecting the origin $o$ to $K$ uniformly chosen points 
in $\mc{T}$,  conditional on the tree $\mc{T}$ surviving until time $ns$. Our results show that for $d$ and $L$ as noted above, if time 
is scaled by $1/n$ and the spatial vertices are scaled by $1/\sqrt n$, then, as $n \to \infty$, the rescaled ancestral tree with the rescaled graph metric converges weakly (as a metric space) to the ancestral tree of $K$ uniformly chosen points from a SBM
 cluster (conditional on survival until time $s$). Here the ``ancestral tree" refers only to the $K$ leaves, root and internal branch points in the subtree.  We also show that under the same scaling and conditioning, uniformly in $n$, the original subtree is, with high probability, a good approximation of the whole tree if we take $K$ sufficiently large.  
These results establish Conditions $(G)^{s,+}_{1,\sigma_0,C_0}$ and $(S)$, respectively, in \cite{B-ACF18}, which are needed there to verify  that RW on lattice trees conditional on survival to $ns$ converges to Brownian motion on a SBM cluster conditional on survival, as defined in \cite{Croydon09}.
 
 The convergence of the historical processes to HBM
 in \cite{CFHP20}, described above,  plays a central role in the weak convergence result.   In fact we prove an abstract convergence result where the historical convergence is the main assumption.  To study the limit of the rescaled ancestral tree
 we will also augment the historical convergence to obtain convergence of the finer branching structure.  To illustrate the general methodology we  also use it to obtain a similar result for critical branching random walk (BRW) (in any dimension and for any $L\ge 1$), where the underlying 
Galton-Watson (GW) tree is conditioned to survive until time $ns$ and the $K$ individuals are chosen uniformly at random from the entire GW tree.

 To verify that the subtree 
 is a good approximation of $\mc{T}$ when $K$ is large, we will also use the range convergence results \cite{HP19}, including the so-called uniform modulus of continuity therein.  Some of these results apply to the general setting of \cite{HP19} and we think that our general approach may be applicable to this setting and so include other models such as the voter model, oriented percolation, and the contact process.

For a Polish space $E$, let $\mc{D}_t(E)$ be the  space of cadlag $E$-valued paths on $[0,t]$ with the Skorokhod topology, and $\mc{D}_\infty(E)$ be  the same but with $\R_+$ in place of $[0,t]$. 
When $E=\R^d$ we will write $\mc{D}$ for $\mc{D}_\infty(\R^d)$ and for $n\in\N$ define rescaling maps $\rho_n:\mc{D}\to\mc{D}$ by $\rho_n(w)_s=\frac{w_{ns}}{\sqrt n}$.   We also write $\rho_n(w)=(\rho_n(w_1),\dots,\rho_n(w_K))$ for $w=(w_1,\dots,w_K)\in\mc{D}^K$. 
Let $\mc{M}_F(E)$ denote the set of finite measures on the Borel sets of $E$ equipped with the topology of weak convergence.

\subsection{Branching random walk}
\label{sec:BRW}
Let $Y$ be a $\Z_+$-valued random variable with mean 1 and variance $\gamma\in (0,\infty)$, such that 
\begin{equation}\label{Ymom}\E[Y^p]<\infty \text{ for each }p>0.
\end{equation}
 If $n\in\N$, we set $[n]=\{1,\dots,n\}$. Write $I=\{0\}\cup \{0\alpha_1\cdots \alpha_n:n \in \N, \alpha_j \in \N \ \forall j \in [n]\}$ for the countable index set of potential individuals of a 
 GW tree.  Let $I_m=I\cap(\{0\}\times\N^m)\ (I_0=\{0\})$ denote the set of such potential individuals in generation $m\in\Z_+$, set $I_t=I_{\lfloor t\rfloor}$ for $t\ge0$, and let $|\alpha|=m$ iff $\alpha\in I_m$.  Let $(Y_\alpha)_{\alpha \in I}$ be independent random variables with the same distribution as $Y$.  We define a random tree $\mc{T}$ recursively as follows:  $0$ is a vertex in $\mc{T}$, and  
\[0\alpha_1\cdots\alpha_m i\in \mc{T}, \quad \text{ if and only if }
\alpha=0\alpha_1\cdots\alpha_m\in \mc{T} \quad \text{ and }i\in [Y_\alpha].\]
We then construct the GW tree by inserting an edge between each such vertex $0\alpha_1\cdots\alpha_m i\in \mc{T}$ and its \emph{parent} $0\alpha_1\cdots\alpha_m$.  The individuals in $\mc{T}$ of generation $m$, i.e.  those that are tree distance $m$ from the root 0 in this random tree, are denoted $\mc{T}_m$.  For $t\ge 0$ define $\mc{T}_t=\mc{T}_{\floor{t}}$. The ancestors of $\alpha=0\alpha_1\cdots\alpha_m$ are the individuals/vertices $0\alpha_1\cdots\alpha_j$, $j<m$ along the unique path in this tree from $\alpha$ to the root $0$. 

Fix $d\ge 1$, $L\in\N$ and let $D(\cdot)$ be the uniform distribution on a finite box $\Z^d\cap[-L,L]^d\setminus o$, where $o=(0,\dots,0)\in \Z^d$.  Let $(\Delta_{\alpha})_{\alpha \in I}$ be independent random variables (also independent of $(Y_j)_{j \in I}$) with distribution given by $D$ under a probability measure $\P$. We set $\varphi(0)=o\in \Z^d$ and associate to each $\alpha=0\alpha_1\cdots\alpha_m\in I\setminus \{0\}$ the spatial location  
\begin{equation}\label{Deltasum}\varphi(\alpha)=\sum_{j=1}^m \Delta_{0\alpha_1\cdots \alpha_j}.
\end{equation}
In other words, we sum up i.i.d.~increments with law $D$ associated to $\alpha$ and each of its ancestors (excluding the root).   Associated to each $\alpha=0\alpha_1\cdots\alpha_m\in I_m$ is the path $w(m,\alpha)$ defined for $\ell\in \{0,1,\dots, m\}$ by 
\begin{equation}
\label{wpath_BRW}
w_\ell(m,\alpha)=\varphi(0\alpha_1\cdots\alpha_\ell).
\end{equation}
For  $u\ge 0$, 
set $w_u(m,\alpha)=w_{\floor{u}\wedge m}(m,\alpha)$.  Note that the second argument of $w$ is the label of the individual and so multiple occupancy is allowed unlike the lattice trees discussed below.

For $t\ge 0$, $n\in\N$ and $\alpha\in I_{nt}$ define the rescaled history of $\alpha$, $w^{\sss(n)}_\cdot(t,\alpha)=w_{\cdot\wedge t}^{\sn}(t,\alpha)\in\mc{D}$, by 
\begin{align}
w^{\sss(n)}_s(t,\alpha)=\frac{w_{ns}(\floor{nt},\alpha)}{\sqrt{n}}=\rho_n(w(\floor{nt},\alpha))_s,
 \quad \text{ for }s\ge 0.\label{rescaledpathbrw}
\end{align}
In particular, if $\alpha\in \mc{T}_{nt}$, then $w^{\sn}_t(t,\alpha)=\varphi(\alpha)/\sqrt{n}$ is the rescaled spatial location of this particle in the $(nt)^{th}$  generation of the 
BRW. 
The rescaled measure-valued and historical processes associated with the BRW 
are defined for $n \in \N$ by 
\begin{align}
X^{\sss(n)}_t&=\frac{1}{\gamma n}\sum_{\alpha\in \mc{T}_{nt}}
\delta_{{w^{\sss(n)}_t(t, \alpha)}}\in \mc{M}_F(\R^d)\\
H^{\sss(n)}_t&=\frac{1}{\gamma n}\sum_{\alpha \in \mc{T}_{nt}}
\delta_{{w^{\sss(n)}(t,\alpha)}}\in \mc{M}_F(\mc{D}).
\end{align}
(Adding $\gamma^{-1}$ to the renormalization allows us to take branching rate one in the limiting HBM 
 described below.)
Note that $X^{\sss(n)}_{t}$ assigns mass to ``locations of particles" in $\R^d$ (but does not encode the genealogy) whereas $H^{\sss(n)}_{t}$ assigns mass to genealogical paths leading to those particles.   
If $\mu\circ f^{-1}$ denotes the pushforward of a measure $\mu$ by a measurable map $f$,     
it follows from \eqref{rescaledpathbrw} that
\begin{equation}\label{rescaleHBRW}
H^{\sss(n)}_t=\frac{1}{n}H^{\sss(1)}_{\floor{nt}}\circ\rho_n^{-1}\text{ for all }t\ge 0\text{ and }n\in\N.
\end{equation}

\subsection{Lattice trees}
\label{sec:LT}
A lattice tree is a finite connected set of lattice bonds containing no cycles.  We will be considering lattice trees on $\Z^d$ (with $d>8$) consisting of bonds of $\ell_\infty$ length at most $L$, for some $L\gg 1$.  To be more precise, let $d>8$ and let $D(\cdot)$ be as above.  For a lattice tree $T\ni o\in \Z^d$, and for $m \in \Z_+$,  let $T_{m}$ denote the set of vertices in $T$ of tree distance $m$ from $o$.  In particular, $T_0=\{o\}$ (the root of $T$), and for any $x\in T_{m}$ there is a unique path $w(m,x,T):=(w_j(m,x;T): 0\le j\le m)$, from $o$ to $x$ in the tree, of length (number of bonds) $m$.  For $t\ge 0$ 
define $T_t=T_{\floor{t}}$.

We now describe how to choose a \emph{random}  tree $\mc{T}(\omega)$, which is defined on some underlying space $(\Omega,\mc{F})$.  For a lattice tree $T\ni o$ and $z>0$ define $W_{z,D}(T)=z^{|T|}\prod_{e\in T}D(e)$, where the product is over the edges in $T$ and $|T|$ is the number of edges in $T$.    
There exists a critical value $z_{\sss D}$ such that $\rho=\sum_{T\ni o}W_{z_{\sss D},D}(T)<\infty$ and $\E[|\mc{T}|]=\infty$, where $\P(\mc{T}=T)=\rho^{-1}W_{z_{\sss D},D}(T)$ for $T\ni o$ (see e.g.~\cite{H08,HH13}).  Hereafter we write $W(\cdot)$ for the critical weighting $W_{z_{\sss D},D}(\cdot)$ and suppose that we are selecting a random tree $\mc{T}=\mc{T}(\omega)\ni o$ according to this critical weighting.   

Therefore, for $m \in \Z_+$ and any $x\in \mc{T}_{m}$, the unique path in $\mc{T}$ from $o$ to $x$ is $w(m,x)=w(m,x;\mc{T})$, and in particular $w_0(m,x)=o$ and $w_m(m,x)=x$.  For $u\ge 0$,  
set $w_u(m,x)=w_{\floor{u}\wedge m}(m,x)$, and for $t\ge 0$ define $w^{\sss(n)}_\cdot(t,x)=w_{\cdot\wedge t}^{\sn}(t,x)\in\mc{D}$ by 
\begin{align}
w^{\sss(n)}_s(t,x)=\frac{w_{ns}(\floor{nt},x)}{\sqrt{n}}=\rho_n(w(\floor{nt},x))_s,
 \quad \text{ for }s\ge 0.\label{rescaledpath}
\end{align}

There are constants $C_A,C_V>0$ (depending on $d,D$) such that \cite{H08,HH13} 
\begin{equation}
\E[|\mc{T}_n|]\to C_A \label{pop_size_converge},
\end{equation}
 and 
 \begin{equation}\label{survasym}
 n \P(\mc{T}_n\ne \varnothing)\to 2/(C_AC_V) \text{ as }n \to \infty. 
 \end{equation}  
 Let $C_0=C_A^2C_V$ and 
\begin{align}
X^{\sss(n)}_t=&\frac{1}{C_0n}\sum_{x\in \mc{T}_{nt}}\delta_{x/\sqrt n}=\frac{1}{C_0n}\sum_{x\in \mc{T}_{nt}}\delta_{{w_t^{\sss(n)}(t,x)}}\in \mc{M}_F(\R^d)\label{def_x}\\ \label{def_h}
H^{\sss(n)}_t=&\frac{1}{C_0n}\sum_{x\in \mc{T}_{nt}}\delta_{{w^{\sss(n)}(t,x)}}\in \mc{M}_F(\mc{D}),
\end{align}
respectively, denote the (rescaled) standard and historical ``processes''  associated with the random lattice tree $\mc{T}$.   We may again use \eqref{rescaledpath}  in \eqref{def_h} and conclude that
\begin{equation}\label{rescaleHLT}
H^{\sss(n)}_t=\frac{1}{n}H^{\sss(1)}_{\floor{nt}}\circ\rho_n^{-1} \text{ for all }t\ge 0\text{ and }n\in\N.
\end{equation}

\subsection{Scaling limits}\label{sec:scalelim}

For $\phi:E \ra \C$ and $Y_t\in \mc{M}_F(E)$ write $Y_t(\phi)=\int \phi \d Y_t$.  In either of the above settings we have 
$H^{\sss(n)}_{t}(1)\equiv X^{\sss(n)}_{t}(1)$, and we 
define the survival/extinction time as 
\begin{equation}\label{def_S}
S^{\sss(n)}:=\inf\{t>0:X^{\sss(n)}_t(1)=0\}=\inf\{t>0:H^{\sss(n)}_t(1)=0\}.
\end{equation}
Then for both GW and lattice trees (the latter for $d>8$ and $L$ large enough) we have that there exists $C_1>0$ such that 
\begin{equation}
n\P(\mc{T}_{\floor{nt}}\ne \varnothing)=n\P(H^{\sss(n)}_{t}(1)>0)=n\P(S^{\sss(n)}>t)\ra \frac{2}{C_1t}, \qquad \text{ as }n\ra \infty.\label{surv1}
\end{equation}
For GW $C_1=\gamma$ (this is Kolmogorov's classical result).  For lattice trees $C_1=C_AC_V$ by  \eqref{survasym}.

Let $B$ denote a $d$-dimensional Brownian motion, with covariance matrix $\sigma_0^2$   times the identity.  Then the path-valued process $(B\vert_{[0,t]})_{t\ge 0}$ is a (time-inhomogeneous) Markov process in $t$. 
According to \cite[pages 34, 64]{LeGall}, for any $\sigma_0^2>0$, 
there exists a sigma-finite measure $\N_{\sss H}=\N_{\sss H}^{\sigma_0^2}$ on $\mc{D}_{\infty}(\mc{M}_F(\mc{D}))$ with $\N_{\sss H}(H_t(1)>0)=\frac{2}{t}$  such that $\N_{\sss H}$ is the canonical measure associated to the $(B\vert_{[0,t]})_{t\ge 0}$-superprocess, $(H_t)_{t\ge 0}$, with branching rate $1$. In fact the random measures $H_t,t\ge 0$ are all supported on the space, $\mc{C}$, of continuous $\R^d$-valued paths $\N_{\sss H}$-a.e. and $H_\cdot$ 	is a continuous $\M_F(\mc{C})$-valued process $\N_{\sss H}$-a.e. (see Sections II.7, II.8 and V.2 of \cite{Per02}.) ($\mc{C}$ is given the topology of uniform convergence on compacts.) We call $H$ historical Brownian motion or the historical process associated with $B$.  Let $S$ denote the extinction time of $H$, that is, 
\begin{equation}
\label{def_S_2}
S=\inf\{t>0:H_t(1)=0\}.
\end{equation}
  Then by the above and the fact that the historical process sticks at $0$ once it hits 0, 
  \begin{equation}\label{SBMsurv}\N_{\sss H}(S>t)=2/t.
  \end{equation}
  
Let $\N_{\sss H}^s(\cdot)=\N_{\sss H}^{s,\sigma_0^2}(\cdot)=\N_{\sss H}^{\sigma_0^2}(\cdot\,|S>s)$ and  $\P^s_n(\cdot)=\P(\cdot|S^{\sn}>s)$. It is proved in \cite{CFHP20} that for lattice trees in dimensions $d>8$ (with $L$ sufficiently large) and  for all $s>0$ 
\begin{align}
\P_n^s(H^{\sn}\in\cdot)\cweak \N_{\sss H}^{s,\sigma_0^2}(H\in\cdot)
\label{weak_historical}
\end{align}
as probability measures on $\mc{D}_\infty(\mc{M}_F(\mc{D}))$, where  $\sigma_0^2=\sigma_0^2(d,D)$ can be expressed in terms of an explicit infinite series (see e.g.~\cite[(3.49)]{H08}, with our $\sigma_0^2$ being equal to $v\sigma^2/d$ in the notation of that paper). The corresponding historical convergence for BRW (with no restrictions on $d$ or $L$) is much easier and can be proved as in the proof of Theorem~II.7.3 and Remark~II.7.4  in \cite{Per02} which includes the historical setting by Section II.8 of \cite{Per02}. In this case the limiting variance parameter $\sigma_0^2=\sigma_0^2(d,D)$ is just the common variance of each coordinate under $D$. Remark~II.7.4 of \cite{Per02} actually includes the case where $D$ is replaced by a normal law with covariance $\sigma_0^2 I$ by taking independent Brownian motions as the spatial motions considered there.  The martingale problem arguments underlying the proof readily adapt to the setting where the iid displacements have finite $4$th moments and hence include our $D$.

\subsection{Choosing individuals at random}
\label{sec:enlargement}
We continue to work in the lattice tree or branching random walk setting on our $(\Omega,\F,\P)$.  
Our main results require us to enlarge our underlying space $(\Omega,\mc{F},\P)$ to include an i.i.d.~sequence of random quantities with some law depending on $\omega\in \Omega$.  For both lattice trees and BRW we choose $K$  individuals uniformly at random from the random tree $\mc{T}$.

Let $\nu$ be a probability on some measurable space $(\Omega_0,\mc{F}_0)$, and let $M:\Omega_0 \to \mc{M}_F(G)$ denote a $\mc{F}_0$-measurable random measure on a measurable space $(G,\mc{G})$ such that $M\neq 0$ a.s.  We can enlarge $\Omega_0$ so that  we also have a sequence $(Z_i)_{i \in \N}$ of random elements of $G$ with the property that $\nu$-a.s., given $\omega\in\Omega_0$, $(Z_i)_{i \in \N}$ are i.i.d.~with law $\mc{L}(\omega):=M(\omega)/[M(\omega)(G)]$ (define $\mc{L}(\omega)$ to be a convenient law on $G$ on the null set where $M(\omega)$ is zero).  This can be achieved as follows.  For each $\omega \in \Omega_0$, let $P_\omega$ be a probability measure on the product space $(G^{\N},\mc{G}^{\N})$ under which the coordinate variables $(Z_i)_{i \in \N}$ are i.i.d.~with law $\mc{L}(\omega)$.  Measurability of $P_\omega(B)$ in $\omega$ for $B\in \mc{G}^{\N}$ is elementary and so we may define a probability   $\hat{\nu}(M)$ on $(\hat{\Omega}_0,\hat{\mc{F}}_0)=(\Omega_0 \times G^{\N},\mc{F}_0\times \mc{G}^{\N})$ by 
\begin{equation}
\label{nuhat}
\hat{\nu}(M)(A\times B)=\int_A P_\omega(B)\d\nu(\omega).
\end{equation}
Then $P_\omega$ is a regular conditional distribution for $(Z_i)_{i\in\N}$ given $\Pi(\hat{\omega})=\omega$, where $\Pi:\hat\Omega_0\to\Omega_0$ is the projection map. Often we write $\hat\nu$ for $\hat\nu(M)$ if the choice of $M$ is clear from the context.

For $n\in\N$ and $u\ge 0$ we introduce
\begin{equation}\label{brackdef}
[u]_n=\lfloor nu\rfloor/n,
\end{equation}
and if $w\in\D$, we let
\begin{equation}\label{Ldef}
\FL(w)=\inf\{u\ge 0: w\text{ is constant on }[u,\infty)\}\in[0,\infty]\quad(\inf\emptyset=\infty)
\end{equation}
be the ``lifetime" of $w$.  

Note that for lattice trees if $x\in\T_m$, then it follows from $D(o)=0$ that
\begin{equation}\label{LTdist}
\FL(w(m,x))=m\ \ \text{ is the tree distance from $x$ to the root at $o$}.
\end{equation}
Similarly (with $\FL(w(m,\alpha))=m$ for $\alpha\in \mc{T}_m$) in the case of BRW.
In each case it follows by scaling 
that if $x\in\T_{nt}$ 
then
\begin{equation}\label{LTndist}\FL(w^{\sn}(t,x))=\FL(w(\lfloor nt\rfloor, x))/n=[t]_n\ \ \text{is the rescaled graph ``distance'' from $x$ to the root}.
\end{equation}

Let $\overline{\mc{D}}=\R_+\times\mc{D}$ and define random a.s. finite measures on $\overline{\mc{D}}$ by 
\[I^{\sn}(A\times B)=\int_0^\infty\int \1_A(u)\1_B(w)H_u^{\sn}(\d w)\d u,\]
and
\[I(A\times B)=\int_0^\infty\int\1_A(u)\1_B(w)H_u(\d w)\d u,\]
Let $0_M$ denote the zero measure and define the measure-normalising function $F_1:\mc{M}_F(\oD) \to \mc{M}_F(\oD)$ by 
\begin{align}
F_1(\mu)=\begin{cases}
0_M, & \text{ if }\mu=0_M,\\
\dfrac{\mu}{\mu(\oD)}, & \text{ otherwise}.
\end{cases}\label{FK}
\end{align}
We set $J^{\sn}=F_1(I^{\sn})$, which is a probability on $\oD$ since $o\in\mc{T}$, and  $J=F_1(I)$,  which is a.s. a probability on $\oD$ because the excursion measure $\N_{\sss H}$ assigns no weight to the zero path. The fact that $H_u$ is supported on paths in $\mc{C}$ starting at $0$ and constant on $[u,\infty)$ $\N_{\sss H}$-a.e. (for the latter properties use (II.8.6)(a) of \cite{Per02} to see the mean measure of $H_u$ under $\N_{\sss H}$ is a Brownian motion starting at $0$ and stopped at $u$) implies
\begin{equation}\label{Jsupp}
\text{$J$ is supported on $\oC_0=\{(u,w)\in\R_+\times \mc{C}:w_0=0\text{ and $w$ is constant on }[u,\infty)$\}.}  
\end{equation}
Note that for either BRW or lattice trees, given $\mc{T}$,
\begin{equation}\label{Jnunif}J^{\sn}(A\times B)=\frac{1}{|\mc{T}|}\sum_{t \in \Z_+/n}\sum_{x \in \mc{T}_{nt}}\1_A(t)\1_B(w^{\sn}(t,x)).\end{equation}
Therefore $J^{\sn}(\R_+\times \cdot)$ is the law of the rescaled  path associated to a uniformly chosen point in $\mc{T}$. 
Enlarge the probability space $(\Omega, \mc{F},\P)$ as above so that given $\mc{T}$, $(\mathfrak{T}^{\sn}_i,W^{\sn}_i)_{i \in \N}$ are i.i.d.~with law $J^{\sn}$ under $\hat{\P}(J^{\sn})$, \blank{\todo{do we want to already condition on survival here?}  }where $W^{\sn}_{i}=(W^{\sn}_{i,t})_{t\ge 0}$. Let $V^{\sn}_i=W^{\sn}_{i,\FL(W^{\sn}_i)}\in\R^d$. 

The same construction and results hold with $\P_n^s$ in place of $\P$ leading to probabilities
 $\{\hat{\P}^s_n(J^{\sn}):s>0,n\in\N\}$.

\begin{REM}
\label{rem:alphahat}
We note here that for BRW, $(\mathfrak{T}^{\sn}_i,W^{\sn}_i)_{i \in \N}$ are the paths associated to uniformly chosen points $(\hat{\alpha}_i)_{i \in \N}$ in $\mc{T}$.  That is, given $\mc{T}$ we first choose the $(\hat{\alpha}_i)_{i \in \N}$ as independent uniform points from $\mc{T}$,  then define 
\begin{equation}\label{WBRW}W^{\sn}_{i,s}=w_{ns}(|\hat\alpha_i|,\hat\alpha_i)/\sqrt{n},\ i \in \N
\end{equation} 
(as in \eqref{rescaledpathbrw} with $\alpha=\hat{\alpha}_i$  and $nt=|\hat{\alpha}_i|$), and set $\mathfrak{T}^{\sn}_i=|\hat{\alpha}_i|/n$ (recall \eqref{LTndist}). \Enddef
\end{REM}
We enlarge the probability space $(\Omega_{\sss H},\mc{F}_{\sss H}, \N^s_{\sss H})$ for our HBM as well so that under a probability $\hat{\N}^s_{\sss H}(J)$, 
\begin{equation}\label{WinCK}
\text{given $(I,J,S)$, $(\mathfrak{T}_i,W_i)_{i\in\N}$ are i.i.d. in $\oC_0$ (recall \eqref{Jsupp}) with law $J$.} 
\end{equation}
(We usually drop the $J$ and write $\hat{\N}^s_{\sss H}$.)  Then elementary properties of HBM (see Lemma~\ref{lem:distlife} below) imply $\mathfrak{T}_i=\FL(W_i)$, $\hat{\N}^s_{\sss H}$-almost surely.     The collection of the continuous paths $(W_i)_{i \le K}$, truncated at the times $(\mathfrak{T}_i)_{i \le K}$ respectively, can be viewed as a  skeleton to $K$ points chosen i.i.d.~``uniformly'' from the SBM cluster. 

\subsection{Shapes and Graph Spatial Trees}
\label{sec:GST}
We largely follow Section~7 of Croydon \cite{Croydon09} to introduce the space of graph spatial trees, although due to our more specialized application, we will be able to work in the simpler setting of graph spatial trees with {\it non-degenerate}  shapes. After introducing the abstract definitions we specialize to the setting we will use: the graph spatial tree of $K$ continuous paths.  We begin with a definition from the lace expansion literature (e.g. see Definition 2.5 in \cite{CFHP20}).
\begin{DEF} \label{ndshape}Let $K\in\N$ and $T$ be a finite rooted tree graph with $K+1$ labelled vertices, $0,1,\dots,K$, where $0$ is 
the root, and all of them are leaves (hence have degree one). We assume the the remaining vertices are all binary branching vertices (hence have degree $3$), that is,  they each have one parent (closer to the root) and two children.  We call $T$ a {\it non-degenerate shape}. We let $\Sigma_K$ denote the set of such non-degenerate shapes where isomorphic shapes (isomorphisms must also preserve the labelings) are identified. \Enddef
\end{DEF}

\begin{REM}  
A simple induction on $K$ shows that a shape $T$ in $\Sigma_K$  has $K+1$ leaves, $K-1$ vertices of degree $3$, and $2K-1$ edges in $E(T)$, the edge set of $T$. 
There is one shape with $K=2$, three shapes with $K=3$, and in general (for $K\ge 2$) $\prod_{j=2}^K(2j-3)$ shapes in $\Sigma_K$. If $v_1,v_2\in T$ (formally we mean they are vertices of $T$), let $[[v_1,v_2]]$ denote the unique path of non-overlapping edges from $v_1$ to $v_2$. \Enddef
\end{REM}
\begin{DEF}\label{metricdef} If $T\in\Sigma_K$ and $\ell$ is a strictly positive function on $E(T)$, we call $\ell(e)$ the length of edge $e$ and define a metric $d=d_\ell$ on $T$ by $d(v_1,v_2)=\sum_{e\in[[v_1,v_2]]}\ell(e)$.  Embed $(T,d)$ in a metric space $(\overline T,d)$ by adding line segments along the edges of $T$, and linearly extend $d$ to a metric on $\overline T$, so that the line segment along  edge $e$ is isometric to $[0,\ell(e)]$. If $\phi:\overline T\to \R^d$ is continuous and maps the root to the origin, we call $(T,d,\phi)$ a graph spatial tree (gst).    
The space of all gst's together with a distinguished additional point  $\degK$ is denoted by $\TT_{gst}$. \Enddef
\end{DEF}

\begin{figure}
\begin{tikzpicture}[nn/.style={circle,scale=0.2,fill=black}]
\node[circle,scale=0.2,fill=black,label=left:{$0$}] (O) at (0,0) {};
\node[nn,] (O1) at (1,0) {};
\node[nn,] (O11) at (2,1) {};
\node[nn,] (O12) at (2,-1) {};
\node[nn,] (O111) at (3,2) {};
\node[nn,label=right:{$2$}] (O112) at (3,1) {};
\node[nn,] (O121) at (3,0) {};
\node[nn,] (O122) at (3,-2) {};
\node[nn,label=right:{$1$}] (O1111) at (4,2.5) {};
\node[nn,label=right:{$4$}] (O1112) at (4,1.5) {};

\node[nn,label=right:{$3$}] (O1211) at (4,0.5) {};
\node[nn,label=right:{$5$}] (O1212) at (4,-0.5) {};

\node[nn,label=right:{$6$}] (O1221) at (4,-1.5) {};
\node[nn,label=right:{$7$}] (O1222) at (4,-2.5) {};
\draw (O)--(O1)--(O11)--(O111)--(O1111);
\draw (O111)--(O1112);
\draw (O11)--(O112);
\draw (O1)--(O12)--(O121)--(O1211);
\draw (O121)--(O1212);
\draw (O12)--(O122)--(O1221);
\draw (O122)--(O1222);
\node at (1.4, 0.7) {\tiny $e_{9}$};
\node at (0.4, 0.2) {\tiny $e_{8}$};
\node at (2.4, -0.3) {\tiny $e_{12}$};
\node at (2.4, -1.7) {\tiny $e_{13}$};
\node at (2.4, 1.7) {\tiny $e_{10}$};
\node at (3.4, 2.4) {\tiny $e_{1}$};
\node at (2.4, 0.8) {\tiny $e_{2}$};
\node at (3.4, -2.4) {\tiny $e_{7}$};
\node at (1.4, -0.7) {\tiny $e_{11}$};
\end{tikzpicture}
\hspace{2cm}
\begin{tikzpicture}
[nn/.style={circle,scale=0.2,fill=black}]
\node[circle,scale=0.2,fill=black,label=left:{$0$}] (O) at (0,0) {};
\node[nn,] (O1) at (1,0) {};
\node[nn,] (O11) at (2,1) {};
\node[nn,] (O12) at (2,-1) {};
\node[nn,] (O111) at (3,2) {};
\node[nn,label=right:{$2$}] (O112) at (3,1) {};
\node[nn,] (O121) at (3,0) {};
\node[nn,] (O122) at (3,-2) {};
\node[nn,label=right:{$1$}] (O1111) at (4,2.5) {};
\node[nn,label=right:{$4$}] (O1112) at (4,1.5) {};

\node[nn,label=right:{$3$}] (O1211) at (4,0.5) {};
\node[nn,label=right:{$5$}] (O1212) at (4,-0.5) {};

\node[nn,label=right:{$6$}] (O1221) at (4,-1.5) {};
\node[nn,label=right:{$7$}] (O1222) at (4,-2.5) {};
\draw (O)--(O1)--(O11)--(O111)--(O1111);
\draw (O111)--(O1112);
\draw (O11)--(O112);
\draw (O1)--(O12)--(O121)--(O1211);
\draw (O121)--(O1212);
\draw (O12)--(O122)--(O1221);
\draw (O122)--(O1222);
\draw[ultra thick] (1.5,0.5)--(1,0)--(2,-1)--(2.5,-0.5);
\node at (1.8,0.2) {\tiny $\ell(e_{9})/2$};
\node at (1.1,-0.6) {\tiny $\ell(e_{11})$};
\node at (2.9,-0.7) {\tiny $\ell(e_{12})/2$};
\end{tikzpicture}
\caption{A non-degenerate shape $T\in\Sigma_7$ also showing some of the corresponding edge labels.  If $\ell:E(T)\to (0,\infty)$ is given then the corresponding metric space $(\bar{T},d)$ has e.g.~the distance between the midpoints of $e_{9}$ and $e_{12}$ equal to $\ell(e_{9})/2+\ell(e_{11})+\ell(e_{12})/2$}
\label{fig:7shape}
\end{figure}
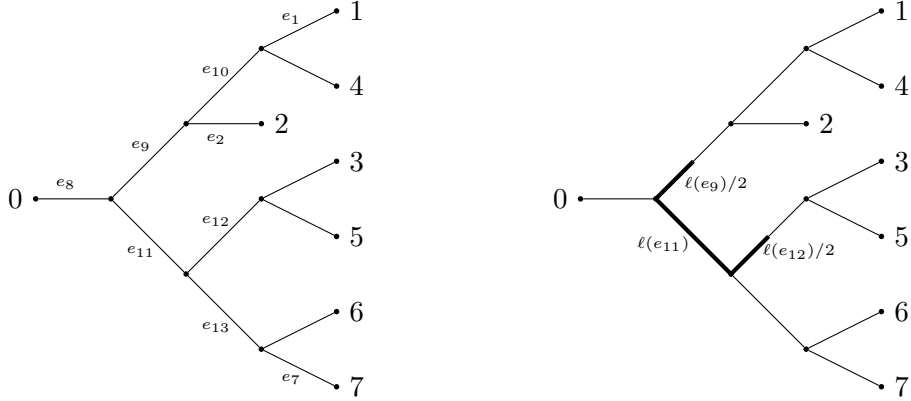

\begin{REM} The definition of $(\overline T,d)$ applies to any rooted tree $T$ equipped with a metric $d$.\Enddef
\end{REM}

\begin{REM}\label{rem:labels} (a) In practice the point $\degK$  
will represent gst's, defined as above, but  arising from a shape $T$ which is degenerate, that is, a rooted tree which does not satisfy the degree assumptions in Definition~\ref{ndshape}.

(b) In \cite{Croydon09} it is assumed that the edges of $T\in\Sigma_K$ are labelled as $e_1,\dots,e_{2K-1}$.  In our non-degenerate setting this can be easily done using the labeling of the leaves. We label the $K-1$ remaining branching vertices (those of degree $3$) as $K+1,\dots,2K-1$ in the order you encounter them when moving from the root to vertex $1$, then continue the labeling as you encounter {\it new} branching  vertices when you move from the root to vertex $2$, and so on up to vertex $K$. For {\it any} labeling of the vertices, if $v$ is a non-root vertex we let $e_v$ be the unique edge between $v$ and its parent.  See e.g. Figure \ref{fig:7shape}. In particular in the above labelling we have labelled the edges as $e_1,\dots,e_{2K-1}$. Note that any graph isomorphism preserving the labeling of the $K+1$ leaves, will preserve this labeling of all vertices and edges.    
Hence we {\it usually} will just identify $T\in\Sigma_K$  
 with the tree with vertices $\{0,\dots, 2K-1\}$ and edges $e_1,\dots, e_{2K-1}$, thus (as in \cite{Croydon09}) avoiding explicit identification of isomorphic trees. However the above edge-labelling convention can be used for any labelling of the leaves. 

If $T$ is a shape and $i\wedge j$ is the greatest common ancestor of $i$ and $j$ in $T$, then every branching vertex $v$ in $T$ can be written (non-uniquely) as $i\wedge j$ for $i\neq j\in [K]\cup\{0\}$ (choose $i,j$ which are descendants of the two children of $v$). In practice we will often use this labelling of branching vertices.\Enddef
\end{REM}

Let $(T,d,\phi),(T',d',\phi')$ be (non-$\degK$) gst's, and let $\ell$ and $\ell'$ be the edge-length functions corresponding to $d$ and $d'$, respectively. If $T=T'$ define 
\[d_1\big((T,d),(T',d')\big)=\sup_{1\le i\le 2K-1}|\ell(e_i)-\ell'(e_i)|,\]and if $T\neq T'$  
the above distance is defined to be $\infty$.  
If $T=T'$  we have a homeomorphism $\Upsilon_{d,d'}:(\overline T,d)\to(\overline{T'},d')$ which maps a point $x$, which is $d$-distance $\alpha\in[0,\ell(e)]$ along the edge $e$ (measured from the endpoint closest to the root), to the point $x'$ which is $d'$-distance $\alpha\ell'(e)/\ell(e)$ along the edge $e$. Let $\Vert\cdot\Vert$ be the Euclidean norm on $\R^d$ and define
\[d_2\big((T,d,\phi),(T',d',\phi')\big)=\sup_{x\in\overline T}\Vert\phi(x)-\phi'(\Upsilon_{d,d'}(x))\Vert.\]
Set the above distance to be $\infty$ if $T\neq T'$.  
 We define a metric on $\TT_{gst}$ by 
 \[D\big((T,d,\phi),(T',d',\phi')\big)=\big(d_1\big((T,d),(T',d')\big)+d_2\big((T,d,\phi),(T',d',\phi')\big)\big)\wedge 1,\]
 $D((T,d,\phi),\degK)=1$, and $D(\degK,\degK)=0$.  
If $(T,d)=(T',d')$ then $\Upsilon_{d,d'}$ is the identity, and so $D((T,d,\phi),(T',d',\phi'))=0$ implies 
$(T,d,\phi)=(T',d',\phi')$. In fact, 
 $(\TT_{gst},D)$ is a separable metric space (see Section~7 of \cite{Croydon09} and the references cited there).

\medskip

We next introduce the non-degenerate shape and gst associated to a collection of $K$ paths, each starting at the origin. If $w,w'\in\D$, let 
\begin{equation}\label{bartaudefn}\bar\tau(w,w')=\inf\{t:w_{t}\neq w'_{t}\}\in[0,\infty] \ (\inf\emptyset=\infty),
\end{equation}
 and define the branch time of $w$ and $w'$ by
\begin{equation}
\tau(w,w')=\bar\tau(w,w')\wedge \FL(w)\wedge \FL(w'). \label{tau_def}
\end{equation}
For $K\in\N$ and $w=(w_1,\dots,w_K)\in \mc{D}^K$, let $\bs{\tau}(w)=(\tau_{i,j}(w))$ be the $K\times K$ branching matrix given by
\[\tau_{i,j}(w)=\tau(w_i,w_j).\]
Let $C_K$ be the set of $w=(w_1,\dots,w_K)\in \mc{C}^K$ such that $w_i(0)=o$ and $\FL(w_i)<\infty$ for all $i\le K$. The following is clear for any $\bs{\tau}\in R_K=\{\bs{\tau}(w):w\in C_K\}$:
\begin{equation}\label{tauprops}
0\le \tau_{i,j}=\tau_{j,i}\le \tau_{i,i}\wedge \tau_{j,j}<\infty\text{ and }\tau_{i,j}\wedge\tau_{j,k}\le\tau_{i,k}\text{ for any }i,j,k\in [K].
\end{equation}
For $w\in\mc{D}$ and $u\ge 0$, let $w\vert_u $ be the restriction of $w$ to $[0,u]$. 
\begin{DEF}We say $\bs{\tau}\in R_K$ is non-degenerate iff 
\begin{equation}\label{noanc}\text{for all distinct }i,j\in[K],\ \tau_{i,j}\in(0,\tau_{i,i}\wedge \tau_{j,j}),
\end{equation}
and
\begin{equation}\label{binaryc}\text{there are no distinct $i,j,k\in[K]$ such that }\tau_{i,j}=\tau_{i,k}=\tau_{j,k}.
\end{equation}
Otherwise we say $\bs{\tau}$ is degenerate. If $\bs{\tau}=\bs{\tau}(w)$, where $w\in C_K$, we also say $w$ is non-degenerate or degenerate, according to the status of $\bs{\tau}$.\Enddef
\end{DEF}

Intuitively speaking, the first condition ensures that for $\bs{\tau}=\bs{\tau}(w)$, in ``the natural tree associated with $w$", the ``ends'' of paths are associated to leaves in the tree 
 and the second condition ensures the remaining vertices will have degree three. 

Let $w=(w_1,\dots,w_K)\in C_K$. We now construct a number of objects which depend on this fixed $w$.  Our definition of the associated shape will in fact depend only on $\bs{\tau}=\bs{\tau}(w)$ but we will often interpret definitions in terms of $w$ as well for clarity. Assume $\bs{\tau}$ is non-degenerate.  Let $\hat T=\hat T(\bs{\tau}):=\cup_{i=1}^K\{i\}\times[0,\tau_{i,i}]$ and define an equivalence relation (the transitivity uses the last property in \eqref{tauprops}) on $\hat T$ by  
\begin{equation}\label{eqrel}(i,u)\sim (j,v)\text{ iff }u=v\le \tau_{i,j},\text{ that is, iff }u=v\text{ and } w_i|_u=w_j|_u.
\end{equation}
We let $[i,u]$ be the equivalence class containing $(i,u)$ and set $\widetilde T=\widetilde T(\bs{\tau})=\hat T(\bs{\tau})/{\sim}$, that is,
\begin{equation}\label{eqclass}
[i,u]=\{(k,u):\tau_{i,k}\ge u\}=\{(k,u):w_k\vert_u=w_i\vert_u\}.
\end{equation}
\begin{figure}
\begin{center}
\begin{tikzpicture}[scale=0.8]
\node[circle,fill=black, scale=0.1] (A1) at (0,0) {};
\node at (-0.2,0) {$0$};
\node[circle,fill=black,scale=0.1] (B1) at (3,0) {};

\node[circle,fill=black, scale=0.1] (A2) at (0,-0.1) {};
\node[circle,fill=black,scale=0.1] (B2) at (3,-0.1) {};
\node[circle,fill=black,scale=0.1] (E1) at (5.5,0.5) {};
\node[circle,fill=black,scale=0.1] (F1) at (7.9,0) {};
\node[circle,fill=black,scale=0.1] (D) at (3.87,-0.6) {};
\node[circle,fill=black, scale=0.1] (A3) at (0,0.1) {};
\node[circle,fill=black,scale=0.1] (B3) at (3,0.1) {};
\node[circle,fill=black,scale=0.1] (C2) at (6.9,0.8) {};
\node[circle,fill=black,scale=0.1] (D) at (3.87,-0.6) {};
\node[circle,fill=black,scale=0.1] (E3) at (5.5,0.6) {};
\draw[thick] (A1)--(B1)--(E1)--(F1);
\draw[thick] (A2)--(B2)--(D);
\draw[thick] (A3)--(B3)--(E3)--(C2);
\draw[dashed] (3,-1)--(3,1);
\draw[dashed] (5.5,-1)--(5.5,1);
\draw[dashed] (3.87,-1)--(3.87,1);
\draw[dashed] (6.9,-1)--(6.9,1);
\draw[dashed] (7.9,-1)--(7.9,1);
\node at (3,-1.5) {$\tau_{1,2}$};
\node at (5.5,-1.5) {$\tau_{2,3}$};
\node at (3.87,1.5) {$\tau_{1,1}$};
\node at (6.9,1.5) {$\tau_{2,2}$};
\node at (7.9,1.5) {$\tau_{3,3}$};

\node[circle,fill=black,scale=0.2,label=above:{$0=\langle 0,0\rangle$}]  at (10,0) {};
\node[circle,fill=black,scale=0.2,label=below:{$\langle 1,2\rangle$}]  at (12,0) {};
\node[circle,fill=black,scale=0.2]  at (14,-1) {};
\node[circle,fill=black,scale=0.2,label=above:{$\langle 2,3\rangle$}]  at (14,1) {};
\node[circle,fill=black,scale=0.2]  at (16,1.5) {};
\node[circle,fill=black,scale=0.2]  at (16,0.5) {};
\draw[thick] (10,0)--(12,0)--(14,1)--(16,1.5);
\draw[thick] (14,1)--(16,0.5);
\draw[thick] (12,0)--(14,-1);
\node at (17.25,0.5) {$3=\langle 3,3\rangle$};
\node at (17.25,1.5) {$2=\langle 2,2\rangle$};
\node at (15.25,-1) {$1=\langle 1,1\rangle$};
\end{tikzpicture}
\end{center}
\caption{Left: a depiction of 3 paths ($w_2,w_3,w_1$ from top to bottom) for which the associated $\bs{\tau}$ is non-degenerate.  Close parallel lines  depict parts of paths that exactly coincide.  
Right: The corresponding shape $T(\bs{\tau})$ (labels 2 and 3 can be swapped without changing the shape).  Edge lengths are differences of $\tau_{\cdot,\cdot}$  (e.g.~the edge adjacent to 2 has length $\ell(e_{\langle 2,2 \rangle})=\tau_{2,2}-\tau_{2,3}$) and these edge lengths define $d$ on $T(\bs{\tau})$.
}
\label{fig:3paths_shape}
\end{figure}
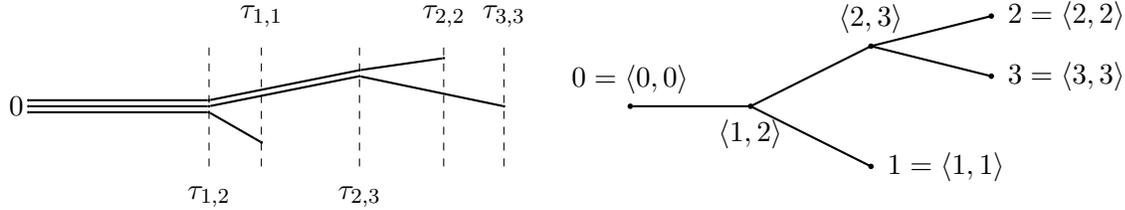
See, for example, Figure \ref{fig:3paths_shape}, where $[1,u]=[2,u]=[3,u]$ for $u\le \tau_{1,2}$, while for $u\in (\tau_{1,2},\tau_{1,1}]$
$[2,u]=[3,u]\ne [1,u]$.

A non-negative function $d=d_{\bs{\tau}}$ on $\widetilde T$ is defined by
\begin{equation}\label{treemet}d_{\bs{\tau}}([i,u],[j,v])=\begin{cases}
u+v-2\tau_{i,j}, & \text{ if }\min(u,v)>\tau_{i,j}\\
|v-u|, & \text{ otherwise}.
\end{cases}
\end{equation}
Later (below Lemma~\ref{lem:samemet}) we will see that $(\widetilde T,d)$ is a real tree (and in particular $d_{\bs{\tau}}$ is a metric) as defined, e.g., in Section~2.1 of \cite{Croydon09}.  We will not use the notion of a real tree, but instead will be primarily focused on the finite tree graph 
consisting of leaves and branch points of $\widetilde T$, which we will denote by $T(\bs{\tau})$ and define next.
We introduce (use \eqref{eqclass} in the second equality below)
\begin{align}\label{eqclass2}\langle i, j\rangle:=&[i,\tau_{i,j}]=\{(k,\tau_{i,j}):\tau_{i,k}\ge \tau_{i,j}, k\in[K]\}=\{(k,\tau_{i,j}):w_k\vert_{\tau_{i,j}}=w_i\vert_{\tau_{i,j}}, k\in[K]\},\\
&\nonumber\text{for }i,j\in[K], \text{ and } \langle 0,0\rangle=\lb i,0\rb=\lb 0,i\rb:=[i,0] \text{
(for any $i\in[K]$)}.
\end{align}
From the symmetry of $\tau_{i,j}$, the fact that $w_i\vert_{\tau_{i,j}}=w_j\vert_{\tau_{i,j}}$ and the last equality in \eqref{eqclass2} we see that $\langle i,j\rangle$ is symmetric in $i$ and $j$.  
We set 
\begin{equation}\label{leafnot}i:=\lb i,i \rb=\{(i,\tau_{i,i})\}\text{ for }i\in\{0,1,\dots,K\},
\end{equation}
where the last equality holds by definition and non-degeneracy of $\bs{\tau}$.
By \eqref{eqrel}, 
\begin{equation}\label{eqclass3}
\aij=\lb k,\ell\rb \text{ iff }\tau_{i,j}=\tau_{k,\ell}\le \tau_{i,k}.
\end{equation}
The vertex set of $T(\bs{\tau})$ is $\{\langle i,j\rangle:i,j\in [K]\text{ or }i=j=0\}$ which we also denote by $T(\bs{\tau})$. 
We call $\lb 0,0\rb$ the root of $T(\bs{\tau})$.

To define the edges, set $\tau_{0,i}=\tau_{i,0}=0$ for all $i\in\{0,\dots K\}$, and then for $i,j\in[K]$, choose $m(i,j)\in\{0,\dots,K\}$ so that
\begin{equation}\label{mdef}\tau_{i,m(i,j)}=\max\{\tau_{i,k}:\tau_{i,k}<\tau_{i,j}, k=0,\dots,K\}(<\tau_{i,j}).
\end{equation}
The above set is non-empty because $\tau_{i,0}=0<\tau_{i,j}$ by non-degeneracy.   If there are multiple choices of $m(i,j)$, choose the smallest, although this choice will not affect $\tau_{i,m(i,j)}$ or $\langle i,m(i,j)\rangle$. (The first is obvious and the second is then easy to see from \eqref{eqclass3}.) 
We define (for $i,j\in[K]$) the parent of $\langle i,j\rangle$ to be 
\begin{equation}\label{pidefn}\pi(\langle i, j\rangle)=\langle i, m(i,j)\rangle.\end{equation}
One easily checks that $\pi:T(\bs{\tau})\setminus \{0\}\ra T(\bs{\tau})\setminus\{1,\dots,K\}$ is well-defined (see Lemma~\ref{parentwd} and recall that $i=\langle i,i\rangle$).
Let $e_{\langle i,j\rb}$ denote an edge between $\pi(\lb i,j\rb)$ and $\lb i,j\rb$ for all $i,j\in[K]$. The edge set $E(\bs{\tau})$ for the tree $T(\bs{\tau})$ is the set of these edges for $i,j\in[K]$. 

\begin{PRP}\label{prop:wshape}
If $\bs{\tau}\in R_K$ is non-degenerate, then $T(\bs{\tau})$ is a non-degenerate shape in $\Sigma_K$.
\end{PRP}
\noindent This should be clear now as the points of $T(\bs{\tau})$ are the $K$ leaves, $1,\dots,K$, corresponding to the ``endpoints" of the paths $w_1,\dots,w_K$, the branch points of these $K$ paths (which are always binary by non-degeneracy), and the root, $0$, which will have degree one by non-degeneracy.  The straightforward proof is in Section~\ref{ssec:gsts}.  
It is not hard to see that $i\wedge j=\aij$ (see Remark~\ref{gcaij} below and recall Remark~\ref{rem:labels}(b) above).  

By definition we have $\tau_{i,j}=\tau_{i',j'}$ whenever $(i,\tau_{i,j})\sim (i',\tau_{i',j'})$ and so can define $\tau_{\aij}=\tau_{i,j}$.  We define an edge weight function $\ell=\ell_{\bs{\tau}}$ on $E(\bs{\tau})$ by
\begin{equation}\label{elldef}
\ell(e_{\aij})=\tau_{\aij}-\tau_{\pi(\aij)}>0\text{ for }i,j\in[K].
\end{equation}
We can then use $\ell$ to define a metric $d_\ell$ on the vertices of $T$ as in Definition~\ref{metricdef}. Recall the function $d_{\bs{\tau}}$ on $\widetilde T\times \widetilde T$ from \eqref{treemet}. The next elementary result is proved again in Section~\ref{ssec:gsts}.

\begin{LEM}\label{lem:samemet}
The metric $d_\ell$ is the restriction of $d_{\bs{\tau}}$ to $T\times T$.
\end{LEM}
\noindent Henceforth we will use $d_{\bs{\tau}}$ in place of $d_\ell$.  Note that
\begin{equation}\label{dtauij0}d_{\bs{\tau}}(\aij,0)=d_{\bs{\tau}}([i,\tau_{i,j}],[i,0])=\tau_{i,j}\text{ for all }i,j\in[K].
\end{equation}

If $u,v\in[\tau_{\pi(\aij)},\tau_{\aij}]$, then since $u,v\le \tau_{i,j}\le \tau_{i,i}$ we have from \eqref{treemet} that 
\[d_\tau([i,u],[i,v])=|v-u|.\] 
In particular we see that
\begin{equation}\label{iuchar}
\text{$[i,u]$ is the point a distance $u-\tau_{\pi(\aij)}$ from $\pi(\aij)$ along the edge $e_{\aij}$.}
\end{equation}
The line segment $L_{\aij}=\{[i,u]:\, u\in [\tau_{\pi(\aij)},\tau_{\aij}]\}$ runs from $\pi(\aij)=[i,\tau_{\pi(\aij)}]$ to $\aij=[i,\tau_{\aij}]$.  By the above, $L_{\aij}$ with the metric $d_\tau$ is isometric to the interval $[0,\ell(e_{\aij})]$, and $d_\tau$ on this line segment is the linear extension of $d_\tau$ on $T\times T$ to $\widetilde T\times \widetilde T$.
 Hence we may use $(\widetilde T,d_\tau)$ as a realization of $(\overline T,d)$ in Definition~\ref{metricdef}. (As noted in Section~7 of \cite{Croydon09}, $(\widetilde T,d_\tau)$ is a real tree.) Henceforth we therefore use the notation $\overline T=\overline{T}(\bs{\tau})$ in place of $\widetilde T(\bs{\tau})$.   

Recall that $w\in C_K$. Define $\phi_w:\overline T\to\R^d$ by
 \begin{equation}\label{phiwdef}
\phi_w([i,u])=w_i(u).
\end{equation}
 One easily checks $\phi_w$ is well-defined.  We have $\phi_w(0)=\phi_w([i,0])=w_i(0)=o$ and the continuity of $\phi_w$ is shown in Lemma~\ref{dtauconv}(b) in Section~\ref{ssec:gsts}.  We conclude from this, Proposition~\ref{prop:wshape}, and Definition~\ref{metricdef} that 
\begin{equation}\label{gstw}
\text{for non-degenerate $w\in C_K$, }\mc{B}_K(w):=(T(\bs{\tau}(w)),d_{\bs{\tau}(w)},\phi_w)\text{ is a graph spatial tree}.
\end{equation}
\begin{DEF}\label{gstwdef} Let $w\in C_K$.  If $w$ is non-degenerate define $\mc{B}_K(w)$ by \eqref{gstw}, and call $\mc{B}_K(w)$ the graph spatial tree of $w$. If $w$ is degenerate, set $\mc{B}_K(w)= \degK\in \TT_{gst}$. \Enddef
\end{DEF}
\begin{REM}\label{idgsts}  Let $(T,d)$ and $(T',d')$ be shapes equipped with edge-length metrics. In the definition of a shape we identified these metric spaces if they are linked by an isometric isomorphism (recall from Definition~\ref{ndshape} this means the leaf-labellings are preserved). This meant we could assume the labelled leaves were  $0,1,\dots,K$, where $0$ is the root and then could in fact label the  branch vertices as $K+1,\dots,2K-1$ in a canonical fashion (see Remark~\ref{rem:labels}(b)).  
 At times it will be convenient to work explicitly with distinct leaf-labellings underlying a gst, connected as above by an isomorphism mapping one leaf-labelling to another.  Let $\psi:(T,d)\to (T',d')$ be an isometric isomorphism.  
 Let $\bar\psi:(\overline{T},d)\to (\overline{T'},d')$ be the unique isometric extension of $\psi$, that is, $\bar\psi$ maps the point $x$ which is $d$ distance $\alpha\in[0,\ell(e)]$ along the edge $e$ (measured from the endpoint closest to the root) to the point $x'$ with is distance $\alpha$ along the edge $\psi(e)$.  If we make obvious changes in the definition of $D$ to accommodate different labellings of $T$ and $T'$ (using the edge-labelling from Remark~\ref{rem:labels}(b)), then 
\begin{align}\label{gstidentn}
\nonumber&\text{$(T,d,\phi)$ and $(T',d',\phi')$ are equal in $\TT_{gst}$ iff there is an isometric}\\
&\text{ isomorphism $\psi:(T,d)\to(T',d')$ such that $\phi'=\phi\circ\bar\psi$.}
\end{align}
\end{REM}\Enddef

Turning to the limiting HBM, 
 recall from \eqref{WinCK} that 
$W=(W_1,\dots,W_K)\in C_K$, $\hat \N_{\sss H}^s$-a.s. We in fact have (the proof is in Section~\ref{ssec:HBMnondeg}):
\begin{LEM} 
\label{lem:hBM_shape0}
\label{hBM_shape0} 
$W$  is non-degenerate $\hat\N_{\sss H}^s$-a.s.~and so $\hat\N_{\sss H}^s(\mc{B}_K(W)=\degK)=0$.  
\end{LEM}

For the rescaled BRW 
and lattice trees in Section~\ref{sec:enlargement} the paths $W^{\sn}=(W^{\sn}_i)_{i\le K}$ chosen according to $J^{\sn}$ will have jumps at lattice points $k/n$, $k\in\N$.  To get paths in $C_K$, as required in Definition~\ref{gstwdef}, we will need to introduce continuous interpolation operators 
$\kappa_n:\mc{D}\to\mc{C}$ satisfying $\kappa_n(w)_{k/n}=w_{k/n}$ and interpolated linearly on each $[k/n,(k+1)/n]$, for $k\in\Z_+$. We abuse notation and also let $\kappa_n:\mc{D}^K\to\mc{C}^K$ be the above map applied componentwise for any $K\in\N$. Therefore we have $\kappa_n(W^{\sn})\in C_K$.

\subsection{Main results}
\label{sec:main_results}

We can now state our first main result, which gives joint convergence of the (embedding of the) minimal subtree in $\mc{T}$ containing the root and $K$ uniformly chosen points.

\begin{THM}[Joint weak convergence] \label{thm:jointconvergence} 
For 
\begin{itemize}
\item[\emph{(i)}] branching random walk in dimension $d\ge 1$ and any $L\ge 1$; and 
\item[\emph{(ii)}] lattice trees in dimensions $d>8$, for all $L$ sufficiently large (depending on $d$), 
\end{itemize}
if $s>0$, $K \in \N$, $W^{\sn}=(W_1^{\sn},\dots,W_K^{\sn})$ and $W=(W_1,\dots,W_K)$, then the following hold as $n\to\infty$: 
\begin{equation}\label{rnondeglargen}\hat\P^s_n\big(\kappa_n(W^{\sn})\text{ non-degenerate}\big)\to 1,
\end{equation} and 
\begin{align}\label{wconva}&\hat\P^s_n\Big(\big(I^{\sn},J^{\sn},S^{\sn},W^{\sn},\bs{\tau}(W^{\sn}),\mc{B}_K(\kappa_n(W^{\sn})\big)\in\cdot\Big)\\
\nonumber&\cweak \hat\N^{s,\sigma_0^2}_{\sss H} \Big(\big(I,J,S,W,\bs{\tau}(W),\mc{B}_K(W)\big)\in\cdot\Big)\text{ in } \M_F(\D)^2\times\R_+\times\D^K\times\R_+^{(K^2)}\times \TT_{gst}. 
\end{align}
\end{THM}

We next use the scaling properties of $H^{\sn}$ for branching random walk and lattice trees (\eqref{rescaleHBRW} and \eqref{rescaleHLT}, respectively)  to reinterpret the above convergence. For $w\in C_K$, let
\begin{equation}\label{BnKdef}\mc{B}_{n,K}(w)=(T(\bs{\tau}(w)),d_{\bs{\tau}(w)}/n,\phi_{w}/\sqrt n\,)\in\TT_{gst},
\end{equation}
if $w$ is non-degenerate, and 
let it be 
$\degK\in\TT_{gst}$ otherwise.
Define $\bar \rho_n:\overline{\mc{D}}\to\oD$ by $\bar\rho_n(u,w)=(u/n,\rho_n(w))$. Recall that 
$\mu\circ f^{-1}$ denotes the pushforward of 
%a measure 
$\mu$ by a measurable map $f$.

\begin{THM}[Rescaled Joint weak convergence] \label{thm:rjointconvergence} 
For 
\begin{itemize}
\item[\emph{(i)}] branching random walk in dimension $d\ge 1$ and any $L\ge 1$; and 
\item[\emph{(ii)}] lattice trees in dimensions $d>8$, for all $L$ sufficiently large (depending on $d$), 
\end{itemize}
if $s>0$, $K \in \N$, $W^{\sss(1)}=(W_1^{\sss(1)},\dots,W_K^{\sss(1)})$ and $W=(W_1,\dots,W_K)$, then the following hold as $n\to\infty$: 
\begin{equation}\label{nondeglargen}\hat\P^{sn}_1\Bigl(\kappa_1(W^{\sss(1)})\text{ non-degenerate}\Bigr)\to 1,
\end{equation} and
\begin{align}\label{wconvb}&\hat\P^{sn}_1\Big(\Big(\frac{I^{\sss(1)}\circ\bar\rho_n^{\ -1}}{n^2},J^{\sss(1)}\circ\bar\rho_n^{\ -1},\frac{S^{\sss(1)}}{n},\rho_n(W^{\sss(1)}),\,\frac{\bs{\tau}(\kappa_1(W^{\sss(1)}))}{n},\mc{B}_{n,K}(\kappa_1(W^{\sss(1)}))\Bigr)\in\cdot\Bigr)\\
\nonumber&\cweak \hat\N^{s,\sigma_0^2}_{\sss H} \Big(\big(I,J,S,W,\bs{\tau}(W),\mc{B}_K(W)\big)\in\cdot\Big)\text{ in } \M_F(\D)^2\times\R_+\times\D^K\times\R_+^{(K^2)}\times \TT_{gst}. 
\end{align}
\end{THM}

\begin{REM}\label{rem:limitlaw}(The Limit Law)
The convergence result in Theorems~\ref{thm:jointconvergence}(ii) and \ref{thm:rjointconvergence}(ii) is not expected to hold for lattice trees in dimensions $d<8$,  while for $d=8$ one at least expects the same limit with logarithmic corrections in the scaling.   The limiting historical Brownian motion is a tree in dimensions $d\ge 8$, by the lack of super-Brownian double points in these dimensions (Theorem~1.6(a) of \cite{DIP89}).   On the other hand, the skeleton $(W_1,\dots,W_K)$ connecting the root to $K$ uniformly chosen points in the historical BM is almost surely a tree in dimensions $d\ge 4$. Indeed, these paths will have disjoint ranges once they separate by the lack of double points in the Brownian path for $d\ge 4$ (see, e.g., Theorem~9.1(a) in \cite{MP}) and the absolute continuity result \eqref{Wabscont} below. The latter is potentially useful to deduce other qualitative properties of
the weak limit $(W,\bs{\tau},\mc{B}_K(W))$. 

In \cite{A91} (see (13))  one can find a simple description of the joint law of the tree shape  and edge lengths of the tree generated by the $K$ randomly chosen particles conditional on the total mass of $I$ being one.  Conditioning on $S>s$ does not seem to lead to such explicit formulae.  Nonetheless, a simple consequence of f.d.d.~historical convergence is that for lattice trees for $d>8$ and $L$ large, the unique path in the tree from the root to a uniformly chosen vertex in $\mc{T}_{\floor{nt}}$ (conditional on $\mc{T}_{\floor{nt}}$ being non-empty) converges to a Brownian motion on $[0,t]$ \cite[Theorem 1.3]{H16}.  By comparison, in Theorem~\ref{thm:jointconvergence}(ii), we are sampling $K$ points from the entire tree, conditional on survival beyond time $s$.  When $K=1$, it is still easy to show that for the historical limit, the law of the randomly
chosen path $W$ given $\FL(W)=t$  is a Brownian motion stopped at time $t$. Extensions of this to $K>1$ are also possible (see \cite{HHP24})  
(This is easy to understand from the fact that for the approximating BRW, the branching and spatial motion variables are independent).  Nonetheless, even the law of $\FL(W)$ for $K=1$ under $\N^s_K$ is a bit complicated.    It is not hard to show that for $t>1$, $\hat{\N}^{1}_{\sss H}(\FL(W)>t)=1/(2t)$, whereas for $t<1$ an explicit non-trivial series expansion for its distribution function is found in \cite{HHP24}. 
Although the explicit law of $\bs{\tau}$ seems a bit complicated, the absolute continuity \eqref{Wabscont} allows us to derive simple quantitative properties of $W$ such as the tree property above.  A special case of it will also be used to derive the non-degeneracy of $\bs{\tau}$ for general $K$ in Lemma~\ref{lem:hBM_shape0}.
\Enddef
\end{REM}

\medskip

Theorem~\ref{thm:rjointconvergence} is perhaps a bit more transparent than the version without scaling, but we will prove a more general version of Theorem~\ref{thm:jointconvergence} (Theorem~\ref{thm:Gconvergence} below) in which no scaling property is assumed, and so it was natural to first state the particular application in this form. Theorem~\ref{thm:Gconvergence} only assumes a pair of abstract conditions, Conditions~\ref{cond:fdd} and \ref{cond:tau_stuff}, on a general sequence of historical processes, $H^{\sn}$, under which the conclusions of Theorem~\ref{thm:jointconvergence} hold.  The main part of Condition~\ref{cond:fdd} asserts weak convergence of the finite-dimensional distributions of $H^{\sn}$ to those of historical Brownian motion (historical convergence). Convergence of $H^{\sn}$ to $H$ does not in general imply the convergence of $\tau_{i,j}(W^{\sn})$ to $\tau_{i,j}(W)$, and   Condition~\ref{cond:tau_stuff} gives additional conditions on $H^{\sn}$ to insure weak convergence of $\bs{\tau}(W^{\sn})$ to $\bs{\tau}(W)$. Condition~\ref{cond:fdd} is verified for branching random walk and lattice trees using known results, most notably the weak convergence of rescaled historical lattice trees to historical Brownian motion in \cite{CFHP20}.  Theorem~\ref{thm:jointconvergence} is derived from Theorem~\ref{thm:Gconvergence} by verifying Condition~\ref{cond:tau_stuff} for lattice trees (our main model of interest) and branching random walk, in Sections~\ref{sec:cond18lt} and \ref{sec:cond18brw}, respectively.   We also give a more general version of Theorem~\ref{thm:rjointconvergence} under the two conditions noted above and the rescaling relation \eqref{rescaleHLT} (see Theorem~\ref{thm:Grjointconvergence} below). 

 \begin{REM}\label{rem:Genmodels}We believe the general results, Theorems~\ref{thm:Gconvergence} and \ref{thm:Grjointconvergence}, could also apply (with suitable changes in the continuous time setting) to the general class of models from \cite{HP19}, including oriented percolation, the contact process and the voter model.  The f.d.d. historical convergence underlying Condition~\ref{cond:fdd}  was  proved for lattice trees in \cite{CFHP20}, even for convergence on path space (this stronger result is not needed here), and 
 f.d.d.~historical convergence was recently proved  for the voter model in \cite{Banova},  using similar methods.  We also believe that the second condition, Condition~\ref{cond:tau_stuff} (implying convergence of branch times), will hold for all of the above models. \Enddef
 \end{REM}

\begin{EXA}[Lattice trees]\label{latrconv}
Let us examine more closely the gst component  
arising in Theorem~\ref{thm:rjointconvergence} in the lattice tree setting of Section~\ref{sec:LT}. 
We set aside the simple scaling factors in \eqref{BnKdef} ($n$ and $\sqrt n$),  and recall that, given $J^{\sss(1)}$, $W^{\sss(1)}_1,\dots,W^{\sss(1)}_K\in\D$ are i.i.d.~paths in $\mc{T}$, chosen according to $J^{\sss(1)}$. Let  $W=\kappa_1(W^{\sss(1)})\in C_K$ and $T=T(\bs{\tau}(W))$.  
We will see that it is equal to the natural gst embedded in $\Z^d$ arising from the tree generated by $K$ randomly chosen points from the random lattice tree $\mc{T}$. First we recap the   construction of $W$.  Condition $\mc{T}$ to survive beyond generation $ns$ (that is, choose it according to $\P^{ns}_1$).  Choose points $V_i,i\le K$ independently and uniformly from the vertices in $\mc{T}$ and let $k_i\in\Z_+$ be the graph distance of $V_i$ from the root $o$ of $\mc{T}$. Let $W_i(j),\ j=0,1,\dots,k_i$, be the unique path of vertices in $\mc{T}$ from  $o$ to $V_i$ and set $W_i(j)=V_i$ for $j>k_i$. Extend $W_i$ to $[0,\infty)$ by linear interpolation. 
We see from \eqref{Jnunif} that this does give a description of the construction of $W$.  By Theorem~\ref{thm:rjointconvergence}, $\hat \P^{sn}_1(\bs{\tau}(W)\text{ non-degenerate})\to 1$, and so  we may assume $\bs{\tau}(W)$ is non-degenerate, and thus have a gst
\[\mc{B}_{K}(W)=(T,d_{\bs{\tau}(W)},\phi_W)\]
with labelled leaves $i=\lb i,i\rb$, $i=1,\dots,K$, root $0=\lb 0,0\rb$, vertices $\aij$, and edges $e_{\aij}$.

Now let $\mc{T}^{\vec V}$ denote the subtree of $\mc{T}$ with vertices $\{W_i(m):m=0,\dots,k_i, i\in[K]\}$ in $\Z^d$, edges between ``consecutive" vertices $W_i(m-1)$ and $W_i(m)$ ($1\le m\le k_i$), root $o$ (also a leaf), and labelled leaves $V_i=W_i(k_i)$, $i=1,\dots,K$. Equip $\mc{T}^{\vec V}$ with the graph metric $d_{\mc{T}}$. Then the leaves $0,V_1,\dots,V_K$ have degree $1$ in $\mc{T}^{\vec V}$, the branch points $\{W_i(\tau_{i,j} ):i,j\in[K]\}$ have degree $3$ by non-degeneracy, and the remaining vertices have degree $2$. Let $T_{\vec V}$ denote the subtree of $\mc{T}^{\vec V}$ obtained by erasing each degree $2$ vertex  (in any order), and each time collapsing the two edges into a single edge.   
Then $T_{\vec V}$ has vertex set $\{W_i(\tau_{i,j}):i,j\in K\}\cup\{o\}$, also denoted $T_{\vec V}$ (note that $\tau_{i,i}=k_i$ and so $W_i(\tau_{i,i})=V_i$). Thus in $T_{\vec V}$  there is an edge  between $W_i(\tau_{i,m})$ and $W_i(\tau_{i,j})$ iff there are no times $\tau_{i,\ell}\in(\tau_{i,m},\tau_{i,j})$, which means the only edges are  between $W_i(\tau_{i,m(i,j)})$ and $W_i(\tau_{i,j})$ for all $i,j$.  
By its definition, $\phi_W(\aij)=W_i(\tau_{i,j})$ 
 so $\phi_W:T\to T_{\vec{V}}$ is onto. If $W_i(\tau_{i,j})=W_{i'}(\tau_{i',j'})$, then by the tree property of $\mc{T}$ we must have $\tau_{i,j}=\tau_{i',j'}$ and $W_i\vert_{\tau_{i,j}}=W_{i'}\vert_{\tau_{i,j}}$. This implies $\aij=\lb i',j'\rb$ by \eqref{eqclass3}. This shows $\phi_W$ is a root-preserving bijection between the vertex sets of $T$ and $T_{\vec{V}}$ which maps the labelled leaves $1,\dots,K$ to the respective labelled leaves $V_1,\dots,V_K$ of $T_{\vec{V}}$. The above description of edges in $T_{\vec{V}}$ shows it also is a graph isomorphism (recall \eqref{pidefn}).  We give $T_{\vec V}$ the distance $d$ it inherits from $d_{\mc{T}}$. So the length of the edge between $W_i(\tau_{i,j})$ 
and its parent, $W_i(\tau_{i,m(i,j)})$, is $\tau_{i,j}-\tau_{i,m(i,j)}=d_{\bs{\tau}}(\aij,\lb i,m(i,j)\rb)$.  The preservation of edge lengths implies that 
\begin{equation}\label{phiW}
\text{$\phi_W$ is an isometric isomorphism from $(T,d_{\bs{\tau}(W)})$ to $(T_{\vec V},d_{\mc{T}})$.}
\end{equation} 
$\overline{T_{\vec{V}}}$ has a line segment of length $\tau_{i,j}-\tau_{i,m(i,j)}$ inserted along edge $e_{\aij'}$, while $\overline{\mc{T}^{\vec V}}$ has $\tau_{i,j}-\tau_{i,m(i,j)}$ edges of length one (from $W_i(\ell-1)$ to $W_i(\ell)$, $\ell=\tau_{i,m(i,j)},\dots, \tau_{i,j}-1$) inserted between the same endpoints. So we can define an isometry from $(\overline{T_{\vec{V}}},d)$ to $(\overline{\mc{T}^{\vec V}},d_{\mc{T}})$ which is the identity on $T_{\vec{V}}$.  This allows us to identify these two metric spaces and it will be convenient to work with $\overline{\mc{T}^{\vec V}}$ in what follows. Let $id_{\mc{T}^{\vec V}}:\overline{\mc{T}^{\vec V}}\to\R^d$ be the identity map on $\mc{T}^{\vec V}$, extended linearly along the added unit line segments in $\overline{\mc{T}^{\vec V}}$ to the image line segments in $\R^d$, giving us 
the gst:
\[(T_{\vec V},d,id_{\mc{T}^{\vec V}}).\]
Recall the definition of the isometric extension  
of $\phi_W$, $\overline{\phi_W}:(\overline{T},d_{\bs{\tau}})\to (\overline{\mc{T}^{\vec V}},d)$, from Remark~\ref{idgsts}. Note that for $\tau_{i,m(i,j)}\le m\le \tau_{i,j}$, $\overline{\phi_W}([i,m])$ is the point in $\mc{T}^{\vec V}$ which is graph distance $m-\tau_{i,m(i,j)}$ from $W_i(\tau_{i,m(i,j)})$, and hence is $W_i(m)$. So we conclude
 that $id_{\mc{T}^{\vec V}}\circ \overline{\phi_W}([i,m])=W_i(m)$ for all $m\le k_i$.  For $u\in(m-1,m)$ both sides are extended linearly in $\R^d$ and we have proved that
 \[ id_{\mc{T}^{\vec V}}\circ\overline{\phi_W}([i,u])=W_i(u)=\phi_W([i,u])\text{ on }\overline{T}.\]
From this, \eqref{phiW}, and Remark~\ref{idgsts} we conclude that (recall $d$ is just the restriction of $d_{\mc{T}}$ to $T_{\vec V}$)
\begin{equation}\label{ltgst}\mc{B}_K(W):=(T,d_{\bs{\tau}(W)},\phi_W)=(T_{\vec V},d_{\mc{T}},id_{\mc{T}^{\vec V}})\ \text{ in }\TT_{gst},\text{ where }\vec V=(W_i(k_i))_{i\le K}.
\end{equation}
Thus, if we prefer, in the lattice tree setting of Theorem~\ref{thm:rjointconvergence} we may replace $\mc{B}_{n,K}(\kappa_1(W^{\sss (1)}))$ in \eqref{wconvb} with $(T_{\vec V}, d_{\mc{T}}/n,id_{\mc{T}^{\vec V}}/\sqrt n),$ where $\vec{V}$ is as in \eqref{ltgst}. 

Note that for other models, such as branching random walk, this alternative construction of $\mc{B}_K(W)$ will not be available because $\phi_W$ will not be injective.  See Example~\ref{BRWrem} for a  reinterpretation of the  weak convergence result in the setting of BRW.\Enddef
\end{EXA}

\begin{REM}[Condition G] \label{condg}
We may use the continuity theorem to see that in the lattice tree setting,  Theorem~\ref{thm:rjointconvergence} implies 
\[\hat \P^{sn}_1\Bigl(\Bigl(\frac{I^{\sss(1)}(1)}{n^2},\mc{B}_{n,K}(\kappa_1(W^{\sss(1)})\Bigr)\in\cdot\Bigr)\cweak \hat\N^{s,\sigma_0^2}_{\sss H} \Big((I^{\sss(1)}(1), \mc{B}_K(W))\in\cdot\Bigr)\ \ \text{ in }\R_+\times\TT_{gst}.\]
A simple calculation gives $I^{\sss(1)}(1)=(|E(\mc{T})|+1)/C_0$ where $|E(\mc{T})|$ is the number of edges in the random tree $\mc{T}$. Therefore we have
\begin{align*} \hat \P^{sn}_1\Bigl(\Bigl(\frac{|E(\mc{T})|}{C_0n^2},(T(\kappa_1(W^{\sss(1)})),d_{\bs{\tau}(\kappa_1(W^{\sss (1)}))}&/n,\phi_{\kappa_1(W^{\sss(1)})}/\sqrt n\,)\Bigr)\in\cdot\Bigr)\\
&\cweak \hat\N^{s,\sigma_0^2}_{\sss H} \Big((I^{\sss(1)}(1), \mc{B}_K(W))\in\cdot\Bigr)\ \ \text{ in }\R_+\times\TT_{gst}.
\end{align*}
This is Condition $(G)^{s,+}_{1,\sigma_0,C_0}$ in \cite{B-ACF18} (see the previous Example) which was required therein to derive convergence of rescaled RW on $\mc{T}$ to Brownian motion on SBM  (see \cite{Croydon09} for the latter). \Enddef
\end{REM}

We  continue to focus on lattice trees.
For $K\in \N$ recall from Example~\ref{latrconv} that $\mc{T}^{\vec V}$ is the minimal subtree of the lattice tree $\mc{T}$ consisting of the vertices and edges along the paths 
 joining the origin to $K$ uniformly chosen vertices $V_1,\dots, V_K$ from $\mc{T}$.   Here $\vec V=(V_1,\dots,V_K)$. 
For $x\in \mc{T}$, let  $\pi_{K}(x)$ denote the point in $\mc{T}^{\vec V}$ which is closest to $x$ in the graph metric, $d_{\mc{T}}$, of $\mathcal{T}$.  One can readily check that 
\begin{equation}\label{piKanc}
\text{for any $x\in\mc{T}$, $\pi_K(x)$ is the most recent ancestor of $x$ in $\mc{T}^{\vec V}$.}
\end{equation} 
Here is our second main result.  $B(v,r)$ denotes the open Euclidean ball about $v$  of radius $r$.
\begin{THM}[Approximation by the subtree]
\label{thm:main2}
For  
lattice trees in dimensions $d>8$, for all $L$ sufficiently large (depending on $d$) 
the following holds: 

\noindent For any 
any $\vep,s>0$,
\begin{itemize}
\item[\emph{(a) }]    
$\displaystyle{\limsup_{K\to \infty}\sup_{n \in \N}\hat{\P}_1^{ns}\Big(\mc{T}
\not\subset \cup_{i=1}^KB(V_i,\vep\sqrt{n})\Big)=0, \quad \text{ and}}$
\item[\emph{(b) }] 
$\displaystyle{\limsup_{K\to\infty}\sup_{n\in\N}\hat{\P}_1^{ns}\big(\max_{x\in\mc{T}}\, d_{\mc{T}}(x,\pi_K(x))/n>\vep\big)=0,\quad \text{ and}}$
\item[\emph{(c) }] 
$\displaystyle{\limsup_{K\to\infty}\sup_{n\in\N}\hat{\P}_1^{ns}\big(\max_{x\in\mc{T}} |x-\pi_K(x)|/\sqrt n>\vep\big)=0.}$
\end{itemize}

\end{THM}
Parts (b) and (c) of Theorem \ref{thm:main2} constitute the main part of  Condition $(S)$ of \cite{B-ACF18} in the context of lattice trees.  There $(S)$ was needed to establish the weak convergence of RW on lattice trees to Brownian motion on SBM.

\begin{REM}\label{rem:thm2}  A general class of lattice models with a notion of ancestry is introduced in \cite{HP19}.  This class includes lattice trees, critical oriented percolation, critical contact processes and voter models. The lattice tree $\mc{T}$ is replaced by an abstract space-time graph in $\Z_+\times\Z^d$. Under Conditions 1--7 of that reference, Theorem~2 of \cite{HP19} proves weak convergence of the rescaled ranges to the range of SBM in the Hausdorff topology. 
For the discrete time setting ($I=\Z_+$ in the notation of \cite{HP19}) these conditions include sufficiently spread out lattice trees for $d>8$ and sufficiently spread out oriented percolation for $d>4$.  Using this weak convergence of the ranges and the compactness of the range for SBM one can easily derive Theorem~\ref{thm:main2}(a) under Conditions 1--7 of \cite{HP19} in the discrete time setting,  and hence for the spread-out lattice trees ($d>8$) in the statement of (a) above.  Although (a) is similar in spirit to parts (b) and (c), we omit its simple proof because the result is not used in the proofs of (b) and (c), and it is the latter results that are used in \cite{B-ACF18}. 
The proof of (b) can also be adapted to  the discrete time setting of \cite{HP19} under Conditions~1--7 of \cite{HP19}, where $d_{\mc{T}}$ becomes the graph metric associated with the aforementioned space-time graph in $\Z_+\times\Z^d$. In particular (b) also holds for sufficiently spread out oriented percolation with $d>4$.   We   only prove  it in the setting of lattice trees, as stated above, where the argument simplifies. 
Whether or not (c) remains valid under these general Conditions remains unresolved. \Enddef
\end{REM}

Theorem \ref{thm:main2}~(b,c) is proved in 
Section~\ref{sec:S}.  The proof of Theorem~\ref{thm:main2} does not use historical processes per se, although a notion of ancestry plays a key role.   

\section{Some Elementary Results on the Graph Spatial Tree of $w$}\label{ssec:gsts}

We start by proving a number of  results stated in Section~\ref{sec:GST} without proof. Throughout this section $w\in C_K$ and $\bs{\tau}=\bs{\tau}(w)$ is non-degenerate. Recall the definition of $m(i,j)$ from \eqref{mdef}.
\medskip
\begin{LEM}\label{parentwd} \,
\begin{itemize}
\item[\emph{(a)}] Let $i,j,i',j'\in[K]$. If $\aij=\langle i',j'\rangle$, then $m(i,j)=m(i',j')$, $\tau_{i,m(i,j)}=\tau_{i',m(i',j')}$, and $\langle i,m(i,j)\rb=\langle i',m(i',j')\rb$.
\item[\emph{(b)}] $\pi:T(\bs{\tau})\setminus \{0\}\ra T(\bs{\tau})\setminus\{1,\dots,K\}$.
\end{itemize}
\end{LEM}
\begin{proof} (a) Assume $\lb i,j\rb=\lb i',j'\rb$ for $i,j,i',j'\in[K]$.  By \eqref{eqclass2} we have
\[ \forall k\in[K],\quad \ \ \tau_{i,k}\ge \tau_{i,j}\text{ iff }\tau_{i',k}\ge \tau_{i',j'}=\tau_{i,j},\]
or equivalently,
\begin{equation}\label{complement}
\forall k\in[K]\cup\{0\}, \quad \ \tau_{i,k}< \tau_{i,j}\text{ iff }\tau_{i',k}< \tau_{i',j'}=\tau_{i,j}.
\end{equation}
Note also that by \eqref{eqclass3} we have $\tau_{i,i'}\ge \tau_{i,j}$.  This and \eqref{complement} show that $\tau_{i,k}<\tau_{i,j}$ implies $\tau_{i',k}=\tau_{i,k}$ since $w_i\vert_{\tau_{i,j}}=w_{i'}\vert_{\tau_{i,j}}$.  
The definition of $m(i,j)$ in \eqref{mdef} and the above imply that $m(i,j)=m(i',j')$ and $\tau_{i,m(i,j)}=\tau_{i',m(i',j')}$.  
By $\lb i,j\rb=\lb i',j'\rb$, an application of \eqref{eqclass3} gives $\tau_{i,i'}\ge \tau_{i,j}>\tau_{i,m(i,j)}=\tau_{i',m(i',j')}$, where the strict inequality is by definition of $m$. This last conclusion implies that $\lb i, m(i,j)\rb=\lb i',m(i',j')\rb$ by another application of \eqref{eqclass3}.\\
(b) 
If $\pi(\lb i, j\rb)=k\in\{1,\dots,K\}$, 
then since $\lb k,k\rb$ is a singleton (by \eqref{leafnot}),  $i=k=m(i,j)$ which contradicts $\tau_{i,m(i,j)}<\tau_{i,j}\le\tau_{i,i}$ (by definition).  This gives (b).
\end{proof}

\begin{LEM}\label{treegraph} $T(\bs{\tau})$ is a tree graph.
\end{LEM}
\begin{proof} If $\aij\in T$ by taking iterates of $\pi$ we obtain a sequence of neighbouring edges starting with $e_{\aij}$ and following a sequence of vertices $\lb i, j_n\rb$ with $\tau_{i,j_n}$ decreasing at least by a minimal amount $\delta>0$ at each step. This must end at $0$ after a finite number of steps, proving connectivity of $T$. The definition of the edge set in $T(\bs{\tau})$ implies that 
\begin{align}\nonumber&\text{for any non-root vertex $\aij$, $\pi(\aij)$}\text{ is the only neighbour, $\lb i',j'\rb$, of $\aij$ s.t. $\tau_{i',j'}\le \tau_{i,j}$,}\\
\label{nocyc}&\text{ and it satisfies $\tau_{i',j'}< \tau_{i,j}$, and the root $\lb 0,0\rb$ has no neighbour s.t. $\tau_{i',j'}\le \tau_{0,0}$}.
\end{align}
Suppose that there is a cycle.  Choose a vertex in the cycle, $\aij$, so that  $\tau_{\aij}$ is maximal over all  vertices in the cycle.  Then it has two neighbours $\lb i',j'\rb$ with  $\tau_{\lb i',j'\rb}\le \tau_{\aij}$, contradicting the above. Hence $T$ contains no cycles and is a tree. 
\end{proof}
  
Define a partial order, $\le$, on $\overline T=\overline T(\bs{\tau})$ by 
\[[i,u]\le [j,v]\text{ iff }u\le v \text{ and }u\le \tau_{i,j},\text{ that is, iff }u\le v\text{ and }w_i\vert_u=w_j\vert_u.\]
It is easy to check that $\le$ is a well-defined partial order.  We also write $\le$ for the restriction of this partial order to $T=T(\bs{\tau})$, that is, for $i,j,i',j'\in \{0,1,\dots,K\}$, 
\begin{equation}\label{ledef}\lb i',j'\rb\le \aij\text{ iff }[i',\tau_{i',j'}]\le [i,\tau_{i,j}]\text{ iff } \tau_{i',j'}\le \min(\tau_{i,j},\tau_{i,i'})\text{ iff }\tau_{i',j'}\le \tau_{i,j}\text{ and }w_{i'}\vert_{\tau_{i',j'}}=w_{i}\vert_{\tau_{i',j'}}.
\end{equation}
The relation $<$ will mean $\le$ but not equal to. 
By \eqref{ledef} and \eqref{eqclass3} for $i,j,i',j'\in \{0,1,\dots,K\}$, 
\begin{equation}\label{strictle}\apij<\aij\text{ iff }\tau_{i',j'}<\tau_{i,j}\text{ and }\tau_{i',j'}\le\tau_{i,i'}.
\end{equation}
We also introduce the ancestral ordering $\prec$ on the tree $T$, that is $\lb i',j'\rb\prec\aij$ iff $\lb i',j'\rb=\pi^k(\aij)$ for some $k\in\N$, where $\pi^k$ is the $k$-fold composition of $\pi$.  

\medskip
\begin{LEM}\label{orderprops}\,
\begin{itemize}
\item[\emph{(a)}] If $\lb i',j'\rb\le\aij$, then $\apij=\lb i,i'\rb$ or $\lb i,j'\rb$.
\item[\emph{(b)}] The orders $<$ and $\prec$ agree on $T$.
\end{itemize}
\end{LEM}
\begin{proof}(a) By assumption, \eqref{ledef} and symmetry (interchange $i'$ and $j'$) we have 
\begin{equation}\label{ineq1}\tau_{i',j'}\le\tau_{i,j}\text{ and }\tau_{i',j'}\le \min(\tau_{i,i'},\tau_{i,j'}).
\end{equation}
Assume $\tau_{i,i'}\le \tau_{i,j'}$.  Then by the last inequality in \eqref{tauprops}, $\tau_{i,i'}=\min(\tau_{i,i'},\tau_{i,j'})\le \tau_{i',j'}$. We also have $\tau_{i,i'}\ge \tau_{i',j'}$ by \eqref{ineq1} and hence $\tau_{i,i'}=\tau_{i',j'}$.  It follows from the definitions now that $\apij=\lb i',i\rb=\lb i,i'\rb$.  If $\tau_{i,i'}\ge \tau_{i,j'}$, analogous reasoning (interchange $i'$ and $j'$) leads to $\apij=\lb i,j'\rb$. 

\noindent (b) To show that $v_1\prec v_2$ implies $v_1<v_2$, note first by transitivity of $<$, it suffices to show $\pi(v_2)<v_2$, which follows from \eqref{eqclass3}--\eqref{pidefn}.  Assume now that $\apij<\aij$.  By (a), we have $\apij=\lb i,k\rb$ for some $k$. We have $\tau_{i,k}<\tau_{i,j}$ which implies by definition of $m(i,j)$ that $\tau_{i,k}\le \tau_{i,m(i,j)}$.  This in turn easily implies $\apij=\lb i,k\rb\le\lb i,m(i,j)\rb=\pi(\aij)$.  If equality holds here we are done since then $\apij\prec\aij$.  If not, continue this process until we arrive at $\apij=\pi^m(\aij)$ (it must stop after finitely many iterations) and conclude $\apij\prec\aij$.
\end{proof}
Henceforth we will write  $<$ for the strict ancestral ordering and drop the notation $\prec$.  
Recall that if $v_1,v_2$ are vertices in the tree $T$, $v_1\wedge v_2$ denotes their greatest common ancestor with respect to $\le$ on $T$, and $[[v_1,v_2]]$ is the unique path of non-overlapping edges from $v_1$ to $v_2$. 
\medskip

\begin{LEM}\label{obvfact} For any $i,j,k,\ell\in[K]$, if $\min(\tau_{i,k},\tau_{j,\ell})>\tau_{i,j}$, then $\tau_{k,\ell}=\tau_{i,j}$.
\end{LEM}
\begin{proof} The hypotheses imply that for some $\delta>0$ we have $w_i\vert_{\tau_{i,j}+\delta}=w_k\vert_{\tau_{i,j}+\delta}$ and $w_j\vert_{\tau_{i,j}+\delta}=w_\ell\vert_{\tau_{i,j}+\delta}$. It follows that $w_k$ and $w_\ell$ separate at the first time at $\tau_{i,j}$, whence the result.
\end{proof}

\begin{LEM}\label{ancestry}\,
\begin{itemize}
\item[\emph{(a)}] If $\tau_{i,k}<\min(\tau_{i,j},\tau_{k,l})$, then $\aij\wedge \lb k,\ell\rb=\lb i, k\rb$.
\item[\emph{(b)}] Assume $\tau_{i,k}\ge \min(\tau_{i,j},\tau_{k,\ell})$.
\begin{itemize}
\item[\emph{(i)}] If $\tau_{i,j}\le \tau_{k,\ell}$, then $\aij\wedge \lb k,\ell\rb=\aij$.
\item[\emph{(ii)}] If $\tau_{i,j}\ge \tau_{k,\ell}$, then $\aij\wedge \lb k,\ell\rb=\lb k,\ell\rb$.
\end{itemize}
\end{itemize}
\end{LEM}
\begin{proof} (a) Apply \eqref{strictle} with $i=i'$ to see that $\tau_{i,k}<\min(\tau_{i,j},\tau_{k,l})$ implies
$\lb i,k\rb<\aij$ and $\lb i,k\rb<\lb k,\ell\rb$. Assume that $m,n\in[K]$  satisfy $\lb m,n\rb\le \aij$ and $\lb m,n\rb\le \lb k,\ell\rb$.   Then by \eqref{ledef}, 
$\tau_{m,n}\le \min(\tau_{m,i},\tau_{m,k})\le \tau_{i,k}$ (the last by \eqref{tauprops}).  In addition by \eqref{ledef}, $w_i\vert_{\tau_{m,n}}=w_m\vert_{\tau_{m,n}}=w_k\vert_{\tau_{m,n}}$. The last two results imply (use \eqref{ledef}) $\lb m,n\rb\le \lb i,k\rb$, and this is trivial if $m=0$ or $n=0$.  From the above results we conclude that $\aij\wedge \lb k,\ell\rb=\lb i, k\rb$.

\noindent(b)(i) We have $\tau_{i,j}\le \min(\tau_{i,k},\tau_{k,\ell})$, which by \eqref{ledef} implies $\aij\le \lb k,\ell\rb$.\\
(ii) follows from (i) by interchanging $\aij$ and $\lb k,\ell\rb$.
\end{proof}
\begin{REM}\label{gcaij} It follows from (a) that $i\wedge k =\lb i,k\rb$. If $i=k$ this is trivial. For $i\neq k$, we have
$\tau_{i,k}<\min(\tau_{i,i},\tau_{k,k})$ by non-degeneracy.  We can therefore apply (a) with $j=i$ and $k=\ell$ to get the required equality.\Enddef
\end{REM}

\noindent{\bf Proof of Proposition~\ref{prop:wshape}.} By Lemma~\ref{treegraph}, $T$ is a finite tree graph and we designate $0$ as the root.  Suppose $i$  is the parent of $\lb k,\ell\rb$.  
By \eqref{strictle} we have $\tau_{i,i}<\tau_{k,\ell}$ and $\tau_{i,i}\le \tau_{i,k}$. The last inequality and the non-degeneracy of $\bs{\tau}$ imply that $i=k$. So the first inequality implies $\tau_{i,i}<\tau_{i,\ell}$, which is impossible.  Hence the labelled vertices $1,\dots,K$ have no children and are all leaves of degree $1$ (each has a neighbouring parent $\pi(i)$). 

Consider next the root $0$. By definition it has no parent.  Choose $(i_0,j_0)\in[K]^2$ such that $\tau_{i_0,j_0}=\min\{\tau_{i,j}:i,j\in[K]\}>0$. It follows by definition that $m(i_0,j_0)=0$ and so $0$ is the parent of $\lb i_0,j_0\rb$.  
Suppose now $0$ is the parent of $\lb i_1,j_1\rb$.  By the choice of $i_0,j_0$,  $\tau_{i_0,j_0}\le \min(\tau_{i_1,j_1},\tau_{i_0,i_1})$, and so by \eqref{ledef}, $\lb i_0,j_0\rb\le \lb i_1,j_1\rb$. Lemma~\ref{orderprops}(b) implies that for some $k\in\Z_+$, $\lb i_0,j_0\rb=\pi^k(\lb i_1,j_1\rb)$.  If $k>0$ then since $\pi(\lb i_1,j_1\rb)=0$ we get $\lb i_0,j_0\rb=\pi^{k-1}(0)$ which is impossible.   Thus $k=0$, whence $\lb i_0,j_0\rb=\lb i_1,j_1\rb$ and  $0$  
 has degree one.

Now consider $\aij$ for $i\neq j$ in $[K]$.  Define $c(i,j)\in [K]$ by 
$\tau_{i,c(i,j)}=\min\{\tau_{i,k}:\tau_{i,k}>\tau_{i,j}\}$ where we choose $c(i,j)$ minimal if there is more than one   $k$.
Note that the minimum is over  a 
non-empty set because it contains $i$ by the non-degeneracy of $\bs{\tau}$. 
 The definition of $\pi$ easily gives $\pi(\lb i,c(i,j)\rb)=\aij$. Interchanging $i$ and $j$ we also see that 
 $\pi(\lb j,c(j,i)\rb)=\aij$.  
 These two children of $\aij$, $\lb i,c(i,j)\rb$ and $\lb j,c(j,i)\rb$, are distinct since otherwise by \eqref{eqclass3} we have $\tau_{i,j}\ge \tau_{i,c(i,j)}=\tau_{j,c(j,i)}$, which contradicts $\tau_{i,c(i,j)}>\tau_{i,j}$ (from the definition of $c(i,j)$). 
  
Finally, assume that a parent has three distinct children $\lb i_r,j_r\rb$ (for $r=1,2,3$) in $T$.  Consider any pair of these siblings, say $\lb i_1,j_1\rb$ and $\lb i_2,j_2\rb$.  
Neither is a descendent of the other because they are distinct siblings, so by Lemma \ref{ancestry} (b) we must have $\tau_{i_1,i_2}<\min(\tau_{i_1,j_1},\tau_{i_2,j_2})$.  By  Lemma \ref{ancestry} (a) their most recent common ancestor is $\lb i_1,i_2\rb$.  Since the greatest common ancestor is  the common parent, we conclude that the parent is $\lb i_1,i_2\rb$, and so we also have $i_1\neq i_2$ by Lemma~\ref{parentwd}(b).  Applying this for each pair of siblings gives that the common parent is 
$\lb i_1,i_2\rb=\lb i_1,i_3\rb=\lb i_2,i_3\rb$, and $i_1,i_2$ and $i_3$ are distinct.  This shows that $\tau_{i_r,i_s}$ for  $r\ne s$ are all equal.  This violates the non-degeneracy condition \eqref{binaryc} and so the proof is complete.\qed

\medskip

\noindent{\bf Proof of Lemma~\ref{lem:samemet}.} If $v_1$ and $v_2$ are distinct vertices in the non-degenerate binary tree $T$, the path $[[v_1,v_2]]$ is the sequence of neighbouring edges going down from $v_1$ to $v_1\wedge v_2$, followed by the sequence of neighbouring edges going back up from $v_1\wedge v_2$ to $v_2$. Therefore (recall Definition~\ref{metricdef})
\begin{equation}\label{dtauform}
d_{\ell}(\aij,\lb k,\ell\rb)=\sum_{e_v\in[[\aij\wedge \lb k,\ell\rb,\aij]]} \ell(e_v )+\sum_{e_v\in[[\aij\wedge \lb k,\ell\rb,\lb k,\ell\rb]]} \ell(e_v).
\end{equation}
Here $[[v,v]]=\emptyset$ and so either of the above sums may be empty.\\
{\bf Case 1.} $\tau_{\lb i,k\rb}<\min(\tau_{\aij},\tau_{\lb k,\ell\rb})$.\\
By Lemma~\ref{ancestry}(a) and $\ell(e_v)=\tau_v-\tau_{\pi(v)}$, the right-hand side of \eqref{dtauform} telescopes to
give
\[d_{\ell}(\aij,\lb k,\ell\rb)=\tau_{\aij}+\tau_{\lb k,\ell\rb}-2\tau_{\lb i,k\rb}.\]
{\bf Case 2.} $\tau_{\lb i,k\rb}\ge\min(\tau_{\aij},\tau_{\lb k,\ell\rb})=\tau_{\aij}$.\\
By Lemma~\ref{ancestry}(b)(i), the first summand in \eqref{dtauform} vanishes and the second telescopes to give
\[d_{\ell}(\aij,\lb k,\ell\rb)=\tau_{\lb k,\ell\rb}-\tau_{\aij}=|\tau_{\lb k,\ell\rb}-\tau_{\aij}|.\]
{\bf Case 3.} $\tau_{\lb i,k\rb}\ge\min(\tau_{\aij},\tau_{\lb k,\ell\rb})=\tau_{\lb k,\ell\rb}$.\\
As in Case 2 we get
\[d_{\ell}(\aij,\lb k,\ell\rb)=\tau_{\aij}-\tau_{\lb k,\ell\rb}=|\tau_{\lb k,\ell\rb}-\tau_{\aij}|.\]
To summarize the three cases, we have
\begin{equation}\label{dellform1}
d_{\ell}(\aij,\lb k,\ell\rb)=\begin{cases} \tau_{\aij}+\tau_{\lb k,\ell\rb}-2\tau_{\lb i,k\rb}, & \text{ if }\min(\tau_{\aij},\tau_{\lb k,\ell\rb})>\tau_{i,k}\\
|\tau_{\lb k,\ell\rb}-\tau_{\aij}|, &\text{ otherwise}.
\end{cases}
\end{equation}
On the other hand, we have that
\[d_{\bs{\tau}}(\aij,\lb k,\ell\rb)=d_{\bs{\tau}}([i,\tau_{\aij}],[k,\tau_{\lb k,\ell \rb}]),\]
and the last expression equals the right-hand side of \eqref{dellform1} by \eqref{treemet}.
\qed

\medskip
Recall our standing assumption that $w\in C_K$.
\begin{LEM}\label{dtauconv}\,
\begin{itemize}
\item[\emph{(a)}] For any $[i,u],[i_n,u_n]\in\overline T(\bs{\tau})$, 
$$d_{\bs{\tau}}([i_n,u_n],[i,u])\to 0\text{ iff }u_n\to u\text{ and for large enough }n, [i_n,u]=[i,u].$$
\item[\emph{(b)}] If $\phi_w:\overline T\to\R^d$ is given by $\phi_w([i,u])=w_i(u)$, then $\phi_w$ is continuous.
\end{itemize}
\end{LEM}
 \begin{proof}  It is easy to check from the definition of $d_{\bs{\tau}}$ that 
\begin{equation}\label{dtaubnd}
\text{for any }[i,u],[j,v]\in \overline T,\ \ |u-v|\le d_{\bs{\tau}}([i,u],[j,v]).
\end{equation}
Let $[i_n,u_n]\to[i,u]$ in $\overline T$. By the above we have $u_n\to u$.  Let $i'$ be a limit point of $\{i_n\}$, so that we may choose a subsequence s.t. $i_{n_k}=i'$ for all $k$.  Therefore we have 
\begin{equation}\label{conv1}
[i',u_{n_k}]=[i_{n_k},u_{n_k}]\to [i,u].
\end{equation}
We claim that $u\le \tau_{i,i'}$. Suppose not, that is $u>\tau_{i,i'}$.  Then $u_{n_k}>\tau_{i,i'}$ for $k$ large since $u_n\to u$. Therefore we have from \eqref{conv1},
\begin{align*}
0=\lim_kd_{\overline T}([i',u_{n_k}],[i,u])=\lim_k u_{n_k}+u-2\tau_{i,i'}=2u-2\tau_{i,i'}.
\end{align*}
This implies $u=\tau_{i,i'}$, a contradiction.  We conclude that for any limit point $i'$ we have $u\le \tau_{i,i'}$, which implies by \eqref{eqrel} that $[i',u]=[i,u]$.  It follows that for $n$ large enough, $[i_n,u]=[i,u]$. 

Conversely assume $u_n\to u$ and for $n\ge N$, $[i_n,u]=[i,u]$. By \eqref{eqrel} we have for $n\ge N$, $u\le \tau_{i_n,i}$ which implies that $d_{\bs{\tau}}([i_n,u_n],[i,u])=|u_n-u|\to 0$.

(b) We have already noted in the Introduction that $\phi_w$ is well-defined.  If $[i_n,u_n]\to[i,u]$ in $\overline T$, part (a) implies for large enough $n$, 
 \begin{align*}|\phi_w([i_n,u_n])-\phi_w([i,u])|&=|\phi_w([i_n,u_n])-\phi_w([i_n,u])|\\
 &=|w_{i_n}(u_n)-w_{i_n}(u)|\le \max_{i}|w_{i}(u_n)-w_{i}(u)|\to 0\text{ as }n\to\infty.
 \end{align*} 
\end{proof}

\begin{LEM}\label{lem:comp}
Assume $\bs{\tau}$ and ${\bs{\tau}'}$ are non-degenerate in $R_K$ such that 
\begin{equation}\label{compcond}
\text{for any }i,j,k,\ell\in[K]\ \ \tau_{i,j}<\tau_{k,\ell}\Rightarrow \tau_{i,j}'<\tau_{k,\ell}'.
\end{equation}
Then $T(\bs{\tau})=T(\bs{\tau}')$ in the space $\Sigma_K$ of non-degenerate shapes.
\end{LEM}
\begin{proof} Let $T=T(\bs{\tau})$ and $T'=T(\bs{\tau}')$. We write $\aij'$ for points in $T'$, and write $\le'$ and $<'$ for the ancestral relations in $T'$. By non-degeneracy we may extend \eqref{compcond} to allow $i,j,k,\ell\in[K]\cup\{0\}$.  Take the contrapositive in \eqref{compcond} to get
\begin{equation}\label{contra}
\text{for any }i,j,k,\ell\in[K]\cup\{0\}\ \ \tau_{k,\ell}'\le\tau'_{i,j}\Rightarrow \tau_{k,\ell}\le\tau_{i,j}.
\end{equation}
The characterization \eqref{ledef}  and the above show that for all $i,j,k,\ell\in[K]\cup\{0\}$, 
\begin{equation}\label{leimp}
\lb k,\ell\rb'\le'\aij'\Rightarrow \tau'_{k,\ell}\le \min(\tau'_{i,j},\tau'_{i,k})\Rightarrow \tau_{k,\ell}\le\min(\tau_{i,j},\tau_{i,k})\Rightarrow\lb k,\ell\rb\le\aij.
\end{equation}
By the antisymmetry of the partial order $\le$ on $T$ we conclude that $\lb k,\ell\rb'=\aij'$ implies $\lb k, \ell\rb=\aij$ and so may define $f:T'\to T$ by $f(\aij')=\aij$. Clearly $f$ is onto, and since $T$ and $T'$ both have $2K$ vertices, $f$ is a bijection, and so for all $i,j,k,\ell\in[K]\cup\{0\}$,
\begin{equation}\label{fbij}
\lb k,\ell\rb'=\aij'\text{ iff }\lb k,\ell\rb=\aij.
\end{equation}
This and \eqref{leimp} show that for all $ i,j,k,\ell\in[K]\cup\{0\}$,
\begin{equation}\label{oneway}
\lb k,\ell\rb'<'\aij'\Rightarrow \lb k,\ell\rb<\aij. 
\end{equation}
Conversely assume $\lb k,\ell\rb<\aij$. 
By Lemma~\ref{orderprops}(a) $\lb k,\ell\rb=\lb i,k\rb$ or $\lb i,\ell\rb$, so assume without loss of generality that
\begin{equation}\label{gettoi}
 \lb k,\ell\rb=\lb i,k\rb.
\end{equation}
Therefore $\tau_{i,k}=\tau_{k,\ell}<\tau_{i,j}$, the latter by our assumption. This implies  $\tau'_{i,k}<\tau'_{i,j}$ by \eqref{compcond}.  An application of \eqref{strictle} with $i'=i$ and $j'=k$ gives $\lb i,k\rb' <'\aij'$.  	By \eqref{fbij} and \eqref{gettoi} we have $\lb k,\ell\rb'=\lb i,k\rb'$ and so from the above, $\lb k,\ell\rb'<'\aij'$.  So we have proved the converse implication to \eqref{oneway} and hence $f$ preserves ancestral ordering, that is,
\[f(\lb k,\ell\rb)<f(\aij)\text{ iff }\lb k,\ell\rb<\aij.\]
From this we readily see that $f$ maps the parent of $\aij'\in T'\setminus\{0\}$ to the parent of $\aij\in T\setminus\{0\}$, and so is a (root preserving) graph isomorphism which clearly preserves the labelling.
\end{proof}

Let $\Vert\cdot\Vert$ denote the Euclidean norm on the space $\R^{K\times K}$ of $K\times K$ real matrices.
\begin{PRP}\label{prop:taupert}
Let $\bs{\tau}\in R_K$ be non-degenerate. There is a $\delta>0$ so that if $\bs{\tau}'\in R_K$ satisfies $\Vert \bs{\tau}'-\bs{\tau}\Vert<\delta$, then
\begin{itemize}
\item[\emph{(i)}] $\bs{\tau}'$ is non-degenerate, \text{ and }
\item[\emph{(ii)}] $T(\bs{\tau})=T(\bs{\tau}')$ in $\Sigma_K$.
\end{itemize}
\end{PRP}
\begin{proof} The space of degenerate branching matrices, $\bs{\tau}$, is 
\[\cup_{i\neq j}\{\bs{\tau}:\tau_{i,j}=0\text{ or }\tau_{i,j}=\tau_{i,i}\}\bigcup\cup_{i,j,k\text{ distinct}}\{\bs{\tau}:\tau_{i,j}=\tau_{i,k}=\tau_{j,k}\}.\]
This is a finite union of closed sets, so there is a $\delta_1>0$ so that $\Vert\bs{\tau}'-\bs{\tau}\Vert<\delta_1$ implies $\bs{\tau}'$ is non-degenerate.
Let 
\[\delta_2=\frac{1}{3}\min\{\tau_{i,j}-\tau_{k,\ell}:\tau_{i,j}>\tau_{k,\ell}, i,j,k,\ell\in[K]\},\] 
and take $\delta=\min(\delta_1,\delta_2)$ which is strictly positive.  Assume $\bs{\tau}'\in R_K$ satisfies $\Vert\bs{\tau}'-\bs{\tau}\Vert<\delta$. Then $\bs{\tau}'$ is non-degenerate by the choice of $\delta_1$.  Assume $\tau_{i,j}>\tau_{k,\ell}$.  Then
\begin{align*}\tau_{i,j}'-\tau_{k,\ell}'\ge \tau_{i,j}-\tau_{k,l}-2\delta\ge 3\delta_2-2\delta>0.
\end{align*}
Lemma~\ref{lem:comp} implies $T(\bs{\tau}')=T(\bs{\tau})$ and the proof is complete.
\end{proof} 

$C_K$ is given the subspace topology it inherits from $\mc{C}^K$, so $w^{\sn}\to w$ iff $w^{\sn}_i\to w_i$ uniformly on compacts for each $i\le K$.

\begin{PRP}\label{shapecont} Assume $w\in C_K$ is non-degenerate, $w^{\sn}\to w$ in $C_K$, and $\bs{\tau}(w^{\sn})\to \bs{\tau}(w)$ in $\R^{K\times K}$.  Then for $n$ large $\bs{\tau}(w^{\sn})$ is non-degenerate, and $\mc{B}_K(w^{\sn})\to\mc{B}_K(w)$ in $\TT_{gst}$ as $n\to\infty$.
\end{PRP}
\begin{proof} Let $\bs{\tau}^{n}=\bs{\tau}(w^{\sn})$, $\bs{\tau}=\bs{\tau}(w)$, and $T^n=T(\bs{\tau}^{n})$.  
Proposition~\ref{prop:taupert} implies that for $n\ge N$, $\bs{\tau}^n$ is non-degenerate and $T^n=T(\bs{\tau})$.  Recall the distances $d_1$ and $d_2$ from Section~\ref{sec:GST}. Henceforth assume $n\ge N$. We have 
\begin{align}\label{d1bnd}
\nonumber d_1\big((T^n,d_{\bs{\tau}^n}),(T(\bs{\tau}),d_{\bs{\tau}})\big)&=\sup_{i,j\in[K]}|(\tau^n_{i,j}-\tau^n_{i,m(i,j)})-(\tau_{i,j}-\tau_{i,m(i,j)})|\\
&\le 2\sup_{i,j}|\tau^n_{i,j}-\tau_{i,j}|\to 0.
\end{align}
Let $d^n=d_{\bs{\tau}^{n}}$, $d=d_{\bs{\tau}}$, let $\ell^n$ and $\ell$ denote their respective edge length functions,  and let $[i,u]_n$ and $[i,u]$ denote points in $\overline{T}_n$ and $\overline{T}(\bs{\tau})$, respectively. Recall the map $\Upsilon_{d,d^n}$ arising in the definition of $d_2$ in Section~\ref{sec:GST}. If $[i,u]\in\overline T(w)\setminus\{0\}$, we may choose $j\in[K]$ so that $u\in(\tau_{\pi(\aij)},\tau_{\aij}]$. Then, recalling \eqref{iuchar} for both $d$ and $d^n$, we see that 
$\Upsilon_{d,d^n}([i,u])=[i,u_n(i,u)]_n$, where
\begin{align*}u_n(i,u)&=\tau^n_{\pi(\aij)}+(u-\tau_{\pi(\aij)})\frac{\ell^n(e_{\aij})}{\ell(e_{\aij})}\\
&=\tau^n_{\pi(\aij)}+(u-\tau_{\pi(\aij)})\frac{\tau^n_{\aij}-\tau^n_{\pi(\aij)}}{\tau_{\aij}-\tau_{\pi(\aij)}}.
\end{align*}
Write $u=\tau_{\pi(\aij)}+(u-\tau_{\pi(\aij)})$ and do a bit of arithmetic to see that 
\begin{equation}\label{undiff}
|u_n-u|\le 3 \sup_{i,j}|\tau^n_{i,j}-\tau_{i,j}|\le 3\Vert\bs{\tau}^n-\bs{\tau}\Vert.
\end{equation}
Therefore by the definition of $d_2$, for $n\ge N$ (we clearly may restrict $x\neq 0$ in the definition),
\begin{align*}
d_2(\mc{B}_K(w),\mc{B}_K(w^{\sn}))&=\sup_{[i,u]\in \overline{T}(\bs{\tau})\setminus\{0\}}\Vert w([i,u])-w^{\sn}(\Upsilon_{d,d^n}([i,u])\Vert\\
&=\sup_{[i,u]\in \overline{T}(\bs{\tau})\setminus\{0\}}\Vert w_i(u)-w_i^{\sn}(u_n(i,u))\Vert\\
&\le \sup_{i\in[K]}\ \ \sup_{|u-v|\le 3\Vert \bs{\tau}^n-\bs{\tau}\Vert, u\le {\bs{\tau}_{i,i}}}\Vert w_i(u)-w_i^{\sn}(v)\Vert.
\end{align*}
The last expression approaches $0$ because $w^{\sn}_i\to w_i$ and $\Vert \bs{\tau^n}-\bs{\tau}\Vert\to0$.  This and \eqref{d1bnd} imply the required convergence result.
\end{proof}

Here are some semicontinuity properties of $\mc{L}$ and  $\tau$.  
\begin{LEM}
\label{lem:uppersemi} Assume that for $i=1,2$, $(w^{\sn}_i)_{n \in \N}$ are sequences in $\mc{D}$ such that $w^{\sn}_i\rightarrow w_i$ in $\D$.
\begin{itemize}
\item[\emph{(a)}] $\mc{L}$ is lower semicontinuous on $\D$, that is, $\liminf_{n\to\infty}\mc{L}(w_i^{\sn})\ge \mc{L}(w)$ ($i=1,2$).
\item[\emph{(b)}] $\bar\tau$ is upper semicontinuous on $\D^2$, that is, $\limsup_{n\to\infty}\bar\tau(w_1^{\sn},w_2^{\sn})\le \bar\tau(w_1,w_2)$.
\item[\emph{(c)}] If $\FL(w^{\sn}_i)\to \FL(w_i)$ as $n\to\infty$ for $i=1,2$, 
then $\limsup_{n\to\infty}\tau(w_1^{\sn},w_2^{\sn})\le \tau(w_1,w_2)$. 
\end{itemize}
\end{LEM}
\begin{proof} (a),(b) We leave these as easy exercises.

\noindent(c) This is now immediate from (b).
\end{proof}
Let $D_K=\{w=(w_1,\dots,w_K)\in\mc{D}^K: w_i(0)=0\text{ and }\mc{L}(w_i)<\infty\text{ for all }i\in[K]\}$. 

\begin{LEM}\label{interpol} Assume $w\in C_K$ and $w^{\sn}\in D_K$ satisfy $w^{\sn}\to w$ and $\bs{\tau}(w^{\sn})\to\bs{\tau}(w)$ as $n\to\infty$,  Then as $n\to\infty$:\\
(a) $w_i^{\sn}\to w_i$ and $\kappa_n(w_i^{\sn})\to w_i$  in the supremum norm for each $i\in[K]$,\\
(b) $\bs{\tau}(\kappa_n(w^{\sn}))\to\bs{\tau}(w)$.
\end{LEM}
\begin{proof} (a) The convergence of $w^{\sn}_i$ to $w_i$ implies uniform convergence on compacts since $w_i$ is continuous.  This and the convergence of $\mc{L}(w_i^{\sn})=\tau_{i,i}(w_i^{\sn})$ to $\mc{L}(w_i)=\tau_{i,i}(w_i)$ implies uniform convergence of $w^{\sn}_i$ to $w_i$.  The second convergence is now an elementary calculation.\\
(b) It follows easily from the definition of $\mc{L}$, that 
$\mc{L}(\kappa_n(w_i^{\sn}))\le \mc{L}(w_i^{\sn})+\frac{1}{n}$.
This and our hypothesis imply $\limsup_{n\to\infty}\mc{L}(\kappa_n(w_i^{\sn}))\le \mc{L}(w_i)$. By (a) and Lemma~\ref{lem:uppersemi}(a), \break
$\liminf_{n\to\infty}\mc{L}(\kappa_n(w_i^{\sn}))\ge \mc{L}(w_i)$, and we conclude that
\begin{equation}\label{Lconv}
|\mc{L}(\kappa_n(w^{\sn}_i))-\mc{L}(w_i)|\to 0\text{ as }n\to\infty.
\end{equation}

\eqref{Lconv}  gives the convergence in (b) for the diagonal entries of $\bs{\tau}$.  Now consider $i\neq j$ in $[K]$. A simple calculation using the definition of $\bar\tau(w'_i,w'_j)$ shows that 
\[\bar\tau(\kappa_n(w^{\sn}_i),\kappa_n(w^{\sn}_j))\ge \bar\tau(w^{\sn}_i,w^{\sn}_j)-\frac{1}{n}.\]
Use this with \eqref{Lconv}  
and our convergence hypothesis on $\tau_{i,j}(w^{\sn})$ to see that
\[\liminf_{n\to\infty}\tau(\kappa_n(w^{\sn}_i),\kappa_n(w^{\sn}_j))\ge\tau(w_i,w_j).\]
On the other hand, part (a) above, \eqref{Lconv}, and Lemma~\ref{lem:uppersemi}(c) imply
\[\limsup_{n\to\infty}\tau(\kappa_n(w^{\sn}_i),\kappa_n(w^{\sn}_j))\le\tau(w_i,w_j).\]
The last two inequalities complete the proof of (b).
\end{proof}

Recall the rescalings $\rho_n:\mc{D}^K\to\mc{D}^K$ from Section~\ref{sec:intro}. An elementary calculation shows that
\begin{equation}\label{resctau}
\text{for any }w\in D_K\quad \bs{\tau}(\rho_n(w))=\bs{\tau}(w)/n.
\end{equation}
It follows that 
\begin{equation}\label{resnd}
\text{for any $w\in C_K$\quad$w$ is non-degenerate iff $\rho_n(w)$ is non-degenerate.}
\end{equation}
If $w\in C_K$ and $n\in\N$, recall the definitions of $\mc{B}_K(w)$ and $\mc{B}_{n,K}(w)$ (in $\TT_{gst}$) from Definition~\ref{gstwdef} and \eqref{BnKdef}, respectively.  
\begin{LEM}\label{lem:rescBK}
If $w^{\sss (1)}\in C_K$, then for all $n\in\N$,
\begin{equation}\label{rescBK}
\mc{B}_K(\rho_n(w^{\sss (1)}))=\mc{B}_{n,K}(w^{\sss (1)}).
\end{equation} 
\end{LEM}
\begin{proof}  Let $w^{\sn}=\rho_n(w^{\sss(1)})$. If $w^{\sss(1)}$ is degenerate, so is $w^{\sn}$ by \eqref{resnd}, and the result is immediate by definition.  Assume $w^{\sss(1)}$, and hence $w^{\sn}$ for any $n$, is non-degenerate. Let $[i,u]_n$ ($i\in [K]\cup\{0\}, 0\le u\le \tau_{i,i}/n$) denote generic points in $\overline{T}^n:=\overline{T}(w^{\sn})$, and let $\lb i,j\rb_n=[i,\tau_{i,j}(w^{\sn})]_n$ ($i,j\in [K]\cup\{0\}$) denote generic points in $T^n:=T(w^{\sn})$.
Define $\tilde \psi_n:\overline{T}^1\to\overline{T}^n$ by
\begin{equation}\label{tpsidef}
\tilde\psi_n([i,u]_1)=[i,(u/n)]_n.
\end{equation}
It is easy to check $\tilde\psi_n$ is a well-defined injection, and it is surjective by \eqref{resctau}.  
It follows from the above definition and \eqref{resctau} that if $\psi_n$ is the restriction of $\tilde \psi_n$ to $T^1$, then
\begin{equation}\label{psidef}
\psi_n(\aij_1)=\aij_n.
\end{equation}
Therefore $\psi_n$ is an injective map from $T^1$ to $T^n$, and as these vertex sets have the same cardinality, $\psi_n:T^1\to T^n$ is a bijection.  Clearly it preserves the root and preserves the labelled leaves $1,\dots,K$. 
Let
$\le_n$ and $<_n$ denote the ancestral orderings in $\overline{T}^n$. Another application of \eqref{resctau} shows that 
\begin{equation}\label{orders}
[i,u]_1\le_1 [j,v]_1\text{ iff }\tilde \psi_n([i,u]_1)\le_n\tilde\psi_n([j,v]_1).
\end{equation}
In particular, $\psi_n$ is a graph isomorphism between $T^1$ and $T^n$. 
Turning next to the metrics, note that 
\begin{equation*}
d_{\bs{\tau}(w^{\sn})}(\tilde\psi_n([i,u]_1),\tilde\psi_n([j,v]_1))=d_{\bs{\tau}(w^{\sn})}([i,u/n]_n,[j,v/n]_n)=n^{-1}d_{\bs{\tau}(w^{\sss (1)})}([i,u]_1,[j,v]_1),
\end{equation*}
where in the last we have used \eqref{resctau} and the definition of $d_{\bs{\tau}}$ in \eqref{treemet}. Therefore 
\[\tilde\psi_n:(\overline{T}^1,d_{\bs{\tau}(w^{\sss (1)})}/n)\to (\overline{T}^n, d_{\bs{\tau}(w^{\sn})})\text{ is an isometry}.
\]
This shows $\psi_n:(T^1,d_{\bs{\tau}(w^{\sss (1)})}/n)\to (T^n,d_{\bs{\tau}(w^{\sn})})$ is an isometric isomorphism, and
$\tilde \psi_n$ is the unique isometric extension of $\psi_n$ to $(\overline{T}^1,d_{\bs{\tau}(w^{\sss (1)})}/n)$,  denoted by $\bar\psi_n$ in Remark~\ref{idgsts}.  We have 
\[\phi_{w^{\sn}}\circ\bar\psi_n([i,u]_1)=\phi_{w^{\sn}}([i,(u/n)]_n)=w^{\sn}_i(u/n)=w_i^{\sss (1)}(u)/\sqrt n=\phi_{w_1^{\sss (1)}}([i,u]_1)/\sqrt n.\]
By Remark~\ref{idgsts} the result follows.
\end{proof}

\section{General Convergence and  the Proofs of Theorems~\ref{thm:jointconvergence} and \ref{thm:rjointconvergence}}
\label{sec:G}

\subsection{Proof of Lemma~\ref{lem:hBM_shape0}}\label{ssec:HBMnondeg}
\begin{LEM}\label{lem:distlife}
For every $s>0$, 
$\N_{\sss H}^s[J(\{(u,w):\FL(w)\neq u\})]=0$.
\end{LEM}
\begin{proof}  
It suffices to show for each $u>0$, $\N_{\sss H}^s\big[H_u(\{w:\FL(w)\neq u\})\big]=0$, and this %clearly 
follows from
\begin{equation}
\N_{\sss H}\big[H_u(\{w:\FL(w)\neq u\})\big]=0.\label{NH_Lu}
\end{equation}
The mean measure of $H_u$ under $\N_{\sss H}$ is the law of a Brownian path stopped at $u$ (this is an elementary consequence of (II.8.6)(a) of \cite{Per02})  
so \eqref{NH_Lu} reduces to showing that if $B$ is a Brownian motion starting at $0$ and $B^u(t)=B(t\wedge u)$, then $\FL(B^u)=u$ a.s., which is obvious.
\end{proof}

To prove Lemma~\ref{lem:hBM_shape0} we need a moment formula for $H$ which is implicit in \cite{CFHP20}.   
Using notation from \cite{CFHP20}, we let $r\in \N$, $\vec t=(t_1,\dots,t_r)\in(0,\infty)^r$, $\digamma\in\Sigma_{r}$ be a non-degenerate shape, and $\prec$ denote the ancestral relation on $\digamma$.  
Consider a metric, $d$, on $\digamma$ satisfying $d(i,0)=t_i$ for the leaves $i=1,\dots,r$, and let $\vec u=(\vec t,u_{r+1},\dots, u_{2r-1})$ where $u_i=d(i,0)$ for the interior vertices $i>r$ in $\digamma$ (recall Remark~\ref{rem:labels}(b)). Hence the set of possible values of $(u_{r+1},\dots, u_{2r-1})$ is
\begin{align*}\M(\vec t,\digamma)=\{(u_{r+1},\dots,u_{2r-1}):& \text{ if $u_i=t_i$ for $i\le r$, then for all $k,\ell\le 2r-1$,}\\
&\qquad k\prec \ell\text{ implies }u_k<u_\ell\} \text{ (if $r=1$, the set is empty}),
\end{align*}
 and $d=d(\vec{u})$ is uniquely determined by $\vec u$.  
Then $(W^{1,\vec u,\digamma},\dots,W^{r,\vec u,\digamma})$ denotes a tree-indexed Brownian motion with variance parameter $\sigma^2$ on $T(\digamma,\vec u):=(\digamma,d(\vec u))$ (see Definition~2.8 of \cite{CFHP20}). This means $\{W_t^{i,\vec u,\digamma}, t\in[0, t_i]:i\le r\}$ are correlated Brownian motions, starting at $o$, (variance parameter $\sigma^2$), where for $i\neq j$ $W^{i,\vec u,\digamma}=W^{j,\vec u,\digamma}$ on $[0,u_{i\wedge j}]$, and  $W^{i,\vec u,\digamma}$ and $W^{j,\vec u,\digamma}$ evolve  independently after time $u_{i\wedge j}$. Here, for $r>1$, $i\wedge j\in\{r+1,\dots,2r-1\}$ is the most recent common ancestor of leaves $i$ and $j$ with respect to $\prec$.  We write $d\vec{u}$ for integration with respect to $(u_{r+1},\dots,u_{2r-1})$.

\begin{LEM}\label{lem:hBMmoments} Fix $d\in \N$.  If $r\in\N$, $\vec t\in(0,\infty)^r$ and  $\phi:\mc{C}^r\to\R_+$ is Borel 
then
\begin{equation}\label{hbmmoments}\N_{\sss H}\Bigl[\int\dots\int\phi(w_1,\dots,w_r)dH_{t_1}(w_1)\dots H_{t_r}(dw_r)\Bigr]=\sum_{\digamma\in\Sigma_r}\int_{\M(\vec t,\digamma)}E\Big[\phi(W^{1,\vec u,\digamma}_{\cdot\wedge t_1},\dots, W^{r,\vec u,\digamma}_{\cdot\wedge t_r})\Big]d\vec{u}.
\end{equation}
\end{LEM}
\begin{proof} It suffices to prove \eqref{hbmmoments} for $\phi\ge 0$ a Borel function of $(w^1|_{\vec s^{(1)}},\dots,w^r|_{\vec s^{(r)}})$ where each $\vec s^{(\ell)}=(s_0^{(\ell)},\dots,s_{m^{(\ell)}}^{(\ell)})$ is a finite dimensional set of times $0=s_0^{(\ell)}<\dots<s_{m^{(\ell)}}^{(\ell)}=t_\ell$.  (For this recall that $w_i(\cdot\wedge t_i)=w_i$ for $H_{t_i}$-a.a. $w_i$.) In this finite dimensional case both sides of \eqref{hbmmoments} are integrals of $\phi$ with respect to a finite dimensional finite measure and so equality of these measures will follow from equality of their characteristic functions. This is established in Propositions~2.6 and 2.9 of \cite{CFHP20}, and we are done.
\end{proof}
\begin{COR}\label{cor:taulaw} Let $r\in \N^{\ge 2}$, $\vec t\in(0,\infty)^r$ and $\psi:\R_+^{\{(i,j)\in [r]\times [r]:i<j\}}\to \R_+$ be Borel. Then
\begin{equation}\label{taumoments}
\N_{\sss H}\Big[\int\dots\int\psi((\tau(w_i,w_j))_{i,j\in[r],i<j})dH_{t_1}(w_1)\dots H_{t_r}(dw_r)\Bigr]=\sum_{\digamma\in\Sigma_r}\int_{\M(\vec t,\digamma)}\psi((u_{i\wedge j})_{i<j\le r})d\vec u.
\end{equation}
\end{COR}
\begin{proof} Let $r$, $\vec{t}$ and $\digamma\in\Sigma_r$ be as above, and $\vec u\in \M(\vec t,\digamma)$. By definition this means that $u_{i\wedge j}<\min(t_i,t_j)$. By definition and elementary properties of Brownian motion,  
\begin{equation}\label{uistau}
u_{i\wedge j}=\tau(W^{i,\vec u,\digamma}_{\cdot\wedge t_i},W^{j,\vec u,\digamma}_{\cdot\wedge t_j}) \text{ for all }i,j\in [r],\ i<j \ a.s.
\end{equation}
Let $\psi$ be as above and apply Lemma~\ref{lem:hBMmoments} with $\phi(w_1,\dots,w_r)=\psi((\tau(w_i,w_j))_{i,j\in[r],i<j})$ and then use \eqref{uistau} to complete the proof. 
\end{proof}

\begin{proof}[Proof of Lemma~\ref{lem:hBM_shape0}]
By the definition of non-degeneracy and simple union bounds, it suffices to show the following:
\begin{itemize}
\item[(0)] for $K=1$, $\hat\N^s_{\sss H}\big(\tau(W_1,W_1)=0\big)=0$,
\item[(a)] for $K=2$, $\hat\N^s_{\sss H}\big(\tau(W_1,W_2)=0\big)=0$,
\item[(b)] for $K=2$, $\hat\N^s_{\sss H}\big(\tau(W_1,W_2)=\mc{L}(W_1)\big)=0$,
\item[(c)] for $K=3$, $\hat\N^s_{\sss H}\big(\tau(W_1,W_2)=\tau(W_1,W_3)=\tau(W_2,W_3)\big)=0$.
\end{itemize}
Item (0) is immediate from Lemma~\ref{lem:distlife}.  
In each of (a)-(c) we want to show a result of the form $\hat\N_{\sss H}^s\big((W_1,\dots, W_r)\in A\big)=0$, i.e.
\begin{align}
\N_{\sss H}^s\Bigg[\dfrac{\int_{(0,\infty)^r}H_{t_1}\times\cdots \times H_{t_r}(A)dt_1\dots dt_r}{\Big(\int_{(0,\infty)}H_t(1) dt\Big)^r}\Bigg]=0.
\end{align}
Elementary arguments show it suffices to  to prove that $\N_{\sss H}\big[H_{t_1}\times\cdots \times H_{t_r}(A)\big]=0$ for every $t_1,\dots, t_r>0$.

By \eqref{NH_Lu} we have $\N_{\sss H}$-a.e. that $\FL(w_i)=t_i$ for $H_{t_i}$-a.a. $w_i$.  Therefore to prove (a) and (b) it suffices to show that for fixed $t_1,t_2\in(0,\infty)$,
\begin{equation}\label{abeq}
\N_{\sss H}\big[H_{t_1}\times H_{t_2}(\tau(w_1,w_2)\notin(0,t_1\wedge t_2))\big]=0.
\end{equation}
To see this apply Corollary~\ref{cor:taulaw} with $r=2$.  Recall that $\Sigma_{[2]}$ consists of a single shape $\digamma$~corresponding to a single internal binary branch, and $u\in \M(\vec t,\digamma)$ iff $0<u<\min(t_1,t_2)$.  So this and Corollary~\ref{cor:taulaw} 
show that the left-hand side of \eqref{abeq} equals $\int_0^{\min(t_1,t_2)}\indic{u\notin(0,\min(t_1,t_2))}du=0$. 

Turning to (c), let $t_i\in(0,\infty)$ for $i=1,2,3$. As for \eqref{abeq} it suffices to prove
\begin{equation}\label{ceq}
\N_{\sss H}\big[H_{t_1}\times H_{t_2}\times H_{t_3}(\tau(w_1,w_2)=\tau(w_2,w_3)=\tau(w_1,w_3))\big]=0.
\end{equation}
In this case we take $r=3$ in Corollary~\ref{cor:taulaw} and recall that $\Sigma_3=\{\digamma_i:i=1,2,3\}$ has 3 shapes.  For each shape, the event in \eqref{ceq} reduces to equality of the two internal branch times $u_4, u_5$, that is, by Corollary~\ref{cor:taulaw} the left-hand side of \eqref{ceq} equals 
\[\sum_{i=1}^3\int_{\M(\vec t,\digamma_i)}\indic{u_{1\wedge 2}=u_{1\wedge 3}=u_{2\wedge 3}}du_4du_5=\sum_{i=1}^3\int_{\M(\vec t,\digamma_i)}\indic{u_4=u_5}du_4du_5=0.\]
\end{proof}
 Let $\tau_{i,j}(\digamma,\vec u)=u_{i\wedge j}$ if $i\neq j$ and $\tau_{i,i}(\digamma,\vec u)=t_i$, extend $W^{i,\vec u,\digamma}$ to be constant on $[t_i,\infty)$,  and then let $W^{\vec u,\digamma}=(W^{i,\vec u,\digamma})_{i\le r}\in C_r$. Then by \eqref{uistau}, $\bs{\tau}(\digamma,\vec u)=\bs{\tau}(W^{\vec u, \digamma})=:\bs{\tau}$.  The construction of $T(W^{\vec u,\digamma})$, and in particular the construction of $\pi$, in Section~\ref{sec:GST} easily gives $\digamma=T(W^{\vec u,\digamma})$. Remark~\ref{gcaij} and \eqref{dtauij0} show that $d_{\tau}(0,i\wedge j)=\tau_{i,j}=u_{i\wedge j}=d(\vec{u})(0,i\wedge j)$ (there is now no ambiguity in interpreting $i\wedge j$),  and therefore
 $d(\vec{u})=d_{\bs{\tau}}$. These equalities imply that $\overline\digamma=\overline T(W^{\vec u,\digamma})$ and so 
 $\phi_{W^{\vec u,\digamma}}([i,v])=W_i^{\vec u,\digamma}(v)$ is an embedding on $\overline\digamma$.  We conclude that $(\digamma,d(\vec u),\phi_{W^{\vec u,\digamma}})$ is a non-degenerate gst. As in Corollary~\ref{cor:taulaw} we can use the above identifications to see that if $Z_{s,r}=(s/2)\1(H_s(1)>0)\Bigl[\int_0^\infty H_t(1)\,dt\Bigr]^{-r}$, then (recall $\vec u=(\vec t, u_{r+1},\dots,u_{2r-1})$)
 \begin{align}\label{Wabscont}
\nonumber& \hat\N^s_{\sss H}((W,\bs{\tau}(W),\mc{B}_r(W))\in\cdot)\\
\nonumber&=\N_{\sss H}\Bigl(Z_{s,r}\int_{\R_+^r}\int_{\mc{C}^r}\1((W,\bs{\tau}(W),\mc{B}_r(W))\in\cdot)\,dH_{t_1}(W_1)\,\dots\,dH_{t_r}(W_r)d\vec{t}\Bigr)\\
\nonumber&\ll\N_{\sss H}\Bigl(\int_{\R_+^r}\int_{\mc{C}^r}\1((W,\bs{\tau}(W),\mc{B}_r(W))\in\cdot)\,dH_{t_1}(W_1)\,\dots\,dH_{t_r}(W_r)d\vec{t}\Bigr)\\
&=\int_{\R_+^r}\sum_{\digamma\in\Sigma_r}\Bigl[\int_{\M(\vec t,\digamma)}P\Bigl((W^{\vec u,\digamma},\bs{\tau}(\digamma,\vec u),(\digamma,d(\vec u),\phi_{W^{\vec u,\digamma}}))\in\cdot\Bigr)d\vec u\Bigr]d\vec t.
\end{align}

The above absolute continuity result allows one to deduce qualitative properties of the left-hand side as was done in the proof of Lemma~\ref{lem:hBM_shape0}.
\subsection{A  Convergence Theorem For General Historical Processes}\label{ssec:genthm}
We start with a simple abstract result on weak convergence of a sequence of samples selected according to a sequence of weakly convergent random measures.  For a Polish space $E_2$, let $F_K:\mc{M}_F(E_2)\to \mc{M}_F((E_2)^K)$ denote the normalised $K$-fold product measure, i.e. for $\mu \in \mc{M}_F(E_2)$,
\begin{equation}
\label{FKdef}
F_K(\mu)=\begin{cases}
0_M, & \text{ if }\mu(1)=0\\
\dfrac{\mu\times \cdots \times \mu}{\mu(1)^K}, & \text{ otherwise.}
\end{cases}
\end{equation}

\begin{LEM} \label{lem:absrandompts} Fix $K\in \N$ and let $E_1,E_2$ be Polish spaces.  Assume for each $n\in\N$, we have a random vector $(Z^{\sn},M^{\sn})\in E_1\times \mc{M}_F(E_2)$ with  $M^{\sn}(E_2)>0$ a.s., and conditional on $(Z^{\sn},M^{\sn})$, $\{V^{\sn}_i:i\le K\}$ is a collection of i.i.d.~random vectors with common law $M^{\sn}/M^{\sn}(E_2)$.  Assume also $(Z,M)\in E_1\times \mc{M}_F(E_2)$, $M(E_2)>0$ a.s., and conditional on $(Z,M)$, $\{V_i:i\le K\}$ are i.i.d. with law $M/M(E_2)$.
If $(Z^{\sn},M^{\sn})\cweak(Z,M)$ in $E_1\times\mc{M}_F(E_2)$, then $(Z^{\sn},M^{\sn},V^{\sn})\cweak(Z,M,V)$ on $E_1\times \mc{M}_F(E_2)\times (E_2)^K$.
\end{LEM}
\begin{proof} 
Let $\phi_1:E_1\times\mc{M}_F(E_2)\to \R$ and $\phi_2:(E_2)^K\to\R$ be bounded and continuous functions.  
Define $\psi:E_1\times\mc{M}_F(E_2)\to\R$ by 
\[\psi(z,\nu)=\Big[\phi_1(z,\nu)\int\phi_2 \d F_K(\nu)\Big]\indic{\nu(1)>0}.\]
Then $\psi$ is a bounded function which is continuous at all points $(z,\nu)$ with $\nu(1)>0$.  Note that $M(E_2)>0$  a.s., and $M^{\sn}(E_2)>0$, a.s., and that
\[\E[\phi_2(V^{\sn})|M^{\sn},Z^{\sn}]=\int\phi_2\d F_K(M^{\sn}), \quad \text{ and }\E[\phi_2(V)|M,Z]=\int\phi_2\d F_K(M).\]
We may conclude from the assumed weak convergence that
\begin{align*} \lim_{n\to\infty}\E\big[\phi_1(Z^{\sn},M^{\sn})\phi_2(V^{\sn})\big]
&= \lim_{n\to\infty}\E\big[\psi(Z^{\sn},M^{\sn})\big]\\
&=\E[\psi(Z,M)]\\
&=\E[\phi_1(Z,M)\phi_2(V)].
\end{align*}
The result follows.
\end{proof}

In this subsection we work in an abstract setting that includes our rescaled historical processes for BRW and lattice trees.  
Assume that for each $n\in\N$,  $(H_t^{\sn}, t\ge 0)$ is a stochastic process with sample paths in 
$\mc{D}_{\infty}(\mc{M}_F(\mc{D}))$ on some probability space $(\Omega,\F,\P)$. We let $S^{\sn}=\inf\{t:H_t^{\sn}=0\}$ {\bf and assume throughout that}
\begin{equation}\label{Sncond} 
0<S^{\sn}<\infty, \text{ and }H^{\sn}_t=0\text{ for all }t\ge S^{\sn} \text{ and }n\in\N\  a.s. 
\end{equation}
This is clear for both BRW and lattice trees with $H_t^{\sn}$ defined as in Sections \ref{sec:BRW} and \ref{sec:LT},  respectively. The first part holds by \eqref{surv1} and because there is an individual of generation 0, and the second part is obvious. 
 Condition~\ref{cond:fdd}(i)  below will imply that $\P(S^{\sn}>s)>0$ for all $s>0$, which will allow us to introduce $\P_n^s=\P(\cdot|S^{\sn}>s)$. 
We may define $I^{\sn}$ and $J^{\sn}$ as before in this setting, that is,
$I^{\sn}$ is the finite (by \eqref{Sncond}) measure on $\overline{\mc{D}}$ defined by 
\[I^{\sn}(A\times B)=\int_0^\infty\int \1_A(u)\1_B(w)H_u^{\sn}(\d w)\d u,\]
and $J^{\sn}=F_1(I^{\sn})$. 
 For $t>0$, let $I^{\sn}_{t}(G)=I^{\sn}(G\cap([0,t]\times\D))$, and $J_t^{\sn}=F_1(I_t^{\sn})$. Note that $J^{\sn}$ and $J_t^{\sn}$ are probabilities for all $t>0,\, n\in\N$ a.s. because $0<I^{\sn}_t(\oD)\le I^{\sn}(\oD)<\infty$ for all $t>0$ a.s. by \eqref{Sncond}. For the limiting historical Brownian motion we also introduce $I_t(\cdot)=I(\cdot\cap([0,t]\times\D))\in \mc{M}_F(\mc{\overline{D}})$ a.s., and $J_t=F_1(I_t)$ which are probabilities on $\oD$ for all $t>0$  a.s.  
 We suppress dependence on the spatial variance $\sigma_0^2$ of the limiting HBM. Recall  the definition of the lifetime $\FL(w)$ of a path $w\in\oD$ from \eqref{Ldef}.
\begin{COND}
\label{cond:fdd}
The following hold for every $s>0$:
\begin{itemize}
\item[(i)]  $\exists c_0>0$ such that  $\lim\limits_{n \to \infty}\sup\limits_{t\ge s}|nt\P(S^{\sn}>t)- c_0|=0$, and $\P(S^{\sn}>s)>0\ \forall n\in\N$,
\item[(ii)] as $n \to \infty$,   \begin{equation}\label{fddcvgce}\P_n^s(H^{\sn}\in\cdot)\cfdd\N_{\sss H}^s(H\in\cdot),
\end{equation}  
\item[(iii)] for every $t>0$, \begin{equation}\label{mombound} \sup_{n \in \N}\sup_{u\le t}\E^s_n[H^{\sn}_u(1)^p]<\infty\ \ \forall p\in\N,
\end{equation}
\item[(iv)] for every $n\in \N$,  $\E_n^s[J^{\sn}(\{(u,w):\FL(w)\neq[u]_n\})]=0$.
\end{itemize}
\end{COND}

\begin{LEM}\label{i'} 
$\quad$
\begin{itemize}
\item[\emph{(a)}] Assume that  $S^{\sn}$ has the same law as $S^{\sss(1)}/n$ and that $\lim_{t\to\infty}t\P(S^{\sss(1)}>t)=c_0$ for some $c_0>0$.  Then Condition~\ref{cond:fdd}(i) holds.
\item[\emph{(b)}] If Condition~\ref{cond:fdd}(i) holds then for each $s>0$ there 
 an $n_0=n_0(s)\in\N$ so that 
$$\frac{s}{3(s\vee t)}\le \P^{s}_n(S^{\sn}>t)\le \frac{3s}{s\vee t},\ \ \text{for all }t>0, \ n\ge n_0.$$
\end{itemize}
\end{LEM}
\begin{proof} 
(a) Let $s>0$. Then $\sup_{t\ge s}|nt\P(S^{\sn}>t)-c_0|=\sup_{t\ge s}|nt\P(S^{\sss(1)}>nt)-c_0|\rightarrow0$ as $n\to\infty$. Also the given convergence implies $S^{\sss(1)}$ has unbounded support and so the same is true for $S^{\sss(1)}/n$.

\noindent(b) By Condition~\ref{cond:fdd}(i) there is an $n_0$ so that for $n\ge n_0$, $\sup_{t\ge s}|nt\P(S^{\sn}>t)-c_0|\le c_0/2$. Therefore for $n\ge n_0$ and $t\ge s$ we have 
\[\P_n^s(S^{\sn}>t)=\frac{\P(S^{\sn}>t)}{\P(S^{\sn}>s)}\in\Bigl[\frac{s}{3t},\frac{3s}{t}\Bigr]=\Bigl[\frac{s}{3(s\vee t)},\frac{3s}{s\vee t}\Bigr].\]
For $t<s$ the probability in the Lemma is one and the bound is trivial.\end{proof}

Condition~\ref{cond:fdd} holds both  for lattice trees with $d>8$ and $L$ large, and for BRW.  In either case, (i) 
follows from \eqref{surv1} and Lemma~\ref{i'}(a). 
  Part (iv) holds (again in either case) by \eqref{LTndist}, 
and the  validity of (ii) was discussed at the end of Section~\ref{sec:scalelim}. For lattice trees, (iii) is well-known (see e.g.~\cite[Lemma 2.13]{CFHP20}),  
and for BRW, (iii) is readily proved using a predictable square function inequality of Burkholder as in Exercise II.4.1 of \cite{Per02} (recall \eqref{Ymom}).

\begin{LEM} \label{lem:IJStconv} 
Suppose that Condition \ref{cond:fdd} holds.  Then for any $t,s>0$, under $\P^s_n$ we have 
\[\P^s_n\big((I^{\sn}_t,J^{\sn}_t,S^{\sn})\in\cdot\big)\cweak \N_{\sss H}^s\big((I_t,J_t,S)\in\cdot\big)\ \ \text{in }\mc{M}_F(\oD)^2\times \R_+\text{ as }n\to\infty.\]
\end{LEM}
\begin{proof} Let $s,t>0$.  The weak convergence of $I^{\sn}_t$ to $I_t$ is proved as in Lemma~2.2 of \cite{HP19} using the method of moments and Condition \ref{cond:fdd}.    
Although only minor changes are needed in that argument, we sketch the proof below.   
Let $p\in\N$, $u_1,\dots,u_p\in[0,t]$,  and $\phi\in C_b(\oD)$. Then \eqref{fddcvgce}, \eqref{mombound} and a uniform integrability argument give
\begin{equation} 
\lim_{n\to\infty}\E^s_n\Bigl[\prod_{i=1}^pH^{\sn}_{u_i}(\phi(u_i, \cdot))\Bigr]=\N_{\sss H}^s\Bigl[\prod_{i=1}^pH_{u_i}(\phi(u_i,\cdot))\Bigr].\label{moments_conv}
\end{equation}
A simple Fubini argument together with \eqref{mombound}, \eqref{moments_conv} and Dominated Convergence then implies
\begin{equation}\label{pmomc}
\lim_{n\to\infty}\E^s_n\Bigl[\Bigl(\int_0^tH^{\sn}_u(\phi(u,\cdot))\,du\Bigr)^p\Bigr]=\N_{\sss H}^s\Bigl[\Bigl(\int_0^tH_u(\phi(u,\cdot))\,du\Bigr)^p\Bigr].
\end{equation}

Taking $p=1$ and $\phi(u,w)=\psi(w)$ gives 
\[\E^s_n\Bigl[\int_0^t H^{\sn}_u(\cdot)\,du\Bigr]\rightarrow\N_{\sss H}^s\Bigl[\int_0^t H_u(\cdot)\,du\Bigr]\text{ in }\M_F(\D).\]
This implies tightness of $\{\E^s_n[I^{\sn}_t(\cdot)]:n\in\N\}$ on $\D$, and so if $\vep>0$ we may find a compact set $K\subset\D$ so that 
\[\sup_n\E^s_n\Bigl[\int_0^tH^{\sn}_u(K^c)\,du\Bigr]<\vep^2.\]
By Markov's inequality we obtain 
\begin{equation}
\sup_n\P_n^s\Bigl(I^{\sn}_t\big(([0,t] \times K)^c\big)>\vep\Bigr)\le \sup_n\E_n^s\Bigl[\int_0^t H^{\sn}_u(K^c)\,du\Bigr]/\vep<\vep.\label{comp_cont_Int}
\end{equation}
A simpler argument gives tightness of the total masses $\int_0^tH^{\sn}_u(1)du=I_t^{\sn}(1)$.  By a well-known characterization of tightness for random measures (e.g., take constant processes in Theorem~II.4.1 of \cite{Per02}), the compact containment  in \eqref{comp_cont_Int} and tightness of the total mass 
now implies tightness of the laws $\{\P^s_n(I^{\sn}_t\in\cdot):n \in \N\}$ on $\M_F(\oD)$. 

Let $\tilde I_t$ denote a subsequential limit point of these laws (under a probability $\P$). It follows from \eqref{mombound} and a uniformly integrability argument that for $\phi\in C_b(\oD)$
\[\lim_{n\to\infty}\E^s_n\big[I^{\sn}_t(\phi)^p\big]=\E\big[\tilde I_t(\phi)^p\big],\]
and so by \eqref{pmomc} we have
\begin{equation}\label{momid}\E[\tilde I_t(\phi)^p]=\N^s_{\sss H}[I_t(\phi)^p]\ \ \forall p\in\N.
\end{equation}
Use the fact that $\N_{\sss H}^s[e^{\theta I_t(1)}]<\infty$ for small $\theta>0$ (e.g. see the proof of Lemma~2.2 in \cite{HP19}) to conclude that $\P(\tilde I_t(\phi)\in\cdot)=\N_{\sss H}^s(I_t(\phi)\in\cdot)$ for every $\phi\in C_b(\oD)$, and in particular $\E[e^{-\tilde I_t(\phi)}]=\N_{\sss H}^s[e^{-I_t(\phi)}]$ for all such $\phi$.
It follows (e.g. see Lemma~II.5.9 of \cite{Per02}) that  
$\P(\tilde I_t\in\cdot)=\N^s_{\sss H}(I_t\in\cdot)$  and hence we get $\P^s_n(I^{\sn}_t\in\cdot)\cweak \N^s_{\sss H}(I_t\in\cdot)$ on $\M_F(\oD)$. 

Note that 
$F_1:\mc{M}_F(\oD)\to \mc{M}_F(\oD)$ is continuous except at $0_M$ which has zero probability with respect to $\N^s_{\sss H}(I_t\in\cdot)$, and so by the continuous mapping theorem we conclude
\begin{equation}\label{IJcvgce}
\P_n^s\big((I^{\sn}_t,J_t^{\sn})\in\cdot\big)\cweak \N_{\sss H}^s\big((I_t,J_t)\in\cdot\big) \text{ for all }s>0.
\end{equation}

It remains to accommodate $S^{\sn}$ but this is implicitly contained in the above since $s>0$ is arbitrary, as we now show.
Let $\psi:\M_F(\oD)^2\to \R$ be bounded and continuous and $t>0$. Then
\begin{align} 
\nonumber\E_n^s\big[\psi(I^{\sn}_t,J^{\sn}_t)\indic{S^{\sn}>t}\big]&=\dfrac{\E\big[\psi(I^{\sn}_t,J^{\sn}_t)\indic{S^{\sn}>s\vee t}\big]}{\P(S^{\sn}>s)}\\
\nonumber&=\E_n^{s\vee t}[\psi(I^{\sn}_t,J^{\sn}_t)]\cdot \frac{\P(S^{\sn}>s\vee t)}{\P(S^{\sn}>s)}\\
\label{weaklimit}&\rightarrow \N^{s\vee t}_{\sss H}[\psi(I_t,J_t)]\cdot \frac{s}{s\vee t}\ \ \text{ as }n\to\infty,
\end{align}
where we have used \eqref{IJcvgce} and the survival asymptotics Condition \ref{cond:fdd}(i).
Now argue as above and use the survival probability \eqref{SBMsurv} 
to see that the right side of \eqref{weaklimit} is 
$\N^s_{\sss H}[\psi(I_t,J_t)\indic{S>t}]$.
The result now follows easily.
\end{proof}
\begin{COR}\label{cor:IJScvgce}
Suppose that Condition \ref{cond:fdd} holds.   
Then for any $s>0$,
\[\P_n^s\big((I^{\sn},J^{\sn},S^{\sn})\in\cdot\big)\cweak \N_{\sss H}^s\big((I,J,S)\in\cdot\big)\ \ \text{on }\M_F(\oD)^2\times\R_+\ \ \text{ as }n\to\infty.\]
\end{COR}
\begin{proof} Let $\psi:\M_F(\oD)^2\times\R_+\to\R$ be bounded and continuous. 
Now observe that on $\{S^{\sn}\le t\}$, we have $I^{\sn}=I^{\sn}_t$ and $J^{\sn}=J^{\sn}_t$.  So by 
Lemma~\ref{i'} we have for $n\ge n_0$,
\begin{align*}\Big|\E_n^s&\big[\psi(I^{\sn},J^{\sn},S^{\sn})-\psi(I_t^{\sn},J_t^{\sn},S^{\sn})\big]\Big|\\
&\le \Big|\E_n^s\big[\psi(I^{\sn}_t,J^{\sn}_t,S^{\sn})\indic{S^{\sn}>t}\big]\Big|+\Big|\E_n^s\big[\psi(I^{\sn},J^{\sn},S^{\sn})\indic{S^{\sn}>t}\big]\Big|\\
&\le 2\Vert\psi\Vert_\infty\frac{3s}{s\vee t}.
\end{align*}
Similarly,  since for all $s,t>0,$ 
$\N_{\sss H}^s(S>t)=\frac{s}{s\vee t}$ (by \eqref{SBMsurv}),  
we have
\[\Big|\N_{\sss H}^s\big[\psi(I_t,J_t,S)-\psi(I,J,S)\big]\Big|\le 2\Vert\psi\Vert_\infty\frac{s}{s\vee t}.\]
Choose $t$ large enough and apply Lemma~\ref{lem:IJStconv} to finish the proof.
\end{proof}

Fix $K\in\N$.  We use the same notation as in Section~\ref{sec:enlargement}, so that under $\hat\P^s_n$, given $(I^{\sn},J^{\sn},S^{\sn})$, 
$(\mathfrak{T}^{\sn}_i,W^{\sn}_i)_{i\le K}$ are i.i.d.~random vectors distributed according to the (given) probability $J^{\sn}$.  We set $\overline{W}^{\sn}=(\mathfrak{T}^{\sn}_i,W^{\sn}_i)_{i\le K}\in\oD^K$. The above prescribes the law of $(I^{\sn},J^{\sn},S^{\sn},\overline{W}^{\sn})$ on $\M_F(\oD)^2\times\R_+\times\oD^K$ under $\hat\P^s_n$.  
Similarly the joint law of $(I,J,S,\overline W)$ on $\M_F(\oD)^2\times\R_+\times\oD^K$ under $\hat\N^s_{\sss H}$ is such  that, given $(I,J,S)$, $\overline W=(\mathfrak{T}_i,W_i)_{i\le K}$ are i.i.d.~with law $J$. In fact, under $\hat\N^s_{\sss H}$,  $(I,J,S,\overline W)$ is supported on $\mc{M}_F(\R_+\times \mc{C})^2\times\R_+\times (\R_+\times \mc{C})^K$, where $\mc{C}\subset \mc{D}$ is the space of continuous functions from $\R_+$ to $\R^d$ (see Section~\ref{sec:scalelim}).

\begin{REM}
Note that $S^{\sn}=\inf\{t:I^{\sn}([t,\infty)\times\D)=0\}$ and $S=\inf\{t:I([t,\infty)\times\D)=0\}$, and so in
the above conditioning we are really only conditioning on $I^{\sn}$ and $I$, respectively.\Enddef
\end{REM}
\begin{LEM}\label{quenchlimit} 
Suppose that Condition \ref{cond:fdd} holds.  Then
\begin{equation*}
\hat\P^s_n\big((I^{\sn},J^{\sn},S^{\sn},\overline W^{\sn})\in\cdot\big)\cweak \hat\N^s_{\sss H}\big((I,J,S,\overline{W})\in\cdot\big) \quad \text{ in }\mc{M}_F(\oD)^2\times\R_+\times\oD^K \quad \text{as } n\to\infty.
\end{equation*}
\end{LEM}
\begin{proof} We apply Lemma~\ref{lem:absrandompts} with $Z^{\sn}=(J^{\sn},S^{\sn})$, $Z=(J,S)$, $M=I$, $M^{\sn}=I^{\sn}$, and $V^{\sn}=\overline W^{\sn}$.  Corollary~\ref{cor:IJScvgce} provides the weak convergence required to use Lemma~\ref{lem:absrandompts}, and the a.s. strict positivity of $I^{\sn}(\overline{\D})$ and $I(\overline{\D})$ was noted in Section~\ref{sec:enlargement}
 when defining $J^{\sn}$ and $J$.
\end{proof}

In view of this and the fact that the weak convergence is for random vectors in a Polish space we may use Skorokhod's representation and work on a new probability space $(\Omega,\F,\P^s)$ where we have 
\begin{equation}\label{asconv} (I^{\sn},J^{\sn},S^{\sn},\overline {W}^{\sn})\rightarrow (I,J,S,\overline W), \  \quad \P^s \text{-a.s.}
\end{equation}
Recall the definition of the branch time $\tau(w,w')$ of  $w,w'\in\D$  from \eqref{tau_def}  and let $K\in\N$.  
Condition \ref{cond:fdd}(iv) and  Lemma~\ref{lem:distlife} imply  that 
\begin{equation}\label{UL}
\forall i\le K, \quad [\mathfrak{T}_i^{\sn}]_n=\FL(W_i^{\sn})<\infty\ \  \forall n,\quad \text{ and } \  \mathfrak{T}_i=\FL(W_i)<\infty,\ \quad  \P^s \text{-a.s.},
\end{equation}
and therefore
\begin{equation}\label{taubnd}
\tau(W^{\sn}_i,W^{\sn}_j)\le \mathfrak{T}^{\sn}_i\wedge \mathfrak{T}^{\sn}_j<\infty\ \ \text{and} \ \ \tau(W_i,W_j)\le \mathfrak{T}_i\wedge \mathfrak{T}_j<\infty,\ \ \forall i, j\le K, n\in\N\ \ \text{a.s.}
\end{equation}
It follows from \eqref{asconv} and \eqref{UL} that 
\begin{equation}\label{asconv2}
\big(I^{\sn},J^{\sn},S^{\sn},(\FL(W^{\sn}_i),W^{\sn}_i)_{i\le K}\big)\rightarrow\big(I,J,S,(\FL(W_i),W_i)_{i\le K}\big)\ \ \text{a.s. in }\M_F(\oD)^2\times\R_+\times\oD^K.
\end{equation}
Our goal is to add the branch times, $\tau(W^{\sn}_i,W^{\sn}_j)$, to the above convergence.

\begin{LEM}\label{Inlt}
Suppose that Condition \ref{cond:fdd} holds.  Then for any $s,\eta>0$ there is a $\what>0$ such that $\P^s_n(I^{\sn}(1)\le \what)<\eta$ for all $n\in \N$.
\end{LEM}
\begin{proof} Given $s,\eta>0$ we may choose $\what>0$ so that $\N_{\sss H}^s(I(1)\le 2\what)<\eta/2$.  The result is now immediate from Corollary~\ref{cor:IJScvgce} (for $n$ large enough) and from $I^{\sn}(1)>0$ $\P^s_n$-a.s.~
 (for remaining $n$ by readjusting $\what$).\end{proof} 

To continue with the proof of Theorem \ref{thm:jointconvergence},  we need the following additional condition.   Part (i) states that for large $n$ it is unlikely that the branch time between two randomly chosen paths is close to one of the lifetimes.   Part (ii) shows it is asymptotically unlikely that two such randomly chosen paths separate early enough but remain close for a while after branching. Part (iii) will ensure that $\kappa_n(W^{\sn})\in C_K$ $\hat\P^s_n$-a.s. (recall $\kappa_n$ from the end of Section~\ref{sec:GST}).
\begin{COND}
\label{cond:tau_stuff}
For every $s,t_0>0$,  
\begin{itemize}
\item[(i)]    
\begin{align}
\lim_{\delta\downarrow0}\limsup_{n\to\infty}\E_n^s\Bigg[\int_0^{t_0}\int_0^{t_0}\int\int \indic{\tau(w_1,w_2)> (u_1\wedge u_2)-\delta}H^{\sn}_{u_1}(\d w_1)H^{\sn}_{u_2}(\d w_2)du_1\,du_2\Bigg]=0,\label{term_b}
\end{align}
\item[(ii)] for every $ \delta>0$,
\begin{align}
\lim_{\vep\downarrow0}\limsup_{n\to\infty}\E_n^s\Bigg[\int_0^{t_0}&\int_0^{t_0}\int\int \indic{|w_{1,\tau(w_1,w_2)+\delta}-w_{2,\tau(w_1,w_2)+\delta}|<\vep}\nn\\
&\times\indic{\tau(w_1,w_2)\le (u_1\wedge u_2)-\delta}
H^{\sn}_{u_1}(\d w_1)H^{\sn}_{u_2}(\d w_2)du_1\,du_2\Bigg]=0.\label{term_a}
\end{align}
\item[(iii)] For all $n$, $\E^s_n[I^{\sn}(\{(u,w):\,w_0\neq o\})]=0$.
\end{itemize}
\end{COND}

\medskip

\begin{THM} \label{thm:Gconvergence}  Assume \eqref{Sncond}
 and Conditions \ref{cond:fdd} and \ref{cond:tau_stuff}.   
Then 
the conclusions, \eqref{rnondeglargen} and \eqref{wconva}, of Theorem~\ref{thm:jointconvergence} hold.
\end{THM}
\begin{proof} We work on the probability space on which \eqref{asconv2} holds. We first fix $i,j$ and show
\begin{equation}\label{tauconvp}\tau(W^{\sn}_i,W^{\sn}_j)\cprobs\tau(W_i,W_j)\ \text{ as }\ n\to\infty.
\end{equation}
If $i=j$ this amounts to $\FL(W^{\sn}_i)\cprobs \FL(W_i)$ which we know by \eqref{asconv2}, so assume $i\neq j$.

Let $\delta>0$.  It suffices to show that
\begin{align}\label{lsctau1}
&\lim_{n\to\infty}\P^{s,K}\big(\tau(W^{\sn}_i,W^{\sn}_j)>\tau(W_i,W_j)+\delta\big)=0, \quad \text{ and }\\
&\label{lsctau} 
 \lim_{n\to\infty}\P^{s,K}\big(\tau(W^{\sn}_i,W^{\sn}_j)+\delta<\tau(W_i,W_j)\big)=0.
\end{align}
Lemma~\ref{lem:uppersemi}(c) and  \eqref{asconv2} 
imply that 
\begin{equation}\label{lsc1}\limsup_{n\to\infty} \tau(W^{\sn}_i,W^{\sn}_j)\le \tau(W_i,W_j),\ \ \P^{s,K}- \text{a.s.}
\end{equation}
Thus,
\begin{align}
\nonumber\limsup_{n\to\infty}\, & \P^{s,K}\big(\tau(W^{\sn}_i,W^{\sn}_j)>\tau(W_i,W_j)+\delta\big)\\
\nonumber&\le\limsup_{n\to\infty} \P^{s,K}\Bigl(\cup_{k=n}^\infty\{\tau(W^{\sss(k)}_i,W^{\sss(k)}_j)>\tau(W_i,W_j)+\delta\}\Bigr)\\
\label{lscbnd}&\le\P^{s,K}\Bigl(\limsup_{n\to\infty} \tau(W^{\sn}_i,W^{\sn}_j)\ge \tau(W_i,W_j)+\delta\Bigr)=0.
\end{align} 
This proves~\eqref{lsctau1}, so it remains to show \eqref{lsctau}.

To prove \eqref{lsctau} it suffices to fix $\eta>0$ and show that for sufficiently small $\delta>0$  and all $n$ sufficiently large, depending on $\delta$,  
\begin{align}
\label{actual} 
\P^{s,K}\big(\tau(W^{\sn}_i,W^{\sn}_j)+\delta<\tau(W_i,W_j)\big)<\eta.
\end{align}
(Here we can restrict to small $\delta$ because the left-hand side is decreasing in $\delta$.)  
By Lemma~\ref{i'}(b)  
and Lemma \ref{Inlt} we can choose $t_0=t_0(\eta)$ large and $\zeta=\zeta(\eta)>0$ small so that 
\begin{equation}\label{tech1}\sup_{n\ge n_0}\P^{s,K}(S^{\sn}\ge t_0/2)+\sup_n\P^s(I^{\sn}(1)\le 2\zeta)<\eta/10.
\end{equation}
Fix such a $t_0$ and $\zeta$.    Then \eqref{asconv2} implies
\begin{equation}\label{tech2}
\P^{s,K}(S\ge t_0)+\P^{s,K}(I(1)\le \zeta)\le\eta/10.
\end{equation}

Recall that $i$ and $j$ are fixed. Let $W_{i,t}\in \R^d$ denote the location of the path $W_i$ at time $t$.  We claim next that for any $n\in \N, \delta, \vep>0$, up to sets of $\P^s$-probability 0, 
\begin{align}
E:=\{&S^{\sn}<t_0,S<t_0\}\cap\Big\{\sup_{t\le t_0}|W^{\sn}_{i,t}-W_{i,t}|<\frac{\vep}{4},\sup_{t\le t_0}|W_{j,t}^{\sn}-W_{j,t}|<\frac{\vep}{4}\Big\}\nonumber\\
\nonumber&\phantom{S^{\sn}<t_0,S<t_0\}}\cap\{\tau(W^{\sn}_i,W^{\sn}_j)+\delta<\tau(W_i,W_j)\}\\
\label{track}&\subset\Big\{\Bigl|W^{\sn}_{i,\tau(W^{\sn}_i,W^{\sn}_j)+\delta}-W^{\sn}_{j,\tau(W^{\sn}_i,W^{\sn}_j)+\delta}\Bigr|<\vep\Big\}\cap\{S^{\sn}<t_0\}.
\end{align}
Assume that the event $E$ on the left-hand side of the inclusion \eqref{track} holds. 
By Lemma~\ref{lem:distlife} 
 and the fact that $J$ is supported on $[0,S]\times\D$ we see that 
\[\tau(W_i^{\sn},W_j^{\sn})+\delta<\tau(W_i,W_j)\le \FL(W_i)\le S<t_0.\]
Therefore $E$ implies that 
\begin{equation}\label{Wnnear}
|W^{\sn}_{k,\tau(W^{\sn}_i,W^{\sn}_j)+\delta}-W_{k,\tau(W^{\sn}_i,W^{\sn}_j)+\delta}|<\vep/4\quad \text{ for } k=i,j.
\end{equation}
In addition, $\tau(W^{\sn}_i,W^{\sn}_j)+\delta<\tau(W_i,W_j)$ implies that 
\[W_{i,\tau(W^{\sn}_i,W^{\sn}_j)+\delta}=W_{j,\tau(W^{\sn}_i,W^{\sn}_j)+\delta}.\]
This and \eqref{Wnnear} allow us to use the triangle inequality to conclude that 
\[\Big|W^{\sn}_{i,\tau(W^{\sn}_i,W^{\sn}_j)+\delta}-W^{\sn}_{j,\tau(W^{\sn}_i,W^{\sn}_j)+\delta}\Big|\le \frac{\vep}{4}+\frac{\vep}{4}<\vep.\]
This proves the claim \eqref{track}.  This latter inclusion implies that for any $\vep,\delta>0$ and $n\in\N$, 
\begin{align}
\nonumber\P^s&\big(\tau(W^{\sn}_i,W^{\sn}_j)+\delta<\tau(W_i,W_j)\big)\\
\nonumber&\le \P^s(S\ge t_0)+\P^s(S^{\sn}\ge t_0)+\P^s(I^{\sn}(1)\le\zeta)+\sum_{k=i,j}\P^s\Big(\sup_{t\le t_0}|W^{\sn}_{k,t}-W_{k,t}|\ge\vep/4\Big)\\
\nonumber&\quad+\P^s\Big(\big|W^{\sn}_{i,\tau(W^{\sn}_i,W^{\sn}_j)+\delta}-W^{\sn}_{j,\tau(W^{\sn}_i,W^{\sn}_j)+\delta}\big|<\vep,S^{\sn}< t_0, I^{\sn}(1)>\zeta\Big)\\
\label{keybound1}&\le\frac{\eta}{2}+\P^s\Big(\big|W^{\sn}_{i,\tau(W^{\sn}_i,W^{\sn}_j)+\delta}-W^{\sn}_{j,\tau(W^{\sn}_i,W^{\sn}_j)+\delta}\big|<\vep,S^{\sn}< t_0, I^{\sn}(1)>\zeta\Big)
\end{align}
where the last line holds for 
$n\ge n_0\vee n_1(\vep)$ by \eqref{asconv2} (recall $W_k$ is continuous a.s.), \eqref{tech1} and \eqref{tech2}.  Let $p_n=p_n(\delta,\vep)$  
denote the probability arising on the right-hand side of \eqref{keybound1}. 

By definition of $W^{\sn}_k$, we have that
\begin{align}
\nonumber p_n&\le \E_n^s\Bigg[\dfrac{\int_0^{t_0}\int_0^{t_0}\int\int \indic{|w_{1,\tau(w_1,w_2)+\delta}-w_{2,\tau(w_1,w_2)+\delta}|<\vep}H^{\sn}_{u_1}(\d w_1)H^{\sn}_{u_2}(\d w_2)\d u_1\d u_2}{I^{\sn}(1)^2}\indic{I^{\sn}(1)>\zeta}\Bigg]\\ 
&\le \E_n^s\Bigg[\int_0^{t_0}\int_0^{t_0}\zeta^{-2}\int\int \indic{|w_{1,\tau(w_1,w_2)+\delta}-w_{2,\tau(w_1,w_2)+\delta}|<\vep}H^{\sn}_{u_1}(\d w_1)H^{\sn}_{u_2}(\d w_2)\d u_1 \d u_2\Bigg].\label{whoops1}
\end{align}
The expectation above is at most 
\begin{align}
&\zeta^{-2}\E_n^s\Bigg[\int_0^{t_0}\int_0^{t_0}\int\int \indic{|w_{1,\tau(w_1,w_2)+\delta}-w_{2,\tau(w_1,w_2)+\delta}|<\vep}\indic{\tau(w_1,w_2)\le (u_1\wedge u_2)-\delta}
H^{\sn}_{u_1}(\d w_1)H^{\sn}_{u_2}(\d w_2)\,du_1du_2\Bigg]\label{term1}\\
&+\zeta^{-2}\E_n^s\Bigg[\int_0^{t_0}\int_0^{t_0}\int\int \indic{\tau(w_1,w_2)> (u_1\wedge u_2)-\delta}H^{\sn}_{u_1}(\d w_1)H^{\sn}_{u_2}(\d w_2)\,du_1du_2\Bigg].\label{term2}
\end{align}
By Condition \ref{cond:tau_stuff}(i) there is a $\delta_0>0$ so that for each $0<\delta\le \delta_0$ there is an $n_2=n_2(\delta)\in\N$ for which \eqref{term2} is at most $\eta/4$ for $n\ge n_2$.  By Condition \ref{cond:tau_stuff}(ii) for each $\delta$ as above there is an $\vep_1=\vep_1(\delta)$ so that for $0<\vep\le \vep_1$ there is an $n_3=n_3(\delta,\vep)$ for which \eqref{term1} is at most $\eta/4$ for $n\ge n_3$. So for $0<\delta\le \delta_0$ we may set $\vep=\vep_1(\delta)$ and then conclude that for $n\ge n_4(\delta):=n_0\vee n_1(\vep_1(\delta))\vee n_2(\delta)\vee n_3(\delta,\vep_1(\delta))$, we have \eqref{keybound1}  with $p_n<\eta/2$.   This proves \eqref{actual} and the proof of \eqref{tauconvp} is complete.
  
 In view of \eqref{asconv2} and \eqref{tauconvp}, to complete the proof of \eqref{wconva} it suffices to show that 
\begin{equation}\label{shapecp}
\mc{B}_K(\kappa_n(W^{\sn}))\cprobs\mc{B}_K(W).
\end{equation}
To prove this, first apply Lemma~\ref{interpol} (use the convergences in \eqref{asconv2} and \eqref{tauconvp}) to see that
\begin{equation}\label{kapconv}
(\kappa_n(W^{\sn}),\bs{\tau}(\kappa_n(W^{\sn})))\cprobs (W,\bs{\tau}(W)).
\end{equation}
Now use Lemma~\ref{lem:hBM_shape0} and the above to apply Proposition~\ref{shapecont} with $w^{\sn}=\kappa_n(W^{\sn})$ and $w=W$ to establish \eqref{shapecp} and also \eqref{rnondeglargen}, thus completing the proof.
\end{proof} 

\begin{THM}\label{thm:Grjointconvergence} 
Assume \eqref{Sncond}, Conditions~\ref{cond:fdd} and \ref{cond:tau_stuff}, and the scaling relation \eqref{rescaleHLT}.
Then the conclusions, \eqref{nondeglargen} and \eqref{wconvb}, of Theorem~\ref{thm:rjointconvergence} hold.
\end{THM}
\begin{proof} By \eqref{rescaleHLT} and a short calculation we have for measurable $\phi:\oD\to [0,\infty)$,
\begin{equation}\label{Iints}
\int \phi(u,w)I^{\sn}(du,dw)=n^{-2}\int \phi\circ\bar\rho_n(u,w)I^{\sss(1)}(du,dw),
\end{equation}
\begin{equation}\label{Jints}
\int \phi(u,w)J^{\sn}(du,dw)=\int \phi\circ\bar\rho_n(u,w)J^{\sss(1)}(du,dw),
\end{equation}
and
\begin{equation}\label{Sresc}
S^{\sn}=S^{\sss(1)}/n,\text{ and so }\P_n^s=\P_1^{ns}.
\end{equation}
Let $\psi_1,\psi_2$ be non-negative measurable functions on $\mc{M}_F(\mc{\oD})^2\times\R_+$ and $\D^K$, respectively. By the definition of $\hat\E^s_n=\hat\E^s_n(J^{\sn})$ and \eqref{Iints}--\eqref{Sresc} we have
\begin{align*}
\hat\E&^s_n\big[\psi_1(I^{\sn},J^{\sn},S^{\sn})\psi_2(W^{\sn}_1,\dots,W_K^{\sn})\big]\\
&=\E^s_n\Bigl[\psi_1(I^{\sn},J^{\sn},S^{\sn})\int\dots\int\psi_2(w)J^{\sn}(du_1,dw_i)\dots J^{\sn}(du_K,dw_K)\Bigr]\\
&=\E^{ns}_1\Bigl[\psi_1(n^{-2}I^{\sss(1)}\circ\bar\rho_n^{\ -1},J^{\sss(1)}\circ\bar\rho_n^{\ -1},S^{\sss(1)}/n)\int\dots\int \psi_2(\rho_n(w))\,J^{\sss(1)}(du_1,dw_1)\dots J^{\sss(1)}(du_K,dw_K)\Bigr]\\
&=\hat\E_1^{ns}\big[\psi_1(n^{-2}I^{\sss(1)}\circ\bar\rho_n^{\ -1},J^{\sss(1)}\circ\bar\rho_n^{\ -1},S^{\sss(1)}/n)\psi_2(\rho_n(W_1^{\sss(1)}),\dots,\rho_n(W_K^{\sss(1)}))\big].
\end{align*}
Therefore (recall $W^{\sn}=(W^{\sn}_1,\dots,W^{\sn}_K)$)
\begin{align}\label{scaleprob}
\hat\P^s_n\big((I^{\sn},J^{\sn},S^{\sn},W^{\sn})\in\cdot\big)=\hat \P^{sn}_1\big((n^{-2}I^{\sss(1)}\circ\bar\rho_n^{\ -1},J^{\sss(1)}\circ\bar\rho_n^{\ -1},S^{\sss(1)}/n,\rho_n(W^{\sss(1)}))\in\cdot\big)
\end{align}
It is easy to check
\begin{equation}\label{kapparho}
\kappa_n\circ\rho_n(w)=\rho_n\circ\kappa_1(w)\quad\forall w\in\D^K.
\end{equation}
By the above and Lemma~\ref{lem:rescBK}, if $\mc{B}_{n,K}$ is as in \eqref{BnKdef}, then for all $w^{\sss(1)}\in D_K$, 
\begin{align}\label{BKrescal}
\mc{B}_K(\kappa_n\circ\rho_n(w^{\sss(1)}))=\mc{B}_K(\rho_n\circ\kappa_1(w^{\sss(1)}))
=\mc{B}_{n,K}(\kappa_1(w^{\sss(1)})).
\end{align} 
We also have from \eqref{kapparho} and \eqref{resctau} that for all $w^{\sss(1)}\in D_K$,
\begin{equation}\label{taukapterm}
\bs{\tau}(\kappa_n\circ\rho_n(w^{\sss(1)}))=\bs{\tau}(\rho_n(\kappa_1(w^{\sss(1)})))=\frac{\bs{\tau}(\kappa_1(w^{\sss(1)}))}{n}.
\end{equation}
Now use  \eqref{scaleprob}, \eqref{taukapterm}, and \eqref{BKrescal} to conclude that
\begin{align}\label{scaequiv}
\hat\P&^s_n\Bigl(\big(I^{\sn},J^{\sn},S^{\sn}, W^{\sn},\bs{\tau}(\kappa_n(W^{\sn})),\mc{B}_K(\kappa_n(W^{\sn}))\big)\in\cdot\Bigr)\\
\nonumber&=\hat\P^{ns}_1\Bigl(\big(n^{-2}I^{\sss(1)}\circ\bar\rho_n^{\ -1},J^{\sss(1)}\circ\bar\rho_n^{\ -1},S^{\sss(1)}/n,\rho_n(W^{\sss(1)}),\bs{\tau}(\kappa_1(W^{\sss(1)}))/n,\mc{B}_{n,K}(\kappa_1(W^{\sss(1)}))\big)\in\cdot\Bigr),
\end{align}
as laws on $\MF(\oD)^2\times\R_+\times\D^K\times\R_+^{K^2}\times\TT_{gst}$.  
By \eqref{kapconv} and \eqref{tauconvp} we may replace $\bs{\tau}(W^{\sn})$ with $\bs{\tau}(\kappa_n(W^{\sn}))$ in the fifth component of the weakly convergent vector in Theorem~\ref{thm:Gconvergence} (that is, in \eqref{wconva}) and still get the weak convergence in Theorem~\ref{thm:Gconvergence}. That is, the left-hand side of \eqref{scaequiv} converges weakly to the historical Brownian limit law in Theorem~\ref{thm:Gconvergence}, and hence so does the right-hand side of \eqref{scaequiv}. This proves \eqref{wconvb}. The equivalence in law of the $\bs{\tau}$ components in \eqref{scaequiv} implies,
\begin{align*}\hat\P^{ns}_1(\kappa_1(W^{\sss(1)})\text{ is non-degenerate})&=\hat\P^{ns}_1\Bigl( \frac{\bs{\tau}(\kappa_1(W^{\sss(1)}))}{n}\text{ is non-degenerate}\Bigr)\\
&=\hat\P_n^s(\bs{\tau}(\kappa_n(W^{\sn}))\text{ is non-degenerate})\to1\text{ as }n\to\infty,
\end{align*}
where \eqref{rnondeglargen} from Theorem~\ref{thm:Gconvergence} is used in the last convergence. This gives \eqref{nondeglargen} and the proof is complete.
\end{proof}
Theorems~\ref{thm:jointconvergence} and \ref{thm:rjointconvergence} will now follow from Theorems~\ref{thm:Gconvergence} and \ref{thm:Grjointconvergence}, respectively, once we establish \eqref{rescaleHLT}, \eqref{Sncond}, and Conditions~\ref{cond:fdd} and \ref{cond:tau_stuff} for lattice trees with $d>8$ and $L$ sufficiently large, and for BRW.  We have already 
seen that \eqref{rescaleHLT}, \eqref{Sncond} and Condition~\ref{cond:fdd} hold 
 in both cases, Condition~\ref{cond:tau_stuff}(iii) is obvious, and so it remains to establish Condition~\ref{cond:tau_stuff}(i),(ii) in both cases.

\begin{REM}\label{scalecondition} Under \eqref{rescaleHLT} the following slightly simpler conditions  imply Condition~\ref{cond:fdd}:
\begin{itemize}
\item[(i)'] $\lim_{t\to\infty}t\P(S^{\sss(1)}>t)=c_0\in(0,\infty)$,
\item[(ii)'] For each $s>0$, $\P_1^{sn}(H^{\sn}\in\cdot)\cfdd\N_{\sss H}^s(H\in\cdot)$ as $n\to\infty$,
\item[(iii)'] For each  $p\in\N$, $\sup_{k\in\Z_+}\E[(H^{\sss(1)}_k(1))^p]/(k^{p-1}\vee 1)<\infty$.
\item[(iv)'] For each $k\in\Z_+$, $H_k^{\sss(1)}(\{w:\FL(w)\neq k\})=0\ \  \P-a.s.$
\end{itemize}
Condition~\ref{cond:fdd}(i) follows easily from (i)' and Lemma~\ref{i'}(a). The other derivations are easy.
\Enddef
\end{REM}

\subsection{Verification of Condition \ref{cond:tau_stuff}(i),(ii) for lattice trees}\label{sec:cond18lt}

\blank{
\todo[inline]{Old stuff:  By Lemma \ref{lem:error} we have that \eqref{term2} is at most 
\[\zeta^{-2}\frac{Cn}{n^2}\Big[ \floor{n(u_1\wedge u_2)}-\floor{n((u_1\wedge u_2)-\delta)}\Big]\le \frac{C'(\zeta)}{n}(n\delta+1)= C'(\delta+\frac{1}{n}),\]
which is less than $\eta/4$ if $\delta<\eta/(10C')$ and $n\ge n_2(\eta)$.

 From \eqref{histprod} and Lemma~\ref{lem:LTbranch}, we have for any $\delta'=\delta'(s,\zeta,t_0,\eta)=\delta'(\eta)$ there is an $\vep=\vep(\delta')>0$ and $n_3=n_3(\delta,\delta',\vep)\in\N$ so that for $n\ge n_3$, \eqref{term1}  above is bounded by 
$Cs\zeta^{-2}n^{-1}\sum_{m=0}^{n(u_1\wedge u_2)}\delta'$,
whence \eqref{whoops1} is at most
\begin{align*}
& Cs\zeta^{-2}\int_0^{t_0}\int_0^{t_0}n^{-1}\sum_{m=0}^{n(u_1\wedge u_2)}\delta'\d u_1\,\d u_2\le Cs\zeta^{-2}t_0^2(t_0+1)\delta'<\eta/4,
\end{align*}
where in the last line we have chosen $\delta'=\delta'(s,\zeta,t_0,\eta)>0$ small enough.  Use the above in \eqref{keybound1} to derive \eqref{lsctau} and so complete the proof.}
}
 
We work in the setting of lattice trees described in Section~\ref{sec:LT} with $d>8$ and $L$ sufficiently large.  In particular Condition~\ref{cond:fdd} holds.
For any Borel $A\subset [0,\infty)$ and $u_1,u_2>0$ we have 
\begin{align}
&\big(H_{u_1}^{\sss(n)}\times H_{u_2}^{\sss(n)}\big)\big(\tau(w_1,w_2)\in A\big)=\frac{1}{(C_0n)^2}\sum_{y_1\in \mc{T}_{nu_1}}\sum_{y_2\in \mc{T}_{nu_2}}\indic{\tau(w^{\sn}(u_1,y_1),w^{\sn}(u_2,y_2))\in A}. \label{Hprod1}
\end{align}
For $y_1,y_2\in T$,   
we use the notation $\tau_{y_1,y_2}(T)$  
to represent the branch time for the paths in $T$ 
from the origin to $y_1$ and $y_2$ respectively.
Taking the expectation of \eqref{Hprod1} 
 we obtain
\begin{align}
\label{histprod}&\E^s_n\Big[\big(H_{u_1}^{\sss(n)}\times H_{u_2}^{\sss(n)}\big)\big(\tau(w_1,w_2)\in A\big)\Big]\\
\nonumber&=(n\P(S^{\sn}>s))^{-1}\frac{C_0^{-2}}{n} 
\sum_{x_1,x_2\in \Z^d}\sum_{T\ni o}W(T)\indic{T_{ns}\ne \varnothing}\indic{x_1\in T_{nu_1},x_2\in T_{nu_2}}\indic{\tau_{x_1,x_2}(T)\in nA}
\\
&
\le \frac{C(s\vee 1)}{n}\sum_{x_1,x_2\in\Z^d}\sum_{T\ni o}W(T)\indic{x_1\in T_{nu_1},x_2\in T_{nu_2}}\indic{\tau_{x_1,x_2}(T)\in nA},\label{yep1}
\end{align}
where $nA=\{na:a\in A\}$, and we have  used $\P(S^{\sn}>s)\ge c(s\vee 1)^{-1}$ for all $n$ and some $c>0$ (by equations (1.22) and (1.27) in \cite{HP19}) in the last line.

We will use the following Lemma to verify Condition~\ref{cond:tau_stuff}(i) for lattice trees.

\begin{LEM}
\label{lem:error}
There exists a constant $C(d,L)$ such that for all  $0\le m_1\le m_2\in \Z_+$, and every $u_1,u_2\ge 0$, $n \in \N$, 
\[\sum_{x_1,x_2\in\Z^d}\sum_{T\ni o}W(T)\indic{x_1\in T_{nu_1},x_2\in T_{nu_2}}\indic{\tau_{x_1,x_2}(T)\in [m_1+1,m_2]}\le C(m_2-m_1).\]
\end{LEM}

\begin{proof}
Let $m_2^*=\min\{m_2,\floor{nu_1},\floor{nu_2}\}$, and write the left hand side of the claimed bound as 
\begin{align}
\sum_{x_0 \in \Z^d}\sum_{m=m_1+1}^{m_2^*}\sum_{x_1,x_2}\sum_{T\ni o}W(T)\indic{x_1\in T_{nu_1},x_2\in T_{nu_2}}\indic{\tau_{x_1,x_2}(T)=m}\indic{w_m(nu_1, x_1)=x_0}.
\end{align} 
We can write $T$ as a union of $T^{[0]}$ containing $o$ and $x_0$, $T^{[1]}$ containing $x_0$ and $x_1$, $T^{[2]}$ containing $x_0$ and $x_2$, that are all pairwise disjoint except for the common point $x_0$ that is in all of them.
 Replace the sum over $T$ with a sum over these various trees and include the indicator $\indic{I'_{0,1,2}}$ that they are ``almost'' disjoint as above.  Then this becomes
\begin{align}\label{Wtriple}
&\sum_{m=m_1+1}^{m_2^*}\sum_{x_0}\sum_{\stackrel{T^{[0]}\ni o:}{x_0 \in T^{[0]}_m}}W(T^{[0]})
\sum_{x_1}\sum_{\stackrel{T^{[1]}\ni x_0:}{x_1 \in T^{[1]}_{nu_1-m}}}W(T^{[1]})
\sum_{x_2}\sum_{\stackrel{T^{[2]}\ni x_0:}{x_2 \in T^{[2]}_{nu_2-m}}}W(T^{[2]})
\indic{I'_{0,1,2}}.
\end{align}
Since all summands are non-negative, we may drop $\indic{I'_{0,1,2}}$ to get an upper bound.  Once this indicator is removed we may perform the sum over $x_1$ and $x_2$, and then $x_0$ to get three terms (up to translation) of the form $\sum_x \sum_{\stackrel{T \ni o:}{x\in T_{j}}}W(T)=\rho\E[|\mc{T}_{j}|]$, which is bounded by $K$, uniformly in $j\ge 0$ (recall \eqref{pop_size_converge}).
  Thus \eqref{Wtriple} is at most $\sum_{m=m_1+1}^{m_2}K^3\le C(m_2-m_1)$.
\end{proof}
Let us now verify Condition~\ref{cond:tau_stuff}(i) for lattice trees (for $d>8$ and $L$ sufficiently large).  Fix $s,t_0>0$.   Condition \ref{cond:fdd}(iv) gives an upper limit on $\tau$, and so \eqref{yep1} shows  
the expression, $\E_n^s\Big[\int\int \indic{\tau(w_1,w_2)> (u_1\wedge u_2)-\delta}H^{\sn}_{u_1}(\d w_1)H^{\sn}_{u_2}(\d w_2)\Big]$, appearing in \eqref{term_b}, is at most
\begin{align}
\frac{C(s\vee 1)}{n}\sum_{x_1,x_2}\sum_{T\ni o}W(T)\indic{x_1\in T_{nu_1},x_2\in T_{nu_2}}\indic{\tau_{x_1,x_2}(T)\in (n(u_1\wedge u_2)-n\delta,n(u_1\wedge u_2)]}.
\end{align}
By Lemma \ref{lem:error} this is at most
\[\frac{C(s\vee1)}{n}\Big[n(u_1\wedge u_2)-n((u_1\wedge u_2)-\delta)+1\Big]\le C_s(\delta+\frac{1}{n}).\] 
The constant is independent of $u_1$ and $u_2$, and so we may integrate the above bound over $u_i\le t_0$ and let $n\to\infty$ and then $\delta\downarrow0$ to derive \eqref{term_b}, as required.

We now turn our attention to Condition \ref{cond:tau_stuff}(ii).  Note that for any $u_1,u_2\ge 0$, $\vep,\delta>0$, $n \in \N$ 
\begin{align*}
&\big(H_{u_1}^{\sss(n)}\times H_{u_2}^{\sss(n)}\big)\big(\{(w,w'):|w_{\tau(w,w')+\delta}-w'_{\tau(w,w')+\delta}|<\vep,\quad \tau(w,w')\le (u_1\wedge u_2)-\delta\}\big)\\
&=\frac{1}{(C_0n)^2}\sum_{x_1\in \mc{T}_{nu_1}}\sum_{x_2\in \mc{T}_{nu_2}}\indic{|w^{\sn}_{(\tau_{x_1,x_2}(\mc{T})/n)+\delta}(u_1,x_1)-w^{\sn}_{(\tau_{x_1,x_2}(\mc{T})/n)+\delta}(u_2,x_2)|<\vep,\ \  \tau_{x_1,x_2}(\mc{T})/n\le (u_1\wedge u_2)-\delta}.
\end{align*}
Taking expectation with respect to $\E^s_n$ we obtain
\begin{align}
\label{histprod*}&\E^s_n\Big[\big(H_{u_1}^{\sss(n)}\times H_{u_2}^{\sss(n)}\big)\big(\{(w,w'):|w_{\tau(w,w')+\delta}-w'_{\tau(w,w')+\delta}|<\vep,\ \  \tau(w,w')\le (u_1\wedge u_2)-\delta\}\big)\Big]\\
\nonumber&=(n\P(S^{\sn}>s))^{-1}\frac{C_0^{-2}}{n}\sum_{m=0}^{\lfloor n(u_1\wedge u_2)-n\delta\rfloor}\sum_{x_1,x_2}\sum_{T\ni o}W(T)\indic{T_{ns}\ne \varnothing}\indic{x_1\in T_{nu_1},x_2\in T_{nu_2}}\indic{\tau_{x_1,x_2}(T)=m}\\
\nonumber&\phantom{(n\P(S^{\sn}>s))^{-1}\frac{c}{n}\sum_{m=0}^{n(t_1\wedge t_2)}\sum_{x_1,x_2}\sum_{T\ni o}W(T)\indic{T_{ns}\ne \varnothing}}\times\indic{|w_{m+n\delta}(nu_1,x_1)-w_{m+n\delta}(nu_2,x_2)|<\sqrt{n}\vep}\\
\nonumber&
\le \frac{C(s\vee 1)}{n}\sum_{m=0}^{\lfloor n(u_1\wedge u_2)-n\delta\rfloor}\sum_{x_1,x_2}\sum_{T\ni o}W(T)\indic{x_1\in T_{nu_1},x_2\in T_{nu_2}}\indic{\tau_{x_1,x_2}(T)=m}\indic{|w_{m+n\delta}(nu_1,x_1)-w_{m+n\delta}(nu_2,x_2)|<\sqrt{n}\vep},
\end{align}
where we have again used equations (1.22) and (1.27) in \cite{HP19} in the last line.  

\begin{LEM}
\label{lem:LTbranch}
For any 
$\delta,\delta'>0$, there is an $\vep=\vep(\delta')>0$ and an $n_0=n_0(\delta,\delta')$ such that for every $u_1,u_2\ge 0$: \hspace{.5cm} for every $m\in \Z_+$ such that $m\le n((u_1\wedge u_2) - \delta)$, 
\[\sup_{n\ge n_0}\sum_{x_1,x_2}\sum_{T\ni o}W(T)\indic{x_1\in T_{nu_1},x_2\in T_{nu_2}}\indic{\tau_{x_1,x_2}(T)=m}\indic{|w_{m+n\delta}(nu_1,x_1)-w_{m+n\delta}(nu_2,x_2)|<\sqrt{n}\vep}<\delta'.\]
\end{LEM}
\begin{proof}
If $\delta>u_1\wedge u_2$ then no claim is being made. 

Otherwise, the argument of the supremum can be written as  
\begin{align}\nn
&\sum_{x_0}\sum_{x_1,x_2}\sum_{y_1,y_2}\sum_{T\ni o}\indic{x_1\in T_{nu_1},x_2\in T_{nu_2}}W(T)\indic{\tau_{x_1,x_2}(T)=m}\indic{w_m(nu_1, x_1)=x_0}\\
\label{supargu}&\phantom{\sum_{x_0}\sum_{x_1,x_2}\sum_{y_1,y_2}\sum_{T\ni o}\indic{x_1,x_2\in T_{m+n\delta}}}
\indic{w_{m+n\delta}(nu_1, x_1)=y_1}\indic{w_{m+n\delta}(nu_2, x_2)=y_2}\indic{|y_1-y_2|<\sqrt{n}\vep}.
\end{align}
We can write $T$ as a union of $T^{[0]}$ containing $o$ and $x_0$, $T^{[1,1]}$ containing $x_0$ and $y_1$, $T^{[1,2]}$ containing $y_1$ and $x_1$, $T^{[2,1]}$ containing $x_0$ and $y_2$, $T^{[2,2]}$ containing $y_2$ and $x_2$ that are all pairwise disjoint except for their common start/end points (e.g.~$T^{[1,2]}\cap T^{[1,1]}=\{y_1\}$).  Replace the sum over $T$ with a sum over these various trees and include the indicator $\indic{I_{0,1,2}}$ that they are ``almost'' disjoint as above.  Then this becomes
\begin{align}\nn
&\sum_{x_0}\sum_{\stackrel{T^{[0]}\ni o:}{x_0 \in T^{[0]}_m}}W(T^{[0]})
\sum_{y_1}\sum_{\stackrel{T^{[1,1]}\ni x_0:}{y_1 \in T^{[1,1]}_{n\delta}}}W(T^{[1,1]})
\sum_{y_2}\sum_{\stackrel{T^{[2,1]}\ni x_0:}{y_2 \in T^{[2,1]}_{n\delta}}}W(T^{[2,1]})\indic{|y_1-y_2|<\sqrt{n}\vep}\\
&\phantom{=\sum_{x_0,x_1,x_2,y_1,y_2}}\sum_{x_1}
\sum_{\stackrel{T^{[1,2]}\ni y_1:}{x_1 \in T^{[1,2]}_{n(u_1-\delta)-m}}}W(T^{[1,2]})
\sum_{x_2}\sum_{\stackrel{T^{[2,2]}\ni y_2:}{x_2 \in T^{[2,2]}_{n(u_2-\delta)-m}}}W(T^{[2,2]})
\indic{I_{0,1,2}}.\label{abodis}
\end{align}
Since all summands are non-negative, we may drop $\indic{I_{0,1,2}}$ to get an upper bound.  Once this indicator is removed we may perform the sum over $(x_1,T^{[1,2]})$ and $(x_2,T^{[2,2]})$ to get two terms (up to translation) of the form $\sum_x \sum_{\stackrel{T \ni o:}{x\in T_{t}}}W(T)=\rho\E[|\mc{T}_{t}|]$, which is bounded by a universal constant $K$ (recall \eqref{pop_size_converge}), uniformly in $t\ge 0$.  Thus \eqref{abodis} is at most $K^2$ times 
\begin{align*}
&\sum_{x_0,y_1,y_2}\sum_{\stackrel{T^{[0]}\ni o:}{x_0 \in T^{[0]}_m}}W(T^{[0]})
\sum_{\stackrel{T^{[1,1]}\ni x_0:}{y_1 \in T^{[1,1]}_{n\delta}}}W(T^{[1,1]})
\sum_{\stackrel{T^{[2,1]}\ni x_0:}{y_2 \in T^{[2,1]}_{n\delta}}}W(T^{[2,1]})\indic{|y_1-y_2|<\sqrt{n}\vep}.
\end{align*}
This is equal to 
\begin{align}
\nonumber& \rho^3 \sum_{x_0}\P(x_0 \in \mc{T}_{m})
 \sum_{y_1,y_2}\indic{|y_1-y_2|<\sqrt{n}\vep}\P(y_1-x_0 \in \mc{T}_{n\delta})\P(y_2-x_0 \in \mc{T}_{n\delta})\\
\nonumber&=\rho^3 \sum_{x_0}\P(x_0 \in \mc{T}_{m})
 \sum_{z_1,z_2}\indic{|z_1-z_2|<\sqrt{n}\vep}\P(z_1 \in \mc{T}_{n\delta})\P(z_2 \in \mc{T}_{n\delta})\\
\nonumber&\le C\sum_{z_1,z_2}\indic{|z_1-z_2|<\sqrt{n}\vep}\P(z_1 \in \mc{T}_{n\delta})\P(z_2 \in \mc{T}_{n\delta})\\
&\le C'\E[|\mc{T}_{n\delta}|]^{-2}\sum_{z_1,z_2}\indic{|z_1-z_2|<\sqrt{n}\vep}\P(z_1 \in \mc{T}_{n\delta})\P(z_2 \in \mc{T}_{n\delta}),\label{muahahaha}
\end{align}
where $C'$ is a universal constant (i.e.~it doesn't depend on $\delta, \delta', t_1,t_2, m,n$) and we have again used the fact that $\sup_{m\ge 0}\E[|\mc{T}_m|]<\infty$.  Define $\nu_{n,1}(A)=\E[|\mc{T}_{\delta n}|]^{-1}\sum_{z\in \Z^d} \indic{n^{-1/2}z\in A}\P(z \in \mc{T}_{\delta n})$ for Borel $A\subset \R^d$, and for Borel $R\subset\R^d\times\R^d$ define 
\[\nu_{n}(R)=\nu_{n,1}\times\nu_{n,1}(R)=\E[|\mc{T}_{\delta n}|]^{-2}\sum_{z_1,z_2\in \Z^d} \indic{n^{-1/2}(z_1,z_2)\in R}\P(z_1\in \mc{T}_{\delta n})\P(z_2\in \mc{T}_{\delta n}).\]
Then $\nu_n$ (dependence on $\delta$ is suppressed) is the law of a pair of iid random vectors $(Z_{n,1},Z_{n,2})\in \R^d \times \R^d$, each with law $\nu_{n,1}$. 

We know from \cite[Theorem 1.12]{H08} that $\int e^{i k\cdot z}\d\nu_{n,1}(z)=\E[|\mc{T}_{\delta n}|]^{-1}\sum_{z\in \Z^d}e^{\frac{ik \cdot z}{\sqrt{n}}}\P(z\in \mc{T}_{\delta n}) \to e^{-k^2\delta\frac{\sigma_0^2}{2}}$ 
(in the notation of \cite{H08} we have $\sigma_0^2:=c_0/d$).   
Therefore $(Z_{n,1},Z_{n,2})$ converges weakly to a Gaussian vector $(Z_1,Z_2)$ with $Z_1,Z_2$ independent, each with characteristic function $e^{- k^2\delta \sigma_0^2/2}$.  Thus the bound in \eqref{muahahaha} equals
\begin{align*}
&C'\E[|\mc{T}_{n\delta}|]^{-2}\sum_{z_1,z_2}\indic{|z_1-z_2|<\sqrt{n}\vep}\P(z_1 \in \mc{T}_{n\delta})\P(z_2 \in \mc{T}_{n\delta})\\
&=C'P(|Z_{n,1}-Z_{n,2}|<\vep)\to C'P(|Z_1-Z_2|<\vep).
\end{align*}
Choose $\vep>0$ sufficiently small (depending only on $\delta'$) so that $P(|Z_1-Z_2|<\vep)<\delta'/(2C')$.  Next, choose $n_0(\delta,\delta',\vep(\delta'))=n_0(\delta,\delta')$ sufficiently large so that 
\[\text{$P(|Z_{n,1}-Z_{n,2}|<\vep)< P(|Z_1-Z_2|<\vep)+\delta'/(2C')$ for all $n\ge n_0$.}\]  
Then for all $n\ge n_0$, \eqref{muahahaha} is at most
\[C' \Bigg[P(|Z_1-Z_2|<\vep)+\frac{\delta'}{2C'}\Bigg]<\delta'.\] 
Our earlier bounds, now show that for all $n\ge n_0$, \eqref{supargu} is also bounded by $\delta'$, 
as required.
\end{proof}

We can now verify Condition \ref{cond:tau_stuff}(ii) for lattice trees (with $d>8$ and $L$ sufficiently large).  Fix $s,t_0,\delta>0$.  For any $\vep>0$, $u_1,u_2\in [0,t_0]$ we have from \eqref{histprod*} that the 
expression, $\E_n^s\Big[\big(H_{u_1}^{\sss(n)}\times H_{u_2}^{\sss(n)}\big)\big(\{(w,w'):|w_{\tau(w,w')+\delta}-w'_{\tau(w,w')+\delta}|<\vep,\  \tau(w,w')\le (u_1\wedge u_2)-\delta\}\big)\Big]$, in \eqref{term_a}  is at most
\begin{align}\label{exprint}
\frac{C(s\vee 1)}{n}\sum_{m=0}^{n(u_1\wedge u_2)-n\delta} \sum_{x_1,x_2}\sum_{T\ni o}W(T)&\indic{x_1\in T_{nu_1},x_2\in T_{nu_2}}\indic{\tau_{x_1,x_2}(T)=m}
\indic{|w_{m+n\delta}(nu_1,x_1)-w_{m+n\delta}(nu_2,x_2)|<\sqrt{n}\vep}.
\end{align}
Let $\delta'>0$ and use Lemma \ref{lem:LTbranch}  to find $\vep(\delta')$ and $n_0(\delta')$ so that for $\vep\le \vep(\delta')$ and $n\ge n_0(\delta')$, \eqref{exprint} is at most $C(s\vee1) (u_1\wedge u_2)\delta'$ uniformly in $u_i\le t_0$.  Now integrate out $u_i\in[0,t_0]$ ($i=1,2$) in this bound and let $n\to\infty$, $\vep\downarrow 0$ and $\delta'\downarrow 0$ (in that order)  to verify Condition~\ref{cond:tau_stuff}(ii).

\subsection{Verification of Condition \ref{cond:tau_stuff} for branching random walk}\label{sec:cond18brw}

We consider the branching random walk setting from Section~\ref{sec:BRW}. If $\alpha,\beta$ are labels in $I$ let $\alpha\wedge \beta$ denote their greatest common ancestor in the tree order on $I$. For $t\ge 0$ let $I_t=I_{\lfloor t\rfloor}$. Recall that $s>0$ is fixed throughout.  We first bound the mean of the integral inside the time integrals on  the left hand side of \eqref{term_b}, where we may assume $0\le u_1\le u_2$, and $n\in\N$, $\delta>0$ are fixed for now. We have from Condition~\ref{cond:fdd}(i), for a universal constant $c$,
\begin{align}
\nn\E^s_n&\Bigl[\int\int \1_{\{\tau(w_1,w_2)>u_1-\delta\}}H^{\sn}_{u_1}(dw_1)H^{\sn}_{u_2}(dw_2)\Bigr]\\
\nn&\le \frac{c}{n}\E\Bigl[\sum_{\alpha_1\in I_{nu_1}}\1_{\{\alpha_1\in \mc{T}_{nu_1}\}}\sum_{\alpha_2\in I_{nu_2}}\1_{\{\alpha_2\in \mc{T}_{nu_2}\}}\1_{\{\tau(w(nu_1,\alpha_1),w(nu_2,\alpha_2))>n(u_1-\delta)\}}\Bigr]\\
\nn&=\frac{c}{n}\Bigg(\E\Big[\sum_{m=\lfloor (nu_1-2n\delta)^+\rfloor}^{\lfloor nu_1\rfloor}\sum_{\alpha_1\in I_{nu_1}}\sum_{\alpha_2\in I_{nu_2}}\1_{\{\alpha_i\in \mc{T}_{nu_i},\,i=1,2, \,|\alpha_1\wedge \alpha_2|=m\}}\1_{\tau(w(nu_1,\alpha_1),w(nu_2,\alpha_2))>n(u_1-\delta)\}}\Big]\\
\nn&\quad+\E\Big[\sum_{\alpha_1\in I_{nu_1}}\sum_{\alpha_2\in I_{nu_2}}\1_{\{\alpha_i\in \mc{T}_{nu_i},\,i=1,2, |\alpha_1\wedge\alpha_2|<\lfloor (nu_1-2n\delta)^+\rfloor\}}\1_{\{\tau(w(nu_1,\alpha_1),w(nu_2,\alpha_2))-|\alpha_1\wedge\alpha_2|> \lceil n\delta\rceil\}}\Big]\Bigg)\\
&:=\frac{c}{n}[E_1+E_2].\label{Edecomp}
\end{align}
If $\beta_0\in I_j$ and $\beta_1\in\N^k$ let $\beta_0\vee\beta_1\in I_{j+k}$ be the concatenation of $\beta_0$ and $\beta_1$.  Consider $\alpha_1,\alpha_2$ and let $m=|\alpha_1\wedge \alpha_2|$, as in the sum contributing to $E_1$. 
We can then set $\beta_0=\alpha_1\wedge \alpha_2\in I_m$ and find unique $\beta_j\in \N^{\lfloor nu_j\rfloor-m}$ (with differing first coordinates) so that $\alpha_j=\beta_0\vee\beta_j$ for $j=1,2$. Note that 
$$\P(\alpha_j\in \mc{T}_{nu_j}\text{ for }j=1,2)=\P(\beta_0\in \mc{T}_m)\P(0\vee\beta_1\in\mc{T}_{nu_1-m})\P(0\vee\beta_2\in\mc{T}_{nu_2-m}).$$
Drop the final indicator in the sum defining $E_1$ to see that 
\begin{align*}
\frac{c}{n}E_1&\le \frac{c}{n}\sum_{m=\lfloor (nu_1-2n\delta)^+\rfloor}^{\lfloor nu_1\rfloor}\sum_{\beta_0\in I_m}\sum_{\beta_1\in\N^{\lfloor nu_1\rfloor-m}}\sum_{\beta_2\in\N^{\lfloor nu_2\rfloor-m}}\P(\beta_0\in \mc{T}_m)\P(0\vee\beta_1\in\mc{T}_{nu_1-m})\\
&\phantom{\le \frac{c}{n}\sum_{m=\lfloor (nu_1-2n\delta)^+\rfloor}^{\lfloor nu_1\rfloor}\sum_{\beta_0\in I_m}\sum_{\beta_1\in\N^{\lfloor nu_1\rfloor-m}}\sum_{\beta_2\in\N^{\lfloor nu_2\rfloor-m}}}\times\P(0\vee\beta_1\in\mc{T}_{nu_1-m})\\
&=\frac{c}{n}\sum_{m=\lfloor (nu_1-2n\delta)^+\rfloor}^{\lfloor nu_1\rfloor} \E[|\mc{T}_m|]\E[|\mc{T}_{\lfloor nu_1\rfloor-m}|]\E[|\mc{T}_{\lfloor nu_2\rfloor-m}|].
\end{align*}
As the critical GW process has constant mean one this shows that 
\begin{equation}\label{E1bound}\frac{c}{n}E_1\le \frac{c}{n}(2\delta n+1)=c(2\delta+\frac{1}{n}).
\end{equation}

Let $\alpha_i\in I_{nu_i}$ for $i=1,2$ as in the sum defining $E_2$ and let $m=|\alpha_1\wedge \alpha_2|<\lfloor (nu_1-2n\delta)^+\rfloor\}$.  Define $S_k=w_{m+k}(nu_1,\alpha_1)-w_{m+k}(nu_2,\alpha_2), k=0,\dots,\lfloor nu_1\rfloor-m$.  Then conditionally on $\mc{T}$ $S_k$ is a simple $d$-dimensional random walk with step distribution that of $X=X_1-X_2$, where $X_i$ are iid with law $D$.  Therefore 
\begin{align*}\P&(\tau(w(nu_1,\alpha_1),w(nu_2,\alpha_2))-|\alpha_1\wedge\alpha_2|> \lceil n\delta\rceil\}|\mc{\T})\\
&=\P(S_k=0\text{ for }k=1,\dots,\lceil n\delta\rceil)\le e^{-c'n\delta},
\end{align*}
for some $c'>0$. Use this to see that 
\begin{equation}\label{E2bound}
\frac{c}{n}E_2\le \frac{c}{n}\E[|\mc{T}_{nu_1}||\mc{T}_{nu_2}|]e^{-c'n\delta}\le c\Bigl[\frac{1}{n}+\gamma u_2\Bigr]e^{-c'n\delta}\le c(1+\gamma (u_1\vee u_2))e^{-c'n\delta}.
\end{equation}
In the above we use the classical result that $\E[|\mc{T}_m|^2]=1+\gamma m$ (see e.g. Section 2 of \cite{Ha48}).
Use the bounds on $E_1$ and $E_2$ in \eqref{E1bound} and \eqref{E2bound}  in \eqref{Edecomp}, integrate out $u_1,u_2\in[0,t_0]$ on the left-hand side of \eqref{Edecomp}, and in the resulting inequality  
let $n\to\infty$, and then $\delta\to 0$ to prove Condition~\ref{cond:tau_stuff}(i). 

The proof of Condition \ref{cond:tau_stuff}(ii) proceeds by a similar argument.  Now $\delta>0$ is fixed and we will let $\vep>0$ and $n$ vary.  Let $\hat{S}_k=w_{\tau+k}(nu_1,\alpha_1)-w_{\tau+k}(nu_2,\alpha_2), k=0,\dots,\lfloor nu_1\rfloor-\tau$, where $\tau=\tau(w(nu_1,\alpha_1),w(nu_2,\alpha_2))$.  Since the $\Delta_\alpha$ variables in \eqref{Deltasum} are independent of each other, and $\tau$ is a stopping time, $(\hat{S}_k)$ is a simple random walk with identical step distribution to $(S_k)$.

Bounding the expectation on the left-hand side of \eqref{term_a} and using very similar reasoning as in the proof of (i), one can reduce \eqref{term_a} to showing 
\[\lim_{\vep\to 0}\limsup_{n\to\infty}P(|\hat{S}_{\lfloor n\delta\rfloor}|\le \vep\sqrt n)=0.\]
For each $\delta$ this holds by the central limit theorem.
 \qed

\medskip
\begin{EXA}[Branching random walk]\label{BRWrem}  Recall the branching random walk notation from the above, Section~\ref{sec:BRW} and 
 Remark~\ref{rem:alphahat}. 
To make the latter a bit more precise we define a random probability $\hat{J}^{\sn}$ on $I\times\oD$ (which projects down to $J^{\sn}$ by \eqref{Jnunif}) by
\[\hat{J}^{\sn}(V\times A\times B)=\frac{1}{|\T|}\sum_{t\in\Z_+/n}\ \sum_{\alpha\in\T_{nt}}\1_V(\alpha)\1_A(t)\1_B(w^{\sn}(t,\alpha)).\]  
Then under $\hat{\P}^s_n=\hat{\P}^s_n(\hat{J}^{\sn})$, given $\T$, $(\hat\alpha_i,\mathfrak{T}^{\sn}_i, W^{\sn}_i)_{i\in\N}$ are iid with law $\hat{J}^{\sn}$. Therefore given $\T$, $\hat\alpha_i$ is uniformly chosen from $\T$, $\mathfrak{T}^{\sn}_i=|\hat\alpha_i|/n$, and $W^{\sn}_i$ is given by \eqref{WBRW}.
 Let 
$\hat\tau^{\sn}_{i,j}=|\hat{\alpha}_i\wedge \hat{\alpha}_j|/n$, which is the rescaled generation of the most recent common ancestor of the randomly selected multi-indices $\hat{\alpha}_i$ and $\hat{\alpha}_j$. (One can easily check that $\hat{\bs{\tau}}\in R_K$ by constructing an appropriate $w\in C_K$.)   
If $\tau^{\sn}_{i,j}=\tau(\kappa_n(W_i^{\sn}),\kappa_n(W_j^{\sn}))$, then the fact that common ancestral lineages imply common spatial trajectories (recall \eqref{WBRW}) implies 
\begin{equation}\label{tauineqs}\hat\tau^{\sn}_{i,j}\le\tau^{\sn}_{i,j}\text{ for all }i,j,n\in\N,
\end{equation} 
but strict inequality is possible if the spatial trajectories coincide for a few steps after $\hat{\bs{\tau}}^{\sn}_{i,j}$.  Asymptotically they are the same:
\begin{equation}\label{sametaus}\text{for every $i,j \in \N$, and $s,\vep>0$,  $\hat{\P}^s_n(|\tau^{\sn}_{i,j}-\hat\tau^{\sn}_{i,j}| >\vep)\to 0$ as $n \to \infty$.}
\end{equation}
Here one uses the bound on $E_2$ in the final step of the previous proof, which gives this with $\tau(W^{\sn}_i,W^{\sn}_j)$ in place of $\tau^{\sn}_{i,j}$, and then applies \eqref{kapconv} and \eqref{tauconvp}.
To study the true ancestral tree, one may prefer to work with $(T(\hat{\bs{\tau}}^{\sn}),d_{\hat{\bs{\tau}}^{\sn}})$ in
place of $(T(\bs{\tau}^{\sn}),d_{\bs{\tau}^{\sn}})$. The inequality \eqref{tauineqs} shows that $\phi_{\kappa_n(W^{\sn})}([i,u])=\kappa_n(W^{\sn}_i)(u)$ is well-defined on $\overline{T}(\hat{\bs{\tau}}^{\sn})$ and it is again easy to verify continuity of $\phi_{\kappa_n(W^{\sn})}$. So we may define $\tilde{\mc{B}}_{n,K}=(T(\hat{\bs{\tau}}^{\sn}),d_{\hat{\bs{\tau}}^{\sn}},\phi_{\kappa_n(W^{\sn})})$, if it is non-degenerate, and $\degK$  otherwise. One can then use \eqref{sametaus} and an easy extension of Proposition~\ref{shapecont} to show that the weak convergence in  Theorem~\ref{thm:Gconvergence} holds with the last two components of the vector in \eqref{wconva} replaced by $(\hat{\bs{\tau}}^{\sn},\tilde{\mc{B}}_{n,K})$.  Now apply a simple scaling argument as in Theorem~\ref{thm:Grjointconvergence} to see that for any $s>0$ $K\in\N$, as $n\to\infty$,  $\hat{\P}^{sn}_1(\hat{\bs{\tau}}^{\sss(1)}\text{ non-degenerate})\to1$ and if $\hat{\mc{B}}_{n,K}=(T(\bs{\tau}^{\sss(1)}),d_{\hat{\bs{\tau}}^{\sss(1)}}/n,\phi_{\kappa_1(W^{\sss(1)})}/\sqrt n)$ then 
\begin{align}\label{bbrhatwconv}&\hat\P^{sn}_1\Big(\Big(\frac{I^{\sss(1)}\circ\bar\rho_n^{\ -1}}{n^2},J^{\sss(1)}\circ\bar\rho_n^{\ -1},\frac{S^{\sss(1)}}{n},\rho_n(W^{\sss(1)}),\frac{\hat{\bs{\tau}}^{\sss(1)}}{n},\hat{\mc{B}}_{n,K}\Bigr)\in\cdot\Bigr)\\
\nonumber&\cweak \hat\N^{s,\sigma_0^2}_{\sss H} \Big(\big(I,J,S,W,\bs{\tau}(W),\mc{B}_K(W)\big)\in\cdot\Big)\text{ in } \M_F(\D)^2\times\R_+\times\D^K\times\R_+^{(K^2)}\times \TT_{gst}. \Enddef
\end{align}

\end{EXA}

\section{Proof of Theorem \ref{thm:main2}(b),(c)}
\label{sec:S}
Define a random measure, $I^{\sn}_X$, and a random probability, $J^{\sn}_X$, on $\R_+\times\R^d$ by
\begin{align*}
I_X^{\sn}(A \times B)&=\int^\infty_0\int \1_A(u)\1_B(x)X_u^{\sn}(\d x)\,du=\int\1_{A\times B}(u,w_u)I^{\sn}(\d u,\d w),\\
J_X^{\sn}&=
I^{\sn}_X/I^{\sn}_X(\R_+\times\R^d), 
\end{align*}
the last equality since $I_X^{\sn}(\R_+\times\R^d)>0$ as $o\in\mc{T}$.  Recall that $(\Omega, \mc{F})$ is the underlying space on which our random tree is defined.  For $\omega\in\Omega$, let $P_\omega$ be the law on $(\R_+\times\R^d)^\N$ under which the coordinate variables $\big((\mathfrak{T}^{\sss(1)}_j,V^{\sss(1)}_j)\big)_{j\in\N}$ are i.i.d.~with law $J_X^{\sss(1)}(\omega)$, and let $\hat{\P}^{ns}_1=\hat{\P}^{ns}_1(J_X^{\sss(1)})$ be the probability on $(\hat\Omega,\hat\F)=(\Omega\times(\R_+\times\R^d)^\N,\F\times\text{Borel sets})$ given by 
\[\hat\P^{ns}_1(A\times B)=\int\1_A(\omega)P_\omega(B)\d\P^{ns}_1.\]
For $\hat\omega\in\hat\Omega$ we write $\hat\omega=(\omega,\omega')$.  We set $(\mf{T}^{\sn}_j,V^{\sn}_j)=(\frac{\mf{T}^{\sss(1)}_j}{n},\frac{V^{\sss(1)}_j}{\sqrt n})$ for $j,n\in\N$, so that by \eqref{Jints}, conditional on $\omega$, $\big((\mf{T}^{\sn}_j,V^{\sn}_j)\big)_{j\in\N}$ are i.i.d.~with common law $J_X^{\sn}(\omega)$ under $\hat{\P}_1^{ns}$. Recalling that $\mc{T}_t=\mc{T}_{\lfloor t\rfloor}$, one easily sees that each $V^{\sss(1)}_i$ is chosen uniformly from the vertices of $\mc{T}$, 
 and that 
\begin{equation}\label{UYrelship}
V^{\sss(1)}_j\in \mc{T}_{\lfloor \mf{T}^{\sss(1)}_j\rfloor}, \text{ so that }\lfloor \mf{T}^{\sss(1)}_j\rfloor=d_{\mc{T}}(o,V^{\sss(1)}_j).
\end{equation}

Let $\mc{T}^{\sn}=\mc{T}/\sqrt n$ and $\mc{T}^{\sn}_t=\mc{T}_{nt}/\sqrt n$ for $n\in\N$ and $t\ge 0$.  As in \cite{HP19} for $s_i\ge 0$, $x_i\in\Z^d$, $i=1,2$, write
\[(s_1,x_1)\ara(s_2,x_2)\text{ iff }x_i\in\mc{T}_{s_i},\ i=1,2,\text{ and $x_2$ is a descendant of $x_1$ in } \mc{T},\]
and
\[(s_1,x_1)\ara s_2\text{ iff }(s_1,x_1)\ara(s_2,x_2)\text{ for some }x_2\in\Z^d.\]
We also write $(s_1,x_1)\aran(s_2,x_2)$ and $(s_1,x_1)\aran s_2$ for $s_i\ge 0$ and $x_i\in(\Z/\sqrt n)^d$ when the above hold with $\mc{T}^{\sn}$ in place of $\mc{T}$.
For $i\in\Z_+$ and $m,n\in\N$ as in \cite[prior to Lemma 4.6]{HP19} we introduce (recall that $C_0$ is as in \eqref{def_x})
\begin{align*}\Omega^n_{i,m}=\Bigl\{\omega\in \Omega:&\exists y\in\mc{T}^{\sn}_{i 2^{-m}}\text{ s.t. }(i2^{-m},y)\aran(i+1)2^{-m}\\
&\text { and }\int_{(i+1)2^{-m}}^{(i+2)2^{-m}}|\{x:(i2^{-m},y)\aran(s,x)\}|\d s/(C_0n)\le 2^{-10m}\Bigr\}.
\end{align*}
We also slightly modify the notation in \cite{HP19} and set
\[\Omega_m^n=\cup_{i=0}^{2^{2m}-1}\Omega^n_{i,m}.\]
As $s>0$ is fixed throughout this section, we omit dependencies on $s$ in our constants below.  Define $\mu_n=C_0n\P$, as in \cite{HP19}.  A key ingredient of our proof is the following restatement of \cite[Lemma 4.6]{HP19}.
\begin{LEM}\label{lem:4.6} There is a $c_{\ref{lem:4.6}}$ and for any $m\in\N$ there is an $n_{\ref{lem:4.6}}\in\N$ such that 
\[\sup_{n\ge n_{\ref{lem:4.6}}}\mu_n(\Omega_m^n)\le c_{\ref{lem:4.6}}2^{-m}.\]
\end{LEM}
We note that (1.22) of \cite{HP19} with $m(t)=C_0(t \vee 1)$ implies that for some $0<\underline c\le \overline C<\infty$,
\begin{equation}\label{Sbounds} \frac{\underline c}{1\vee t}\le \P(S^{\sss (1)}>t)\le \P(S^{ \sss (1)}\ge t)\le \frac{\overline C}{1\vee t}\quad\forall t>0.
\end{equation}
\begin{COR}\label{cor:4.6} For some $c_{\ref{cor:4.6}}$ and all $m\in\N$, 
\[\sup_{n\ge n_{\ref{lem:4.6}}}\P^{ns}_1(\Omega_m^n)\le c_{\ref{cor:4.6}}2^{-m}.\]
\end{COR}
\begin{proof}
By Lemma~\ref{lem:4.6} and \eqref{Sbounds} for $n\ge n_{\ref{lem:4.6}}(m)$,
\begin{align*}
\P^{ns}_1(\Omega_m^n)\le \frac{\mu_n(\Omega_m^n)}{C_0n\P(S^{ \sss (1)}>ns)}\le \frac{c_{\ref{lem:4.6}}2^{-m}}{C_0\underline c/s}:=c_{\ref{cor:4.6}}2^{-m}.\end{align*}
\end{proof}
We next recall Theorem 1' from \cite{HP19} in the context of lattice trees (see Theorem 6 of \cite{HP19}) where $\mu_n$ is as above and we may set $\alpha=1/4$ in the above reference. 

\begin{THM}\label{thm:modulus} There is a $c_{\ref{thm:modulus}}$, and for all $n\in\N$, a random variable $\delta_n\in(0,1]$ such that 
\begin{equation}\label{deltabound}   \P(\delta_n\le \varrho)\le\frac{c_{\ref{thm:modulus}}}{n}\varrho\quad\forall\varrho\in[0,1),\end{equation}
and 
$s_1,s_2\in\Z_+/n,\ |s_1-s_2|\le \delta_n,\text{ and }(s_1,y_1)\aran(s_2,y_2)\text{ imply }
|y_1-y_2|\le c_{\ref{thm:modulus}}|s_1-s_2|^{1/4}$.
\end{THM}

Note that \eqref{Sbounds} and \eqref{deltabound} imply that for all  $\varrho\in[0,1)$ and $n\in\N$, 
\begin{equation}\label{deltabound2}\P_1^{ns}(\delta_n\le \varrho)\le \frac{c_{\ref{thm:modulus}}}{n\P(S^{\sss (1)}>ns)}\varrho\le c'_{\ref{thm:modulus}}\varrho.
\end{equation}

\begin{LEM}\label{ancnumbers} There is a $c_{\ref{ancnumbers}}$ so that if
\begin{equation}\label{mncond}n\ge 2^m\ge s^{-1},\ m,n\in\N,
\end{equation}
then 
\[\P^{ns}_1\Bigl(\max_{0\le i\le 2^{2m}}\sum_{y\in\mc{T}_{ni2^{-m}}}\indic{(ni2^{-m},y)\ara n(i+1)2^{-m}}>2^{4m}\Bigr)\le c_{\ref{ancnumbers}}2^{-m}.\]
\end{LEM}
\begin{proof} We let $\mc{T}_{\le t}$ be the subtree of $\mc{T}$ consisting of vertices in $\cup_{i\le t}\mc{T}_i$ and their connecting edges to the root. Then (9.11) of \cite{HP19} implies that for $m,n$ as in \eqref{mncond}, and on $\{y\in\mc{T}\}$, for all $i\in\Z_+$,
\begin{align*}
\P\big((ni2^{-m},y)\ara n(i+1)2^{-m}\big |\mc{T}_{\le ni2^{-m}}\big)&\le c (\lfloor n(i+1)2^{-m}\rfloor-\lfloor ni2^{-m}\rfloor)^{-1}\\
&\le c_12^mn^{-1},
\end{align*}
where the last inequality is by $n\ge 2^m$ and elementary arithmetic.  Therefore under \eqref{mncond}, for $i\in\Z_+$,
\begin{align}
\nonumber\E_{1}^{ns}&\Bigg[\sum_{y\in\mc{T}_{ni2^{-m}}}\1_{\{(ni2^{-m},y)\ara n(i+1)2^{-m}\}}\Bigg]\\
\nonumber&\le\E\Bigg[\sum_{y\in\mc{T}_{ni2^{-m}}}\P\big((ni2^{-m},y)\ara n(i+1)2^{-m}\big|\mc{T}_{ni2^{-m}}\big)\Bigg]\,/\,\P(S^{\sss(1)}>sn)\\
\nonumber&\le c_1 2^mn^{-1}\E\big[| \mc{T}_{ni2^{-m}}|\big]sn\quad\text{(by \eqref{Sbounds})}\\
\label{meanbnd}&\le C2^m,
\end{align}
where in the last line we used the uniform boundedness of $\E[|\mc{T}_i|]$ from \eqref{pop_size_converge}.  
By Markov's inequality we conclude that for $m,n$ as above,
\begin{align}
\nonumber\P_1^{ns}\Bigg(&\max_{0\le i\le 2^{2m}}\Bigl[\sum_{y\in \mc{T}_{ni2^{-m}}}\1_{\{(ni2^{-m},y)\ara n(i+1)2^{-m}\}}\Bigr]>2^{4m}\Bigg)\label{probbnd}\le C(2^{2m}+1)2^{m-4m}\le c_{\ref{ancnumbers}}2^{-m}.
\end{align}
\end{proof}
\begin{LEM} \label{lem:SBMcalc}$\N_{\sss H}\Bigl[X_1(1)\exp\Bigl(-\lambda\int_0^1 X_u(1)\,\d u\Big)\Bigr]=\frac{4e^{\sqrt{2\lambda}}}{(1+e^{\sqrt{2\lambda}})^2}\le 4e^{-\sqrt{2\lambda}}\quad\forall \lambda>0$.
\end{LEM}
\begin{proof} This is a standard duality exercise in superprocesses. 
 If $g(\lambda,\mu)=(\sqrt{2\lambda}+\mu)/ (\sqrt{2\lambda}-\mu)$, for $0\le\mu<\sqrt{2\lambda}$, then 
 \[v_t=v_t^{\lambda,\mu}=\sqrt{2\lambda}\cdot \frac{(g(\lambda,\mu)e^{\sqrt{2\lambda} t}-1)}{(g(\lambda,\mu)e^{\sqrt{2\lambda} t}+1)}\]
is the unique non-negative solution of 
\[\frac{dv}{dt}=\frac{-v^2}{2}+\lambda,\ \ t\ge 0,\quad v_0=\mu.\]
Exponential duality (see Theorem II.5.11(c) and Theorem II.7.3 (c) of \cite{Per02}) 
gives
\[\N_{\sss H}\Bigl[1-\exp\Bigl(-\mu X_1(1)-\lambda\int_0^1 X_u(1)\d u\Bigr)\Bigr]=v_1^{\lambda,\mu}.\]
Now take the right derivative of the above with respect to $\mu$ at $\mu=0$ to derive the required equality.
\end{proof}

\begin{COR} \label{smallintbnd}For all $m\in\N$ satisfying $2^m\ge 2/s$, 
\begin{equation}\label{limsupint} \limsup_{n\to\infty}\P_1^{ns}\Bigl(\int_0^{2^{-m}}X_u^{\sn}(1)\,\d u\le 2^{-10m}\Bigr)\le 8e \cdot e^{-2^{4m}\sqrt 2}.
\end{equation}
In particular for $m$ as above there is an $n_{\ref{smallintbnd}}(m)\in\N$ so that 
\begin{equation}\label{smallintbndeq}
\sup_{n\ge n_{\ref{smallintbnd}}(m)}\P_1^{ns}\Bigl(\int_0^{2^{-m}}X_u^{\sn}(1)\,\d u\le 2^{-10m}\Bigr)\le 24 \cdot 2^{-4m}.
\end{equation}
\end{COR}
\begin{proof} The last assertion is a trivial consequence of the first.  

The weak convergence of
$\P_n^{s}\Bigl(\int_0^{2^{-m}}X_u^{\sn}(1)\,\d u\in\cdot\Bigr)$ to $\N_{\sss H}^s\Bigl(\int_0^{2^{-m}}X_u(1)\,\d u\in\cdot\Bigr)$ (by Lemma~\ref{lem:IJStconv} or by Condition~6 in \cite{HP19} for the more general setting considered there) and $\P^{sn}_1=\P^s_n$ imply the limsup in \eqref{limsupint} is at most
\begin{align}
\N_{\sss H}^s&\Bigl(\int_0^{2^{-m}}X_u(1)\,\d u\le 2^{-10m}\Bigr)\nonumber\\
&=\N_{\sss H}\Bigl(\int_0^{2^{-m}}X_u(1)\,\d u\le 2^{-10m}, X_s(1)>0\Bigr)/\N_{\sss H}(S>s)\nonumber\\
&=2^m\cdot \N_{\sss H}\Bigl(\int_0^{1}X_u(1)\,\d u\le 2^{-8m}, X_{2^ms}(1)>0\Bigr)\cdot\frac{s}{2},\label{dogeatdog}
\end{align}
where we have used \eqref{SBMsurv} and scaling (use  \cite[Exercise II.5.5]{Per02} and the fact that \break $\N_{\sss H}(X(1)\in\cdot)=\lim_{n \to \infty}n\P_{\delta_0/n}(X(1)\in\cdot)$ therein).  
Applying the Markov property and \cite[(II.5.12)]{Per02} we see that \eqref{dogeatdog} is equal to
\begin{align*}
&\N_{\sss H}\Bigg[\indic{\int_0^{1}X_u(1)\,\d u\le 2^{-8m}}\Bigl[1-\exp\Bigl(\frac{-2X_1(1)}{2^ms-1}\Bigr)\Bigr]\Bigg]\cdot \frac{2^m s}{2}\\
&\le e\cdot \N_{\sss H}\Bigg[\exp\Bigl(-2^{8m}\int_0^1X_u(1)\,\d u\Bigr)\,X_1(1)\Bigg]\cdot \Bigl(\frac{2^ms}{2^ms-1}\Bigr)\\
&\le 8e \cdot e^{-\sqrt{2}\,2^{4m}},
\end{align*}
the last by Lemma~\ref{lem:SBMcalc} and $2^ms/(2^{m}s-1)\le 2$. The result follows.
\end{proof}

\noindent{\it Proof of Theorem~\ref{thm:main2}(b),(c).} Fix $\delta\in(0,1)$.  By \eqref{Sbounds} for $M>s$, 
\[\P_1^{ns}(S^{\sn}>M)=\frac{\P(S^{\sss(1)}>Mn)}{\P(S^{\sss(1)}>sn)}\le c\frac{s}{M}\quad\text{for all }n\in\N.\]
Therefore we may choose $M_1=M_1(\delta,s)\in\N$ so that 
\begin{equation}\label{S1bnd}\sup_{n\in\N}\P_1^{ns}(S^{\sn}>M_1)<\frac{\delta}{100}.
\end{equation}
By Corollary~\ref{cor:IJScvgce} (for the more general setting of \cite{HP19} one could apply Lemma 4.5 of that reference),
\[\P_1^{ns}\Bigl(\int_0^\infty X^{\sn}_u(1)\,\d u\in\cdot\Bigr)\cweak \N_{\sss H}^s\Bigl(\int_0^\infty X_u(1)\,\d u\in\cdot\Bigr)\quad\text{ as }n\to\infty.\]
Therefore by tightness there is an $M_2(\delta,s)\in\N$ so that
\begin{equation}\label{totmassbnd}\sup_{n\in\N}\P_1^{ns}\Bigl(\int_0^\infty X_u^{\sn}(1)\,\d u>M_2\Bigr)<\delta/100.
\end{equation}
Next choose $m=m(\delta,s)\in\N$ large enough so that $2^{1-m}<\delta/5$ and 
\begin{equation}\label{mcond} 2^m-1>M_1\vee M_2\vee \frac{2}{s}, \text{ and }c_{\ref{cor:4.6}}2^{-m}+c'_{\ref{thm:modulus}}2^{4-m}+c_{\ref{ancnumbers}}2^{-m}+24 \cdot 2^{-4m}<\frac{\delta}{50}. 
\end{equation} 
Introduce a measurable set $\tilde \Omega^{\sn}\subset\Omega$, such that on $\tilde \Omega^{\sn}$ we have
\begin{align}\label{omegaconds}
&(a) \ \forall i\le 2^{2m}-1\quad\forall y\in\mc{T}^{\sn}_{i2^{-m}}\quad (i2^{-m},y)\aran (i+1)2^{-m}\quad\text{implies}\\
\nonumber &\qquad\frac{1}{C_0 n}\int_{(i+1)2^{-m}}^{(i+2)2^{-m}}\ |\{x:(i2^{-m},y)\aran (u,x)\}|\,\d u 
>2^{-10m},\\
\nonumber&(b) \ \text{If $\delta_n$ is as in Theorem~\ref{thm:modulus}, then }\delta_n>2^{4-m},\\
\nonumber&(c)\ \max_{0\le i\le 2^{2m}}\sum_{y\in\mc{T}_{ni2^{-m}}}\indic{(ni2^{-m},y)\ara n(i+1)2^{-m}}\le 2^{4m},\\
\nonumber&(d)\ \int_0^{2^{-m}}X^{\sn}_u(1)\,\d u>2^{-10m},\\
\nonumber&(e)\ S^{\sn}\le M_1,\ \ \int_0^\infty X_u^{\sn}(1)\,\d u\le M_2.
\end{align}
If we combine Corollary~\ref{cor:4.6}, 
\eqref{deltabound2}, Lemma \ref{ancnumbers} and Corollary~\ref{smallintbnd} with \eqref{S1bnd} and \eqref{totmassbnd}, we see that we may choose a natural number $n_1(\delta)> 2^m$ (recall $m$ depends on $\delta$) so that by the choice of $m$,
\begin{align}
\inf_{n\ge n_1(\delta)}\P_1^{ns}(\tilde\Omega^{\sn})&\ge 1-c_{\ref{cor:4.6}}2^{-m}-c'_{\ref{thm:modulus}}2^{4-m}-c_{\ref{ancnumbers}}2^{-m}-24 \cdot 2^{-4m}-\frac{\delta}{50}
\label{tomegalb}&\ge 1-\frac{\delta}{25}.
\end{align}
 
For $i\in\Z_+$, $n\in\N$ and $y\in(\Z/\sqrt n)^d$, introduce
\begin{align*}A_{i,y}^n&=\{(u,x)\in[(i+1)2^{-m},(i+2)2^{-m}]\times\R^d: (i2^{-m},y)\aran (u,x)\},\\
A_0^n&=\{(u,x)\in[0,2^{-m}]\times\R^d:x\in\mc{T}^{\sn}_u\}.
\end{align*}
By \eqref{omegaconds}(a),(e), and the definition of $X^{\sn}$ in \eqref{def_x}, $\omega\in\tilde\Omega^{\sn}$ implies
\begin{equation}\label{muchmass}
J_X^{\sn}(A^n_{i,y})\ge\frac{2^{-10m}}{M_2}\quad\forall 0\le i\le 2^{2m}-1\ \forall y\in\mc{T}^{\sn}_{i 2^{-m}}\ \text{ such that }\ (y,i 2^{-m})\aran (i+1)2^{-m},
\end{equation}
and by \eqref{omegaconds}(d),(e), $\omega\in\tilde\Omega^{\sn}$ implies
\begin{equation}\label{lefttailmass} J_X^{\sn}(A_0^n)=\frac{\int_0^{2^{-m}}X^{\sn}_u(1)\,\d u}{\int_0^\infty X_u^{\sn}(1)\,\d u}\ge\frac{2^{-10 m}}{M_2}.
\end{equation}
It follows from \eqref{muchmass} and \eqref{omegaconds}(c) that for $\omega\in\tilde\Omega^{\sn}$ and $K\in\N$, 
\begin{align*}P_\omega&\Big(\exists i\le 2^{2m}-1 \text{ and }y\in\mc{T}^{\sn}_{i2^{-m}}\text{ satisfying }(i2^{-m},y)\aran(i+1)2^{-m}\\
&\qquad\text{ so that } (\mathfrak{T}^{\sn}_j,V^{\sn}_j)\notin A^n_{i,y}\text{ for all }j\le K\Big)\\
&\le 2^{2m}\times 2^{4m}\times\Bigl(1-\frac{2^{-10m}}{M_2}\Bigr)^K\\ 
&\le\exp\Bigl(\frac{-K}{M_2}2^{-10m}+6m\log 2\Bigr).
\end{align*}
Therefore we may choose $K\ge K_1(\delta)$ (recall $m=m(\delta)$ is fixed) large enough so that the above probability is at most $\delta/10$.  A similar (but simpler) calculation using \eqref{lefttailmass} shows that for $\omega\in\tilde\Omega^{\sn}$ and $K\ge K_2(\delta)$,
\[P_\omega\big((\mathfrak{T}^{\sn}_j,V^{\sn}_j)\notin A_0^n\text{ for all }j\le K\big)\le \delta/10.\]
We have shown that if $\omega\in\tilde\Omega^{\sn}$ and $K\ge K(\delta)$, then
\begin{align}
\nonumber P_\omega\Big(&\forall0\le i\le 2^{2m}-1\ \forall y\in\mc{T}^{\sn}_{i2^{-m}}\text{ such that }(i2^{-m},y)\aran(i+1)2^{-m}\\
\nonumber&\exists j\le K\text{ such that }(\mathfrak{T}^{\sn}_j,V^{\sn}_j)\in A^n_{i,y}\text{ and }\exists j\le K\text{ such that }(\mathfrak{T}^{\sn}_j,V^{\sn}_j)\in A_0^n\Big)\\
\nonumber&\equiv{P}_\omega\big(\Lambda^{n,K}(\omega)\big)\\
\label{pomegalb}&\ge 1-\frac{\delta}{5}.
\end{align}
If $\hat\Omega^{n,K}=\{(\omega,\omega')\in\hat\Omega:\omega\in\tilde\Omega^{\sn}\text{ and }\omega'\in\Lambda^{n,K}(\omega)\}$, then by \eqref{tomegalb} and \eqref{pomegalb} for $K\ge K(\delta)$, 
\begin{equation}\label{hatlb1}\inf_{n\ge n_1(\delta)}\hat\P^{ns}_1\big(\hat\Omega^{n,K}\big)\ge\Bigl(1-\frac{\delta}{5}\Bigr)\Bigl(1-\frac{\delta}{25}\Bigr)\ge 1-\frac{\delta}{4}.
\end{equation}
We are now ready to prove (b) of the Theorem.  Let $n\ge n_1(\delta)$. We now show that on $\hat\Omega^{n,K}$, $\max_{x\in\mc{T}}d_{\mc{T}}(x,\pi_K(x))/n$ is small, where $\pi_K(x)$ is the closest point in $\mc{T}^{\vec V}$ to $x$ and $\vec V=(V^{\sss(1)}_1,\dots, V^{\sss(1)}_K)$.   
Fix $\hat\omega=(\omega,\omega')\in\hat\Omega^{n,K}$, and let $x\in\mc{T}_\ell$ for some $\ell\in\Z_+$.

\noindent {\bf Case 1.} $\ell\ge n2^{-m}$.\\
We may then choose $i\in\Z_+$ so that 
\[\frac{i+1}{2^m}\le\frac{\ell}{n}<\frac{i+2}{2^m}.\]
Note that \eqref{omegaconds}(e) and \eqref{mcond} together with the fact that $\mc{T}_\ell\ni x$ imply
\[\frac{i+1}{2^m}\le \frac{\ell}{n}\le S^{\sn}\le M_1\le 2^m-1,\]
so that 
\begin{equation}\label{ibound}i\le 2^{2m}-2^m-1<2^{2m}-1.
\end{equation}
Our choice of $i$ ensures that $ni2^{-m}<\ell$, and so we may let $y$ denote the ancestor of $x$ in $\mc{T}_{ni2^{-m}}$. The facts that $x\in\mc{T}_\ell$ and $\ell\ge(i+1)n2^{-m}$ show that 
\[(ni2^{-m},y)\ara(i+1)n2^{-m}.\]
It follows from this, \eqref{ibound} and $\omega'\in\Lambda^{n,K}(\omega)$ that
there is a $j\le K$ such that $\Bigl(\frac{\mathfrak{T}^{\sss(1)}_j}{n},\frac{V^{\sss(1)}_j}{\sqrt n}\Bigr)\in A^n_{i,y/\sqrt n}$, which in turn implies
(recall \eqref{UYrelship})
\begin{equation}\label{UYprops} V_j^{\sss(1)}\in\mc{T}_{\lfloor \mathfrak{T}_j^{\sss(1)}\rfloor}\text{ where }\mathfrak{T}_j^{\sss(1)}\in[(i+1)2^{-m}n,(i+2)2^{-m}n]\text{ and }(i2^{-m}n,y)\ara(\lfloor \mathfrak{T}_j^{\sss(1)}\rfloor,V_j^{\sss(1)}).
\end{equation}
Our choice of $y$ ensures that $(ni2^{-m},y)\ara(\ell,x)$, so that 
\begin{align}
\nonumber d_{\mc{T}}(x,V^{\sss(1)}_j)/n&\le [d_{\mc{T}}(x,y)+d_{\mc{T}}(y,V^{\sss(1)}_j)]/n\\
\nonumber&\le [(\ell-ni 2^{-m}+1)+(n(i+2)2^{-m}-ni2^{-m}+1)]/n\\
\nonumber&\le [2(n(i+2)2^{-m}-ni2^{-m})+2]/n\\
\nonumber&=4\cdot2^{-m}+2n^{-1}\\
\label{Yclose1}&< 6\cdot2^{-m}<\delta,
\end{align}
where we used $n\ge n_1( \delta)> 2^m$ and $10\cdot 2^{-m}<\delta$ (recall \eqref{mcond} and the line preceding it) in the last line.  As $V^{\sss(1)}_j$ is in $\mc{T}^{\vec V}$ (recall $j\le K$), this shows that 
\begin{equation}\label{acase1}
d_{\mc{T}}(x,\pi_K(x))/n<6\cdot2^{-m}<\delta.
\end{equation}

\noindent {\bf Case 2.} $0\le \ell<n2^{-m}$.\\
In this case $\omega'\in \Lambda^{n,K}(\omega)$ implies there is a $j\le K$ such that $\Bigl(\frac{\mathfrak{T}^{\sss(1)}_j}{n},\frac{V^{\sss(1)}_j}{\sqrt n}\Bigr)\in A^n_0$,  that is, $V^{\sss(1)}_j\in\mc{T}_{\lfloor \mathfrak{T}_j^{\sss(1)}\rfloor}$, where $\lfloor \mathfrak{T}_j^{\sss(1)}\rfloor\le n2^{-m}$.  Therefore
\begin{align}
\nonumber d_{\mc{T}}(x,\pi_K(x))/n\le d_{\mc{T}}(x,V^{\sss(1)}_j)/n&\le (d_{\mc{T}}(x,o)+d_{\mc{T}}(o,V^{\sss(1)}_j))/n\\
\label{acase2}&\le (\ell+n 2^{-m})/n<2^{1-m}<\delta.
\end{align}

Therefore we have shown by \eqref{acase1} and \eqref{acase2}, that
\begin{equation} \label{partabound} n\ge n_1(\delta),\ K\ge K(\delta)\text{ and }\hat\omega\in\hat\Omega^{n,K}\text{ imply }d_{\mc{T}}(x,\pi_K(x))/n< 6\cdot2^{-m}<\delta.
\end{equation}
Thus from \eqref{hatlb1}, for every $\delta,s>0$ there exist $n_1(\delta,s), K(\delta,s)\in \N$ such that for $K\ge K(\delta,s)$,
\begin{equation}
\inf_{n \ge n_1(\delta)}\hat \P_1^{ns}\big(d_{\mc{T}}(x,\pi_K(x))/n<\delta \big)\ge 
\inf_{n\ge n_1(\delta)}\hat\P^{ns}_1\big(\hat\Omega^{n,K}\big)\ge 1-\frac{\delta}{4}.
\end{equation}
This completes the proof of (b).

Turn next to the Euclidean distance, and work with $n,m,K,\delta$, $x\in\mc{T}_\ell$, and $\hat\omega$ as in \eqref{partabound}.   Let $\pi_K(x)\in \mc{T}_{j'}$ and recall from \eqref{piKanc} that $\pi_K(x)$ is an ancestor of $x$ in $\mc{T}$. Therefore $n^{-1}|\ell-j'|=d_{\mc{T}}(x,\pi_K(x))/n<6 \cdot 2^{-m}<\delta_n$ (the latter by \eqref{omegaconds}(b)), and  
we may use Theorem~\ref{thm:modulus} and \eqref{partabound} to conclude that
\begin{align}
\label{partbbound}\frac{|\pi_K(x)-x|}{\sqrt n}&\le  c_{\ref{thm:modulus}}\Bigl|\frac{\ell-j'}{ n}\Bigr|^{1/4}
\le c_{\ref{thm:modulus}}(6\cdot 2^{-m})^{1/4}<c_{\ref{thm:modulus}}\delta^{1/4}.
\end{align}

In view of \eqref{partabound}, \eqref{partbbound}, and \eqref{hatlb1}, we conclude that for $K\ge K(\delta,s)$, 
\begin{align}\label{largen}
&\inf_{n\ge n_1(\delta)}\hat \P^{ns}_1\Bigl(\sup_{x\in\mc{T}}d_{\mc{T}}(x,\pi_K(x))/n<\delta\text{ and }\sup_{x\in\mc{T}}|\pi_K(x)-x|/\sqrt n<c_{\ref{thm:modulus}}\delta^{1/4}\Bigr)\\
\nonumber&\ge \inf_{n\ge n_1(\delta)}\hat\P^{ns}_1\big(\hat\Omega^{n,K}\big)\ge 1-\frac{\delta}{4}.
\end{align}

Now consider $n<n_1(\delta)$.  Since $|\mc{T}|<\infty$ $\P$-a.s., we may choose a natural number $M_3=M_3(\delta,s)$ large enough so that for all $n\le n_1(\delta)$, 
\begin{equation}\label{sizebnd}\P_1^{ns}\big(|\mc{T}|>M_3\big)\le \P\big(|\mc{T}|>M_3\big)/\P(S^{\sss(1)}>n_1 s)<\delta/4.
\end{equation}
If $|\mc{T}(\omega)|\le M_3$, then for any $x\in\mc{T}$ and $j\in\N$,
\[P_{\omega}\big(V^{\sss(1)}_j=x\big)=\frac{1}{|\mc{T}(\omega)|}\ge \frac{1}{M_3},\]
and so for $K\ge K'(\delta,s)$, and $\omega$ as above,
\[P_\omega\big(\forall x\in\mc{T}\ \exists j\le K\text{ so that }V^{\sss(1)}_j=x\big)\ge 1-M_3\Bigl[1-\frac{1}{M_3}\Bigr]^K\ge 1-\frac{\delta}{4},\]
the last by the choice of $K'(\delta,s)$.  It follows from \eqref{sizebnd} and the above that for $K\ge K'(\delta,s)$,
\begin{equation}\label{smalln}
\sup_{n< n_1(\delta)}\hat\P_1^{ns}\Bigl(\sup_{x\in\mc{T}}|x-\pi_K(x)|>0\Bigr)\le 
\frac{\delta}{2}.
\end{equation}
Part (c) of Theorem~\ref{thm:main2} now is immediate from \eqref{largen} and \eqref{smalln}.\qed

\paragraph{Acknowledgments.}
The work of MH was supported in part by Future Fellowship FT160100166 from the Australian Research Council. 
The work of EP was supported by a Discovery Grant from the Natural Science and Engineering Research Council of Canada.  MH thanks the Pacific Institute for the Mathematical Sciences at UBC for hosting him during part of this research.  

\bibliographystyle{plain}
\def\cprime{$'$}

\section*{Appendix. Proof of Theorem \ref{thm:main2}(a)}\label{sec:oldS}

Here we make use of the weak convergence of the range obtained in \cite{HP19} to verify Theorem \ref{thm:main2}(a) holds for lattice trees in dimensions $d>8$ for all $L\ge L_0(d)$.  In this section we do not  use the historical processes.     
In fact we show that Theorem \ref{thm:main2}(a) holds for the general lattice models considered in \cite{HP19} (see Theorem~\ref{thm:1(a)}), and so, in particular, also holds for critical sufficiently spread-out oriented percolation in dimensions $d>4$.

It is trivial to extend the construction of $(X^{\sn},(V^{\sn}_i)_{i\in\N})$ under $\hat\P^s_n$ from Section~\ref{sec:enlargement} to the general lattice models in \cite{HP19} in the discrete time setting{(i.e.~for the index set $I=\Z_+$ therein) and the same formulae apply.    
Let $R^{\sss(1)}=\cup_{m\in \Z_+}\mc{T}_m$ denote the range (set of vertices) of the random sets $\mc{T}_m,m\in\Z_+$, and $R^{\sn}=n^{-1/2}R^{\sss(1)}$ denote the rescaled range.  For $\nu=(\nu_t)_{t\ge 0}\in \mc{D}(\mc{M}_F(\R^d))$, define the integrated measure $\bar{\nu}_\infty(\cdot)=\indic{S(\nu)<\infty}\int_0^\infty \nu_t(\cdot)\d t$, where $S(\nu)=\inf\{t>0:\nu_t(1)=0\}$.  
Recall that $X^{\sn}_t\in \mc{M}_F(\R^d)$.  Let $\K$ denote the space of compact subsets of  $\R^d$ equipped with the Hausdorff metric $d_0$ (see, e.g., (1.2) of \cite{HP19}). Theorem~2 of \cite{HP19} gives conditions on random subsets $\mc{T}_m,\ m\in\Z_+$, of $\Z^d$ (called Conditions~1--7 there) which include our rescaled lattice trees (sufficiently large range, $d>8$) as well as rescaled critical oriented percolation for sufficiently large range and $d>4$, under which one has weak convergence of the rescaled ranges $R^{\sn}$ to the range $R$ of super-Brownian motion on $\K$ (under the canonical measure, conditioned to survive to time $1$).   
The scaling in \cite{HP19} implies that if $r_n(x)=x/\sqrt n$ for $x\in\R^d$, then
\begin{equation}\label{genscale}
\bar{X}^{\sn}_\infty=c_n\bar{X}^{\sss(1)}_\infty\circ r_n^{-1}\text{ for some }c_n>0,\ S^{\sn}=S^{\sss(1)}/n, \text{ and therefore }\supp(\bar {X}_\infty)=R^{\sn}.
\end{equation}
 The reader interested only in lattice trees, should work in this setting only (where $c_n=n^{-2}$).  Let  $R=\textup{supp}(\bar{X}_\infty)$ be the range of the super-Brownian motion $X$. Note that the closed support map, $\textup{supp}:\mc{M}_F(\R^d)\to \K$, is not a continuous function (with the topology of weak convergence and the Hausdorff metric respectively).  
The following apparent extension of Theorem~2 of \cite{HP19} is in fact implicit in its proof. As usual $s>0$ is fixed. The extension (and its proof) holds equally well in the continuous time setting (when $I=[0,\infty)$ in the notation of \cite{HP19}).

\begin{PRP}\label{prop:jtconv}
Assume Conditions 1-7 of \cite{HP19} in the discrete time setting there, that is when $I=\Z_+$ in the notation of \cite{HP19}. Then as $n\to\infty$,
\[\P^s_n\big((R^{\sn},\bar{X}^{\sn}_\infty)\in\cdot\big)\cweak \N^s_{\sss H}\big((R,\bar{X}_\infty)\in\cdot\big)\text{ on }\K\times \mc{M}_F(\R^d).\]
\end{PRP}

\begin{proof} Lemma~4.5 of \cite{HP19} 
states that $\bar{X}^{\sn}_\infty\to \bar{X}_\infty$ weakly, where it is always understood that the law for the latter is the canonical measure conditioned on survival until time $s$.  As in the proof of Theorem~2 in \cite{HP19}, we may use Skorokhod's representation to assume that $\bar{X}^{\sn}_\infty\to \bar{X}_\infty$ a.s., (see \cite[(4.17)]{HP19}).  In the same proof it is shown that on this probability space $R^{\sn}\to R$ in probability as $n\to\infty$ (in the proof of \cite[Theorem 2]{HP19} see (4.23) and the prior display, as well as the final line of the proof).   It follows that $(R^{\sn},\bar{X}^{\sn}_\infty)\to (R,\bar{X}_\infty)$ in probability on $\K\times \mc{M}_F(\R^d)$, and the result follows.
\end{proof}

Under $\hat{\P}^s_n$, given $\bar{X}^{\sn}_\infty$ the random variables $(V^{\sn}_i)_{i\in \N}$ are i.i.d.~with law $\bar{X}^{\sn}_\infty/\bar{X}^{\sn}_\infty(1)$ (equivalently, they are i.i.d.~uniform on $\mc{T}^{\sn}$).  Given  $\bar{X}_\infty$, let $(V_i)_{i \in \N}$ be  i.i.d.~with law $\bar{X}_\infty/\bar{X}_\infty(1)$ under $\hat{\N}_{\sss H}^s$. 
For fixed $K\in \N$, let $V^{\sn}=(V^{\sn}_1,\dots, V^{\sn}_K)\in(\R^d)^K$ and  $V=(V_1,\dots,V_K)\in (\R^d)^K$.

\begin{LEM} \label{lem:wcrandompts} Fix $K\in \N$ and $s>0$.  Under the hypotheses of Proposition~\ref{prop:jtconv}, 
\[\hat\P_n^s\big((R^{\sn},\bar{X}^{\sn}_\infty,V^{\sn})\in\cdot\big)\cweak\hat\N^s_{\sss H}\big((R,\bar{X}_\infty,V)\in\cdot\big) \text{ on }\K\times\mc{M}_F(\R^d)\times (\R^d)^K\text{ as }n\to\infty.\]
\end{LEM}
\begin{proof} We apply Lemma~\ref{lem:absrandompts} with $Z^{\sn}=R^{\sn}$, $M^{\sn}=\bar{X}^{\sn}_{\infty}$, $Z=R$, and $M=\bar{X}_\infty$. The weak convergence hypothesis holds by Proposition~\ref{prop:jtconv}.  
The result follows. 
\end{proof}

We leave the proof of the following as a simple exercise for the reader. 
\begin{LEM}\label{lem:iiddense} Let $\mu$ be a non-zero, compactly supported measure on a metric space $M$ and let $(Z_i)_{i\in\N}$ be a sequence of i.i.d.~random vectors with law $\hat\mu=\mu/\mu(1)$ under a probability measure $\P^*$. Then for any $\vep>0$, $\lim_{K\to\infty} \P^*(\textup{supp}(\mu)\subset\cup_{i=1}^KB(Z_i,\vep))=1$.
\end{LEM}
\blank{
\begin{proof} We leave this as a simple exercise for the reader. [But here are the details just in case.]

Let $\{x_n\}\subset\text{supp}(\mu)$ be a countable dense subset of supp$(\mu)$.  The converse to Borel-Cantelli shows that $\hat\mu$-a.s. for every $n$, $x_n\in B(Z_i,\vep/2)$ for infinitely many $i$.  Fix such an $\omega$. Then for each $n$ we may choose $i_n(\omega)$ so that $x_n\in B(Z_{i_n},\vep/2)$, and so 
\[\text{supp}(\mu)\subset\cup_nB(x_n,\vep/2)\subset\cup_{i}B(Z_i,\vep).\]
By compactness of supp$(\mu)$ there is a $K(\omega)\in\N$ so that supp$(\mu)\subset\cup_{i=1}^KB(Z_i,\vep)$.  
We have shown $\P$-a.s. there is a $K\in\N$ so that supp$(\mu)\subset\cup_{i=1}^KB(Z_i,\vep)$. The result follows.
\end{proof}
}

We are now ready to prove Theorem \ref{thm:main2}(a), which we  restate under the more general conditions of \cite{HP19}. \
\begin{THM}\label{thm:1(a)}
  Assume Conditions 1-7 of \cite{HP19} in the discrete time setting there.  
Then for
any $\vep,\delta,s>0$ there is a $K=K(\vep,\delta,s)\in \N$ so that  \[\hat\P^{ns}_1\Big(R^{\sss(1)}\subset \cup_{i=1}^KB(V_i^{\sss(1)},\vep\sqrt{n})\Big)\ge 1-\delta \qquad \text{ for all }n\in \N.\]
\end{THM}
\begin{proof} Fix $\vep,\delta,s>0$.  By Lemma~\ref{lem:iiddense}, applied conditionally on $\bar{X}_\infty$, and the fact that $R$ is a.s.~compact, we have
\[\lim_{K\to\infty}\hat\N_{\sss H}^s\Big(R\subset\cup_{i=1}^K B(V_i,\vep/2)\,\big |\,\bar{X}_\infty\Big)=1,\qquad   \hat\N^s_{\sss H}-\text{a.s.}\]
It follows by monotone convergence that there is a natural number $K_0$ such that 
\begin{equation}\label{lboundR}\hat\N_{\sss H}^s\Big(R\subset\cup_{i=1}^{K_0}B(V_i,\vep/2)\Big)\ge 1-\frac{\delta}{2}.
\end{equation}
Lemma~\ref{lem:wcrandompts} and Skorokhod's theorem allow us to work on a probability space $(\Omega',\mc{F}',\P')$ on which 
\begin{equation}\label{asconverge}\big(R^{\sn},(V_i^{\sn})_{i\le K_0}\big)\to \big(R,(V_i)_{i\le K_0}\big),\qquad \P'-\text{a.s.~in }\K\times(\R^d)^{K_0}.
\end{equation}
If $R\subset\cup_{i=1}^{K_0}B(V_i,\vep/2)$, then \eqref{asconverge} and the definition of the Hausdorff metric imply that for large enough $n$, $R^{\sn}\subset R^{\vep/4}\subset\cup_{i=1}^{K_0}B(V^{\sn}_i,\vep)$ (recall that $A^\vep$ is the set of points within distance $\vep$ of $A$). This shows that 
\begin{align*}\liminf_{n\to\infty}\hat\P^s_n\Big(R^{\sn}\subset\cup_{i=1}^{K_0}B(V^{\sn}_i,\vep)\Big)&=\liminf_{n\to\infty}\P'\Big(R^{\sn}\subset\cup_{i=1}^{K_0}B(V^{\sn}_i,\vep)\Big)\\
&\ge \P'\Big(\liminf_{n\to\infty}\{R^{\sn}\subset\cup_{i=1}^{K_0}B(V^{\sn}_i,\vep)\}\Big)\\
&\ge\P'\Big(R\subset\cup_{i=1}^{K_0}B(V_i,\vep/2)\Big)\\
&=\hat\N_{\sss H}^s\Big(R\subset\cup_{i=1}^{K_0}B(V_i,\vep/2)\Big)\\
&\ge 1-\frac{\delta}{2},
\end{align*}
the last by \eqref{lboundR}.  It follows by the scaling in \eqref{genscale} that for some $n_0=n_0(\vep,\delta)\in\N$,
\begin{equation}\label{coverall}\text{for all }n\ge n_0,\ \  \hat{\P}^{ns}_1\Bigl(R^{\sss(1)}\subset\cup_{i=1}^{K_0}B(V_i^{\sss(1)},\vep\sqrt n)\Bigr)=\hat\P^s_n\Big(R^{\sn}\subset\cup_{i=1}^{K_0}B(V^{\sn}_i,\vep)\Big)\ge 1-\delta.\end{equation}
For each $n<n_0$, $R^{\sss(1)}$ is $\P^{ns}_1$-a.s.~a finite subset of $\Z^d$ and so $\lim_{K\to\infty}\hat\P^{ns}_1(R^{\sss(1)}\subset\cup_{i=1}^K B(V^{\sss(1)}_i,\vep))=1$. Therefore by increasing $K_0$, if necessary, we can find a $K$ so that the conclusion of \eqref{coverall} holds for all $n\in\N$.
\end{proof}

\end{document}